\documentclass[11pt,oneside,reqno]{amsart}
\usepackage{textcomp}
\usepackage[latin9]{inputenc}
\usepackage{mathrsfs}
\usepackage{bm}
\usepackage{amstext}
\usepackage{amsthm}
\usepackage{amssymb}
\usepackage{geometry}
\usepackage{setspace}
\usepackage[bookmarks=false,
 breaklinks=false,pdfborder={0 0 1},backref=section,colorlinks=false]
 {hyperref}

\makeatletter
\numberwithin{equation}{section}
\numberwithin{figure}{section}

\usepackage{amsthm}
\usepackage{mathrsfs}
\usepackage[noadjust]{cite}
\usepackage{enumitem}
	\setlist[itemize]{leftmargin=*}
	\setlist[enumerate]{leftmargin=*}
\usepackage{tikz}
\usetikzlibrary[patterns]
\usepackage{bbm}
\usepackage{stmaryrd}

\makeatother

\theoremstyle{plain}
\newtheorem{thm}{\protect\theoremname}[section]
\theoremstyle{remark}
\newtheorem{rem}[thm]{\protect\remarkname}
\theoremstyle{plain}
\newtheorem{assumption}[thm]{\protect\assumptionname}
\theoremstyle{remark}
\newtheorem{notation}[thm]{\protect\notationname}
\theoremstyle{plain}
\newtheorem{lem}[thm]{\protect\lemmaname}
\newtheorem{prop}[thm]{\protect\propositionname}
\newtheorem{cor}[thm]{\protect\corollaryname}
\theoremstyle{remark}
\newtheorem*{acknowledgement*}{\protect\acknowledgementname}
\providecommand{\acknowledgementname}{Acknowledgement}
\providecommand{\assumptionname}{Assumption}
\providecommand{\corollaryname}{Corollary}
\providecommand{\lemmaname}{Lemma}
\providecommand{\notationname}{Notation}
\providecommand{\propositionname}{Proposition}
\providecommand{\remarkname}{Remark}
\providecommand{\theoremname}{Theorem}

\begin{document}
\title[Metastability of Systems with Countably Many Metastable States]{Metastability of Interacting Stochastic Systems with Countably Many
Metastable States: Beyond Positive Recurrence}
\author{Seonwoo Kim and Jungkyoung Lee}
\address{S. Kim. Department of Mathematics, Yonsei University, South Korea.}
\email{seonwookim@yonsei.ac.kr}
\address{J. Lee. Department of Mathematics Education, Inha University.}
\email{jungkyoung@inha.ac.kr}
\begin{abstract}
Metastability is typically formulated for systems with finitely many
metastable states, most often in positively recurrent settings. In
this article, we extend the resolvent framework for metastability
to Markov processes with countably many metastable states, thereby
encompassing null-recurrent and transient dynamics. For a family of
processes on locally compact Polish spaces, we prove, under a mild
boundary regularity assumption, that the asymptotic flatness of microscopic
resolvent solutions, supplemented by two compactness conditions, is
equivalent to convergence in law of the projected trace processes
to a limiting Markov chain and to the negligibility of the time spent
outside the metastable sets. 

We apply this framework to two non-compact stochastic systems. First,
we study a condensing inclusion process on a countably infinite, uniformly
locally finite graph, in a setting where the process may be null recurrent
or transient. We prove that the condensate location converges to a
weighted random walk on the underlying graph, while the time spent
away from the fully condensed configurations is negligible. Second,
we study a small-noise one-dimensional Langevin diffusion with countably
many stable equilibria, without assuming ergodicity. In the Eyring--Kramers
time scale associated with the minimal energy barrier, we establish
local equilibration inside each well, convergence of the well-index
process to an explicit nearest-neighbor Markov chain on $\mathbb{Z}$,
and negligibility of inter-well excursions. Together, these results
broaden the scope of resolvent-based metastability theory beyond finite
metastable state spaces and positive recurrence.
\end{abstract}

\maketitle
\tableofcontents{}

\section{\label{sec1}Introduction and Main Results}

Metastability is a widespread phenomenon that occurs in various physical
systems with multiple stable states, and mathematical study of metastability
concerns rare transitions among such states. Most existing works describe
the metastable dynamics in terms of finitely many metastable states;
see, for instance, studies on Langevin dynamics \cite{BR16,BEGK,DGLLPN,LLS24,LS22a,LPM,Mic95},
ferromagnetic systems \cite{BAC96,BM02,KL26,KS25,LS16,NZ19,NS91},
and interacting particle systems \cite{AGL,BL12,GKR07,GRV13,KS21,LMS23,LMS25}.
Nevertheless, systems with infinitely many metastable states arise
naturally in several important models, particularly in non-compact
or infinite-volume settings. For example, in trap models for disordered
systems, the limiting dynamics can be described by a Markov process
on a countable collection of deep traps \cite{FM08,JLT}. Small-noise
diffusions in periodic environments provide another natural example,
where the limiting dynamics describe transitions among countably many
wells \cite{LM26,LS19}. Related phenomena also emerge in the thermodynamic
limits of interacting particle systems. For instance, in condensing
zero-range processes, the number of possible metastable locations
for the condensate diverges with the system size, and the rescaled
location converges to a macroscopic stochastic process on a continuum
state space \cite{AGL}.

The study of metastability with countably many metastable states is
therefore essential for understanding metastable dynamics in non-compact
metastable systems. In this article, we extend the resolvent approach
developed in \cite{LMS25} to metastable systems with countably many
stable states.

\subsection{Review on metastability}

\subsubsection{Three defining conditions for metastability}

To fix ideas, suppose that we have a family of Markov processes $\{X_{N}(t)\}_{t\ge0}$
on state spaces $\Omega_{N}$, where the parameter $N\ge1$ usually
indicates the number of particles or the inverse unit distance between
the particles, to be sent to infinity in the limit. Alternatively,
the processes may be indexed by temperature $\epsilon>0$ or its inverse
$\beta>0$, in which case a metastable scenario occurs when $\epsilon\to0$
or $\beta\to\infty$, i.e., in low temperatures.

Suppose that the dynamical system $X_{N}(t)$ has countable (or finite)
number of (potentially) metastable sets $\mathcal{E}^{j}_{N}$, $j\in S$,
where $S$ denotes the index set of such states. Define
\begin{equation}
\mathcal{E}_{N}:=\bigcup_{j\in S}\mathcal{E}^{j}_{N}\qquad\text{and}\qquad\Delta_{N}:=\Omega_{N}\setminus\mathcal{E}_{N}.\label{eq:EN-DeltaN}
\end{equation}
The remainder set $\Delta_{N}$ collects the intermediate states that
are not metastable. Then, we say that a (dynamical) metastable behavior
occurs in the system in a time scale $\theta_{N}\gg1$ if the following
three conditions $\mathfrak{M}$, $\mathfrak{D}$, and $\mathfrak{C}$
are satisfied, where, $a_{N}\gg b_{N}$ if $b_{N}/a_{N}\to0$ as $N\to\infty$.
Here we first describe them informally, and give the rigorous definitions
later in Section \ref{sec1.3}.
\begin{itemize}
\item ($\mathfrak{M}$, Mixing) The first condition $\mathfrak{M}$ indicates
that the process $X_{N}(t)$ achieves an asymptotic equilibrium inside
each set $\mathcal{E}^{j}_{N}$ before escaping the set. In other
words, the process mixes well inside $\mathcal{E}^{j}_{N}$ before
its exit. This is formulated in the sense that the exit time of the
set $\mathcal{E}^{j}_{N}$ is much larger than the mixing time of
the reflected process in $\mathcal{E}^{j}_{N}$.
\item ($\mathfrak{D}$, Delta-negligibility) If the time is accelerated
by $\theta_{N}$, the process spends negligible amount of time in
the remainder set $\Delta_{N}$. This property guarantees that the
dynamical behavior of the original process is well approximated by
the trace process on $\mathcal{E}_{N}=\Omega_{N}\setminus\Delta_{N}$
(to be defined rigorously in \eqref{eq:trace-def}), obtained by the
harmonic projection from $\Omega_{N}$ onto $\mathcal{E}_{N}$.
\item ($\mathfrak{C}$, Convergence) Denote by $\{X^{{\rm tr}}_{N}(t)\}_{t\ge0}$
the trace process on $\mathcal{E}_{N}$. For each $\eta\in\mathcal{E}^{j}_{N}$
define $\Phi_{N}(\eta):=j$, and consider the projected order process
$Y_{N}(t):=\Phi_{N}(X^{{\rm tr}}_{N}(t))$ on $S$. Then, the law
of the accelerated order process $Y_{N}(\theta_{N}t)$ on $S$ converges
weakly as $N\to\infty$ to the law of a certain Markov chain $\{{\bf y}(t)\}_{t\ge0}$
on $S$. In other words, the inter-valley transitions in the time
scale $\theta_{N}$ can be approximated by a simple Markov chain on
the index set $S$. Here, the approximating law should be Markovian
since, by condition $\mathfrak{M}$, the process reaches a near-equilibrium
before escaping $\mathcal{E}^{j}_{N}$, thus an inter-valley jump
from $\mathcal{E}^{j}_{N}$ to another $\mathcal{E}^{k}_{N}$ would
behave like a Markovian jump from $j$ to $k$.
\end{itemize}
\begin{rem}
The condition $\mathfrak{M}$ was first formulated in terms of the
existence of an attractor (cf. condition \textbf{(V1)} of \cite{BL10}),
which fails to hold in general. For instance, consider the symmetric
simple random walk on a discrete torus $\mathbb{T}^{d}_{N}:=(\mathbb{Z}/N\mathbb{Z})^{d}$
with $d\ge2$ and imagine that the system escapes $\mathbb{T}^{d}_{N}$
with a uniform rate $\alpha^{-1}_{N}\ll1$. The mixing time of the
reflected random walk on $\mathbb{T}^{d}_{N}$ is of order $N^{2}$,
but any point in the torus is visited in the time scale $\beta_{N}\gg N^{2}$
where
\[
\beta_{N}=\begin{cases}
N^{2}\log N & \text{if}\quad d=2,\\
N^{d} & \text{if}\quad d\ge3.
\end{cases}
\]
Thus if $N^{2}\ll\alpha_{N}\ll\beta_{N}$, the system reaches an asymptotic
equilibrium in $\mathbb{T}^{d}_{N}$ before exiting, whereas no point
in $\mathbb{T}^{d}_{N}$ serves as an attractor. Other formulations
of the mixing condition in more general senses can be found in, e.g.,
\cite[Section 6]{LMS25} or \cite[Section 3]{KL26}.
\end{rem}

\subsubsection{Resolvent approach to metastability}

The overall strategy to prove the three conditions $\mathfrak{M}$,
$\mathfrak{D}$, and $\mathfrak{C}$ goes as follows. First, condition
$\mathfrak{M}$ is fundamentally model-dependent and is usually proved
in a quite independent flavor compared to the other conditions.

Recently, it was proved \cite{LMS25} that the remaining two conditions
$\mathfrak{D}$ and $\mathfrak{C}$ can be verified at once if the
so-called resolvent condition $\mathfrak{R}$ holds true. We defer
its exact formulation to Section \ref{sec1.3} and focus here on explaining
the underlying principle by formulating condition $\mathfrak{C}$
on a heuristic level. First, let us assume that $S$ is a finite set.
To prove condition $\mathfrak{C}$, according to the martingale characterization
of Markov processes \cite{SV69}, we should prove that for any limit
point ${\bf Q}$ of the collection of the laws of $Y_{N}(\theta_{N}t)$,
\[
f({\bf y}(t))-f({\bf y}(0))-\int^{t}_{0}\mathfrak{L}f({\bf y}(s))\,{\rm d}s,\qquad t\ge0
\]
is a ${\bf Q}$-martingale for any $f:S\to\mathbb{R}$. Here, $\mathfrak{L}$
is the infinitesimal generator acting on $L^{2}(S)$ corresponding
to a certain limiting Markov chain ${\bf y}(t)$. We may consider
the following alternative martingale problem \cite[Lemma 4.3.2]{EK86}
for a fixed $\lambda>0$ that
\[
e^{-\lambda t}\,f({\bf y}(t))-f({\bf y}(0))+\int^{t}_{0}e^{-\lambda s}\,(\lambda-\mathfrak{L})f({\bf y}(s))\,{\rm d}s
\]
is a ${\bf Q}$-martingale. Since ${\bf Q}$ is a limit point, we
wish to verify that
\[
e^{-\lambda t}\,f(Y_{N}(\theta_{N}t))-f(Y_{N}(0))-\int^{t}_{0}e^{-\lambda s}\,(\lambda-\mathfrak{L})f(Y_{N}(\theta_{N}s))\,{\rm d}s
\]
is a martingale plus a negligible error as $N\to\infty$. Since $Y_{N}(t)=\Phi_{N}(X^{{\rm tr}}_{N}(t))$,
this becomes
\[
e^{-\lambda t}\,(f\circ\Phi_{N})(X^{{\rm tr}}_{N}(\theta_{N}t))-(f\circ\Phi_{N})(X^{{\rm tr}}_{N}(0))-\int^{t}_{0}e^{-\lambda s}\,((\lambda-\mathfrak{L})f\circ\Phi_{N})(X^{{\rm tr}}_{N}(\theta_{N}s))\,{\rm d}s.
\]
The trace process is defined via a time change (see \eqref{eq:SN})
$X^{{\rm tr}}_{N}(t)=X_{N}(S_{N}(t))$, where $S_{N}(t)$ denotes
the total time needed for the process to spend time $t$ inside $\mathcal{E}_{N}$.
By the negligibility condition $\mathfrak{D}$, $S_{N}(t)$ is close
to $t$ as $N\to\infty$ (see \eqref{eq:ST}), thus the last display
is asymptotically close to
\[
e^{-\lambda t}\,(f\circ\Phi_{N})(X_{N}(\theta_{N}t))-(f\circ\Phi_{N})(X_{N}(0))-\int^{t}_{0}e^{-\lambda s}\,((\lambda-\mathfrak{L})f\circ\Phi_{N})(X_{N}(\theta_{N}s))\,{\rm d}s.
\]
Let us assume that $F_{N}:\Omega_{N}\to\mathbb{R}$ solves $(\lambda-\theta_{N}\mathscr{L}_{N})F_{N}=(\lambda-\mathfrak{L})f\circ\Phi_{N}$,
where $\mathscr{L}_{N}$ denotes the infinitesimal generator of $X_{N}(t)$.
Then, we may rewrite this as
\begin{equation}
e^{-\lambda t}\,(f\circ\Phi_{N})(X_{N}(\theta_{N}t))-(f\circ\Phi_{N})(X_{N}(0))-\int^{t}_{0}e^{-\lambda s}\,(\lambda-\theta_{N}\mathscr{L}_{N})F_{N}(X_{N}(\theta_{N}s))\,{\rm d}s.\label{eq:res-term}
\end{equation}
Thus, if we verify that
\begin{equation}
f\circ\Phi_{N}\simeq F_{N}\qquad\text{as}\quad N\to\infty,\label{eq:res-simple}
\end{equation}
we obtain that the penultimate display is close to
\[
e^{-\lambda t}\,F_{N}(X_{N}(\theta_{N}t))-F_{N}(X_{N}(0))-\int^{t}_{0}e^{-\lambda s}\,(\lambda-\theta_{N}\mathscr{L}_{N})F_{N}(X_{N}(\theta_{N}s))\,{\rm d}s,
\]
which is now a martingale since $\theta_{N}\mathscr{L}_{N}$ is the
infinitesimal generator of $X_{N}(\theta_{N}t)$. Thus, condition
\eqref{eq:res-simple} is exactly what we need to prove condition
$\mathfrak{C}$. This is called the resolvent condition $\mathfrak{R}$.
Surprisingly, it was proved \cite[Theorem 2.3]{LMS25} that condition
$\mathfrak{R}$ is actually equivalent to conditions $\mathfrak{D}$
and $\mathfrak{C}$, so verifying conditions $\mathfrak{M}$ and $\mathfrak{R}$
completes the analysis of metastability for a given system.

\subsection{Novelty of the article}

\subsubsection{Main result: extension to infinite $S$ case}

Recall that we assumed $S$ to be finite in the last heuristic computation.
The main theoretical contribution of this article is to extend the
result to the case when $S$ is infinite. As $\Phi_{N}(\eta)=j$ for
all $\eta\in\mathcal{E}^{j}_{N}$, we may rewrite the resolvent condition
\eqref{eq:res-simple} as
\[
\sup_{\eta\in\mathcal{E}^{j}_{N}}|F_{N}(\eta)-f(j)|\to0\qquad\text{as}\quad N\to\infty.
\]
Thus, the exact same logic applies if we could verify that
\[
\sup_{j\in S}\sup_{\eta\in\mathcal{E}^{j}_{N}}|F_{N}(\eta)-f(j)|\to0\qquad\text{as}\quad N\to\infty.
\]
Unfortunately, this is not expected to hold in general, nor is it
necessary for metastability. This is primarily because metastability
is fundamentally a local phenomenon; hence, requiring such global
stability is overly restrictive to prove, or even to expect.

The novel idea to overcome this issue is as follows. Recall the term
\eqref{eq:res-term}. Since we expect that the inter-valley metastable
transitions are well approximated by a certain Markov chain, starting
from a metastable set $\mathcal{E}^{j}_{N}$, with high probability
the process will stay in a near neighborhood of $\mathcal{E}^{j}_{N}$
within a rescaled time window $\theta_{N}T$, where $T>0$ is any
fixed real number. Equivalently, this means that for any $T>0$ and
$\eta>0$, there exists a sequence of compact sets $\mathcal{K}_{N}\subset\Omega_{N}$
chosen near $\mathcal{E}^{j}_{N}$ such that the probability to escape
$\mathcal{K}_{N}$ within time $\theta_{N}T$ should be asymptotically
smaller than $\eta$. This is given by the condition $\mathfrak{K}$.

The first main result of this paper, Theorem \ref{t:main1}, states
the following equivalence:
\[
\mathfrak{R}+\mathfrak{K}_{1}+\mathfrak{K}_{2}\qquad\Longleftrightarrow\qquad\mathfrak{C}+\mathfrak{D},
\]
where conditions $\mathfrak{K}_{1}$ and $\mathfrak{K}_{2}$ are adequate
variants of the previous condition $\mathfrak{K}$ to guarantee the
implications of both directions.

\subsubsection{Applications: particle system on infinite lattice and diffusion with
infinite stable equilibria}

We emphasize that our theory is widely applicable, particularly to
models that are null-recurrent or even transient. To support this
claim, we establish two types of metastable behavior that could not
addressed by previous approaches to metastability. As a discrete example,
we study the condensing inclusion process defined on an arbitrary
infinite graph $\mathscr{G}=(\mathscr{V},\mathscr{E})$. We prove
that, as the diffusivity constant tends to zero, the position of the
condensate performs an asymptotic random walk on $\mathscr{G}$. As
a second continuous example, we consider the overdamped Langevin dynamics
on $\mathbb{R}$ with infinitely many stable minima. Previous works
have been restricted to systems with finitely many local minima, where
the limiting Markov chain lives on a finite set. By contrast, our
framework allows us to handle infinitely many local minima, leading
to a reduced Markov chain on a countably infinite index set, which
naturally induces compactness issues into the theory.

\subsection{\label{sec1.3}Resolvent theory}

Now, we formulate the main results in a mathematically rigorous manner.
Fix locally compact Polish spaces $\Omega_{N}$, $N\ge1$, which are
the state spaces. Consider a family of continuous-time Markov processes
$\{X_{N}(t)\}_{t\ge0}$ on $\Omega_{N}$. Let $\mathscr{L}_{N}$ be
the corresponding infinitesimal generator acting on $C_{0}(\Omega_{N})$,
the space of continuous functions on $\Omega_{N}$ vanishing at infinity.
Denote by $\mathbb{Q}^{N}_{\bm{x}}$ the law on the space of c\`adl\`ag
paths $D([0,\infty);\Omega_{N})$ of the process $\{X_{N}(t)\}_{t\ge0}$
starting from $\bm{x}\in\Omega_{N}$. For simplicity, let us assume
that the process is already accelerated; i.e., $\theta_{N}\equiv1$
in terms of the previous notation.

Fix a countable index set $S$, and suppose that we are given disjoint
nonempty precompact open subsets $\mathcal{E}^{j}_{N}\subset\Omega_{N}$
for $j\in S$, which correspond to the \emph{metastable} states in
$\Omega_{N}$. We assume that the metastable states are isolated,
in the sense that
\[
\overline{\mathcal{E}}^{i}_{N}\cap\overline{\mathcal{E}}^{j}_{N}=\emptyset\qquad\text{for each}\quad i\ne j\in S,\qquad\text{and}\qquad\bigcup_{i\in S}\overline{\mathcal{E}}^{i}_{N}=\overline{\mathcal{E}}_{N},
\]
where $\overline{\mathcal{E}}^{i}_{N}$ (resp. $\overline{\mathcal{E}}_{N}$)
denotes the topological closure of $\mathcal{E}^{i}_{N}$ (resp. $\mathcal{E}_{N}$).
This notation will be adopted throughout, such that overlines indicate
topological closures. Inside each $\mathcal{E}^{j}_{N}$ consider
an open set $\widehat{\mathcal{E}}^{j}_{N}$ such that $\widehat{\mathcal{E}}^{j}_{N}\Subset\mathcal{E}^{j}_{N}$.\footnote{For two sets $A,B$, we write $A\Subset B$ if $\overline{A}\subset B$.}
The set $\widehat{\mathcal{E}}^{j}_{N}$ is introduced to provide
a buffer near the boundary $\partial\mathcal{E}^{j}_{N}$ due to analytical
issues; in the discrete case, we may simply let $\widehat{\mathcal{E}}^{j}_{N}=\mathcal{E}^{j}_{N}$.

Define as in \eqref{eq:EN-DeltaN} and let
\[
\widehat{\mathcal{E}}_{N}:=\bigcup_{j\in S}\widehat{\mathcal{E}}^{j}_{N}.
\]
The following condition $\mathfrak{D}$ states that the non-metastable
sets are negligible in terms of the typical trajectories of the process.
Below, $\mathcal{F}^{j}_{N}$, $j\in S$ are arbitrarily chosen such
that $\widehat{\mathcal{E}}^{j}_{N}\Subset\mathcal{F}^{j}_{N}\Subset\mathcal{E}^{j}_{N}$,
and let
\begin{equation}
\mathcal{F}_{N}:=\bigcup_{j\in S}\mathcal{F}^{j}_{N}.\label{eq:FN}
\end{equation}
Then it is clear that $\widehat{\mathcal{E}}_{N}\Subset\mathcal{F}_{N}\Subset\mathcal{E}_{N}$.\medskip{}

\noindent\textbf{Condition $\mathfrak{D}$ }(Delta-negligibility)\textbf{.}
For all $\{\mathcal{F}^{i}_{N}:i\in S\}$, $j\in S$, and $T>0$,\footnote{In this article, for each probability measure $Q$ its corresponding
expectation is written as ${\rm E}^{Q}$.}
\begin{equation}
\lim_{N\to\infty}\sup_{\bm{x}\in\overline{\mathcal{E}}^{j}_{N}}{\rm E}^{\mathbb{Q}^{N}_{\bm{x}}}\left[\int^{T}_{0}{\bf 1}\left\{ X_{N}(t)\notin\mathcal{F}_{N}\right\} {\rm d}t\right]=0.\label{eq:D}
\end{equation}

\medskip{}

By condition $\mathfrak{D}$, since $\Delta_{N}\subset\Omega_{N}\setminus\mathcal{F}_{N}$,
for all $j\in S$ and $T>0$,
\begin{equation}
\lim_{N\to\infty}\sup_{\bm{x}\in\overline{\mathcal{E}}^{j}_{N}}{\rm E}^{\mathbb{Q}^{N}_{\bm{x}}}\left[\int^{T}_{0}{\bf 1}\left\{ X_{N}(t)\in\Delta_{N}\right\} {\rm d}t\right]=0.\label{eq:D-Delta}
\end{equation}

Next, we characterize the metastable behavior of the processes in
terms of the convergence of laws. Denote by $T_{N}(t)$ the local
time spent by $X_{N}(\cdot)$ in $\mathcal{E}_{N}$ up to time $t\ge0$:
\[
T_{N}(t):=\int^{t}_{0}{\bf 1}\left\{ X_{N}(s)\in\mathcal{E}_{N}\right\} {\rm d}s.
\]
Let $S_{N}(\cdot)$ be the generalized inverse of the non-decreasing
process $T_{N}(\cdot)$:
\begin{equation}
S_{N}(t):=\sup\left\{ s\ge0:T_{N}(s)\le t\right\} \qquad\text{for}\quad t\ge0.\label{eq:SN}
\end{equation}

\noindent Note that $S_{N}(t)$ attains the infinite value $\infty$
with positive probability if $T_{N}(s)$ remains bounded by $t$ for
all $s\ge0$. This may happen if the process is transient and does
not return to any metastable state. We do not exclude this possibility.

Define the \emph{trace process} $\{X^{{\rm tr}}_{N}(t)\}_{t\ge0}$
by
\begin{equation}
X^{{\rm tr}}_{N}(t):=\begin{cases}
X_{N}(S_{N}(t)) & \text{if}\quad S_{N}(t)<\infty,\\
\mathfrak{d} & \text{if}\quad S_{N}(t)=\infty,
\end{cases}\label{eq:trace-def}
\end{equation}
where $\mathfrak{d}$ is an arbitrarily fixed cemetery point. Then,
since the inverse occupation time $S_{N}(t)$ is a stopping time,
the trace process $X^{{\rm tr}}_{N}(t)$ becomes a strong Markov process
on $\overline{\mathcal{E}}_{N}\cup\{\mathfrak{d}\}$ (cf. \cite[Section 6.1]{BL10}).

Further, we define a projection function $\Phi_{N}:\overline{\mathcal{E}}_{N}\cup\{\mathfrak{d}\}\to S\cup\{\mathfrak{d}\}$
as
\begin{equation}
\Phi_{N}(\eta):=\begin{cases}
j & \text{if}\quad\eta\in\overline{\mathcal{E}}^{j}_{N},\\
\mathfrak{d} & \text{if}\quad\eta=\mathfrak{d}.
\end{cases}\label{eq:PhiN}
\end{equation}
Write $S_{\mathfrak{d}}:=S\cup\{\mathfrak{d}\}$. The \emph{order
process} $\{Y_{N}(t)\}_{t\ge0}$ on $S_{\mathfrak{d}}$ is then defined
by
\begin{equation}
Y_{N}(t):=\Phi_{N}(X^{{\rm tr}}_{N}(t))\qquad\text{for}\quad t\ge0.\label{eq:YN}
\end{equation}
Note that $\{Y_{N}(t)\}_{t\ge0}$ is a process on $S_{\mathfrak{d}}$
which is not necessarily Markovian. For $\bm{x}\in\Omega_{N}$, denote
by ${\bf Q}^{N}_{\bm{x}}$ the probability measure on $D([0,\infty);S_{\mathfrak{d}})$
induced by the accelerated order process $\{Y_{N}(t)\}_{t\ge0}$ starting
from $\bm{x}$.

The metastable behavior of the original process is described by a
convergence of $\{{\bf Q}^{N}_{\bm{x}_{N}}:N\ge1\}$ to a specific
law induced by a Markov chain on $S$. In this regard, consider an
$S$-valued continuous-time Markov chain $\{{\bf y}(t)\}_{t\ge0}$
which is the limiting Markov chain. Enlarge the space $S$ to $S_{\mathfrak{d}}$
by adding $\mathfrak{d}$ as an isolated element, and denote by $\mathfrak{L}$
its corresponding infinitesimal generator acting on $C_{0}(S_{\mathfrak{d}})$.
Let ${\bf Q}_{j}$, $j\in S$, be the probability measure on $D([0,\infty);S_{\mathfrak{d}})$
induced by the Markov chain ${\bf y}(\cdot)$ starting from $j$.\medskip{}

\noindent\textbf{Condition $\mathfrak{C}$ }(Convergence)\textbf{.}
For all $j\in S$ and $\bm{x}_{N}\in\overline{\mathcal{E}}^{j}_{N}$,
the laws ${\bf Q}^{N}_{\bm{x}_{N}}$ converge weakly in the Skorokhod
topology to ${\bf Q}_{j}$ as $N\to\infty$.

\medskip{}
Thus, combining conditions $\mathfrak{C}$ and $\mathfrak{D}$, the
Markov chain $\{{\bf y}(t)\}_{t\ge0}$ serves as the \emph{reduced
model} representing the metastable behavior of the collection of processes
$(\{X_{N}(t)\}_{t\ge0}:N\ge1)$.

Now, we present a completely different type of analytic condition
related to the asymptotic flatness of specific resolvent solutions.\medskip{}

\noindent\textbf{Condition $\mathfrak{R}$ }(Resolvent)\textbf{.}
For all $\{\mathcal{F}^{i}_{N}:i\in S\}$, $\lambda>0$, ${\bf g}\in C_{0}(S)$,
and lifts $G_{N}\in C_{0}(\Omega_{N})$ of ${\bf g}$ that satisfy
\begin{equation}
G_{N}|_{\mathcal{F}^{j}_{N}}={\bf g}(j)\qquad\text{for each}\quad j\in S\qquad\text{and}\qquad\sup_{N\ge1}\|G_{N}\|_{\infty}<\infty,\label{eq:lift}
\end{equation}
the unique solution\footnote{By \cite[Proposition 3.30]{Lig10}, since $\mathscr{L}_{N}$ is an
infinitesimal generator acting on $C_{0}(\Omega_{N})$, the existence
and uniqueness of the solution are guaranteed.} $F_{N}\in C_{0}(\Omega_{N})$ to the following \emph{microscopic}
resolvent equation
\begin{equation}
(\lambda-\mathscr{L}_{N})F_{N}=G_{N}\qquad\text{in}\quad\Omega_{N}\label{eq:res}
\end{equation}
satisfies
\begin{equation}
\lim_{N\to\infty}\sup_{\bm{x}\in\overline{\mathcal{E}}^{j}_{N}}|F_{N}(\bm{x})-{\bf f}(j)|=0\qquad\text{for each}\quad j\in S,\label{eq:R}
\end{equation}
where ${\bf f}\in C_{0}(S)$ is the unique solution to the \emph{macroscopic}
resolvent equation
\begin{equation}
(\lambda-\mathfrak{L}){\bf f}={\bf g}\qquad\text{in}\quad S.\label{eq:res-chain}
\end{equation}

\medskip{}
Next, we state two supplementary conditions regarding compactness
properties.\medskip{}

\noindent\textbf{Condition $\mathfrak{K1}$ }(Compactness, 1st)\textbf{.}
For all $j\in S$ and $T,\eta>0$, there exist compact sets $\mathcal{K}_{N}=\mathcal{K}_{N}(j,T,\eta)\subset\Omega_{N}$
such that
\begin{equation}
\left\{ i\in S:\mathcal{E}^{i}_{N}\cap\mathcal{K}_{N}\ne\emptyset\quad\text{for some}\enspace N\ge1\right\} \enspace\text{is a finite set},\label{eq:K1-1}
\end{equation}
and
\begin{equation}
\limsup_{N\to\infty}\sup_{\bm{x}\in\overline{\mathcal{E}}^{j}_{N}}{\rm E}^{\mathbb{Q}^{N}_{\bm{x}}}\left[\int^{T}_{0}{\bf 1}\left\{ X_{N}(t)\in\Omega_{N}\setminus\mathcal{K}_{N}\right\} {\rm d}t\right]<\eta.\label{eq:K1-2}
\end{equation}

\medskip{}
Let $H_{A}:=\inf\,\{t\ge0:Y_{N}(t)\in A\}$ be the hitting time of
the set $A$.\medskip{}

\noindent\textbf{Condition $\mathfrak{K2}$ }(Compactness, 2nd)\textbf{.}
For all $j\in S$ and $T,\eta>0$, there exists a finite collection
$K=K(j,T,\eta)\subset S$ such that 
\begin{equation}
\limsup_{N\to\infty}\sup_{\bm{x}\in\overline{\mathcal{E}}^{j}_{N}}{\bf Q}^{N}_{\bm{x}}\left[H_{S_{\mathfrak{d}}\setminus K}\le T\right]<\eta.\label{eq:K2}
\end{equation}

\medskip{}
We introduce a technical assumption. Denote by $\mathcal{H}_{\mathcal{A}}$
the hitting time of $\mathcal{A}$ with respect to $\{X_{N}(t)\}_{t\ge0}$.
\begin{assumption}
\noindent\label{assu-tech}For all $N\ge1$, $j\in S$, and $\bm{x}\in\partial\mathcal{E}^{j}_{N}$,
\begin{equation}
\mathbb{Q}^{N}_{\bm{x}}\left[\mathcal{H}_{\mathcal{E}^{j}_{N}}=0\right]=1.\label{eq:attr}
\end{equation}
\end{assumption}

Clearly, Assumption \ref{assu-tech} is trivial if $\partial\mathcal{E}^{j}_{N}=\emptyset$,
e.g., when the system is discrete.

Now, we are ready to state our first main theorem. 
\begin{thm}
\label{t:main1}Under Assumption \ref{assu-tech}, conditions $\mathfrak{R}$,
$\mathfrak{K1}$, and $\mathfrak{K2}$ hold if and only if $\mathfrak{C}$
and $\mathfrak{D}$ hold.
\end{thm}

\begin{rem}
\label{rem1}The compactness conditions $\mathfrak{K1}$ and $\mathfrak{K2}$
did not appear in the previous studies \cite{LMS25,LLS24,LLS25} for
two reasons. First, the limiting processes considered therein were
defined on finite state spaces, so the corresponding compactness properties
were automatic. Second, the lifts $G_{N}$ defined in \eqref{eq:lift}
were chosen to be indicator-type functions determined solely by the
valleys. In the present setting, by contrast, we need the flexibility
to allow $G_{N}$ to be an arbitrary lift in $C_{0}(\Omega_{N})$.
On the one hand, condition $\mathfrak{K2}$ is necessary for the convergence
result $\mathfrak{C}$ since it is directly related to the tightness
of the collection of the projected laws $({\bf Q}^{N}_{\bm{x}_{N}}:N\ge1)$
(cf. \cite[eq. (16.22)]{Bil99}). On the other hand, condition $\mathfrak{K1}$
is a purely technical condition to deal with the fact that the infinitesimal
generators are defined on $C_{0}$ spaces, and it is crucial in the
proof of the negligibility condition $\mathfrak{D}$ (see Lemma \ref{lem2.2}).
Surprisingly, $\mathfrak{K1}$ can also be recovered from $\mathfrak{C}$
and $\mathfrak{D}$ (see Lemmas \ref{lem2.8} and \ref{lem2.9}),
thus the equivalent characterization still holds.
\end{rem}

\begin{rem}
\label{rem2}Assumption \ref{assu-tech} holds true for all models
of interest so far. Indeed, in the case of continuous-time Markov
chains, $\partial\mathcal{E}^{j}_{N}=\emptyset$ thus it is trivial.
In the case of diffusions, the metastable nature indicates a drift
towards the interior $\mathcal{E}^{j}_{N}$ from any boundary point
$\bm{x}\in\partial\mathcal{E}^{j}_{N}$, thus the process hits $\mathcal{E}^{j}_{N}$
instantaneously with probability one.
\end{rem}

In practice, to prove metastability of a given large-scale interacting
system, proving the following single condition is more efficient than
proving two separate conditions $\mathfrak{K1}$ and $\mathfrak{K2}$.\medskip{}

\noindent\textbf{Condition $\mathfrak{K}$ }(Compactness)\textbf{.}
For all $j\in S$ and $T,\eta>0$, there exist compact sets $\mathcal{K}_{N}=\mathcal{K}_{N}(j,T,\eta)\subset\Omega_{N}$
such that \eqref{eq:K1-1} holds and
\begin{equation}
\limsup_{N\to\infty}\sup_{\bm{x}\in\overline{\mathcal{E}}^{j}_{N}}\mathbb{Q}^{N}_{\bm{x}}\left[\mathcal{H}_{\Omega_{N}\setminus\mathcal{K}_{N}}\le T\right]<\eta.\label{eq:K}
\end{equation}

\medskip{}

Moreover, the arbitrary buffers $\mathcal{F}^{j}_{N}$, $j\in S$,
need not be introduced for condition $\mathfrak{R}$.\medskip{}

\noindent\textbf{Condition $\widehat{\mathfrak{R}}$ }(Resolvent,
simplified)\textbf{.} For all $\lambda>0$, ${\bf g}\in C_{0}(S)$,
and lifts $G_{N}\in C_{0}(\Omega_{N})$ of ${\bf g}$ that satisfy
\begin{equation}
G_{N}|_{\widehat{\mathcal{E}}^{j}_{N}}={\bf g}(j)\qquad\text{for each}\quad j\in S\qquad\text{and}\qquad\sup_{N\ge1}\|G_{N}\|_{\infty}<\infty,\label{eq:lift-1}
\end{equation}
the unique solution $F_{N}\in C_{0}(\Omega_{N})$ to \eqref{eq:res}
satisfies \eqref{eq:R}.

\medskip{}
Then, the following alternative characterization holds.
\begin{thm}
\label{t:main2}Under Assumption \ref{assu-tech}, if conditions $\widehat{\mathfrak{R}}$
and $\mathfrak{K}$ are true, then conditions $\mathfrak{C}$ and
$\mathfrak{D}$ hold true. In particular, \eqref{eq:D-Delta} is satisfied
by condition $\mathfrak{D}$.
\end{thm}

The proofs of Theorems \ref{t:main1} and \ref{t:main2} are given
in Section \ref{sec2}.

\subsection{\label{sec1.4}Application I: Inclusion process on infinite graphs}

As our first application, we present how one can study the metastable
behavior of interacting particle systems defined on infinite graphs.
To fix ideas, we select the condensing inclusion process defined as
follows. Fix a countably infinite (or finite) connected graph $\mathscr{G}=(\mathscr{V},\mathscr{E})$.
We assume that $\mathscr{G}$ is uniformly locally finite, i.e.,
\begin{equation}
\mathfrak{n}_{\mathscr{G}}:=\sup_{x\in\mathscr{V}}|\mathscr{N}_{x}|<\infty\qquad\text{where}\quad\mathscr{N}_{x}:=\big\{ y\in\mathscr{V}:\{x,y\}\in\mathscr{E}\big\}.\label{eq:loc-fin}
\end{equation}
Assume further that symmetric edge weights $(c_{xy})_{x,y\in\mathscr{V}}$
and site weights $(\alpha_{x})_{x\in\mathscr{V}}$ are given on $\mathscr{G}$
such that
\begin{itemize}
\item $c_{xy}>0$ if and only if $\{x,y\}\in\mathscr{E}$;
\item $\alpha_{x}>0$ for each $x\in\mathscr{V}$;
\item $\alpha_{\sup}:=\sup_{x\in\mathscr{V}}\alpha_{x}<\infty$, $\alpha_{\inf}:=\inf_{x\in\mathscr{V}}\alpha_{x}>0$,
and $c_{\sup}:=\sup_{x,y\in\mathscr{V}}c_{xy}<\infty$;
\end{itemize}
Denote by $\Omega_{N}$ the space of $N$-particle configurations
on $\mathscr{G}$:
\[
\Omega_{N}:=\left\{ \eta\in\mathbb{N}^{\mathscr{V}}:|\eta|:=\sum_{x\in\mathscr{V}}\eta_{x}=N\right\} .
\]
In this article, $\mathbb{N}:=\{0,1,2,\dots\}$. The \emph{inclusion
process} $\{\eta_{N}(t)\}_{t\ge0}$ is defined as the continuous-time
Markov chain on $\Omega_{N}$ whose generator $\mathscr{L}_{N}$ acts
as
\begin{equation}
\mathscr{L}_{N}F(\eta):=\sum_{x,y\in\mathscr{V}}c_{xy}\eta_{x}\left(\alpha_{y}+\epsilon^{-1}_{N}\eta_{y}\right)\left(F(\eta-\delta_{x}+\delta_{y})-F(\eta)\right).\label{eq:LN-inc}
\end{equation}
Above, $\delta_{z}$ is defined as $\delta_{z}(w):={\bf 1}\{z=w\}$.
The system is in the \emph{condensing} regime, i.e., $\epsilon_{N}>0$
is a parameter which vanishes sufficiently fast as $N\to\infty$:
\begin{equation}
\lim_{N\to\infty}\epsilon_{N}N\log N=0.\label{eq:eN-cond}
\end{equation}
Since the transition rate from $\eta$ to $\eta-\delta_{x}+\delta_{y}$
is positive if and only if $\eta_{x}\ge1$ and $c_{xy}>0$, the process
is irreducible on $\Omega_{N}$. Moreover, recall from Remark \ref{rem2}
that Assumption \ref{assu-tech} is obviously satisfied. Denote by
$r_{N}(\cdot,\cdot)$ the corresponding transition rate function.
\begin{rem}
The generator $\mathscr{L}_{N}$ satisfies $\mathscr{L}_{N}=\epsilon^{-1}_{N}\mathscr{L}_{N}'$
where
\[
\mathscr{L}_{N}'F(\eta)=\sum_{x,y\in\mathscr{V}}c_{xy}\eta_{x}(\eta_{y}+\epsilon_{N}\alpha_{y})\left(F(\eta-\delta_{x}+\delta_{y})-F(\eta)\right),
\]
which has a more traditional form (cf. \cite{GKR07}). Thus, $\mathscr{L}_{N}$
can be understood as an accelerated version of the original inclusion
process defined via $\mathscr{L}_{N}'$ by the metastable time scale
$\theta_{N}=\epsilon^{-1}_{N}$ (cf. \cite{GRV13}).
\end{rem}

Denote by $\xi^{x}_{N}$, $x\in\mathscr{V}$, the configuration in
$\Omega_{N}$ such that
\[
(\xi^{x}_{N})_{x}=N\qquad\text{and}\qquad(\xi^{x}_{N})_{y}=0\qquad\text{for all}\quad y\ne x.
\]

\begin{rem}
It is easy to check that the product measure
\[
\nu_{N}(\eta):=\prod_{x\in\mathscr{V}}\frac{\Gamma(\alpha_{x}\epsilon_{N}+\eta_{x})}{\Gamma(\alpha_{x}\epsilon_{N})\eta_{x}!}\qquad\text{for}\quad\eta\in\Omega_{N}
\]
is stationary with respect to the generator. Furthermore, it is reversible:
\[
\nu_{N}(\eta)r_{N}(\eta,\xi)=\nu_{N}(\xi)r_{N}(\xi,\eta)\qquad\text{for any}\quad\eta,\xi\in\Omega_{N}.
\]
One can see that $\nu_{N}$ becomes an infinite measure as soon as
$\mathscr{G}$ is an infinite graph:
\[
\nu_{N}(\Omega_{N})\ge\sum_{x\in\mathscr{V}}\nu_{N}(\xi^{x}_{N})=\sum_{x\in\mathscr{V}}\frac{\Gamma(\alpha_{x}\epsilon_{N}+N)}{\Gamma(\alpha_{x}\epsilon_{N})N!}\ge\sum_{x\in\mathscr{V}}\alpha_{x}\epsilon_{N}\ge\alpha_{\min}\epsilon_{N}|\mathscr{V}|=\infty.
\]
Thus, the particle system is not necessarily positive recurrent, nor
it is necessarily recurrent.
\end{rem}

Define
\[
\mathcal{E}^{x}_{N}:=\{\xi^{x}_{N}\}\qquad\text{for}\quad x\in\mathscr{V},\qquad\text{and}\qquad\mathcal{E}_{N}:=\bigcup_{x\in\mathscr{V}}\mathcal{E}^{x}_{N}.
\]
Since the topology is discrete, we may simply define $\widehat{\mathcal{E}}^{x}_{N}\equiv\mathcal{E}^{x}_{N}$
for each $x\in\mathscr{V}$. In addition, since each metastable set
is singleton, the local mixing property is obvious.

Then, as done in Section \ref{sec1.3}, denote by $\{\eta^{{\rm tr}}_{N}(t)\}_{t\ge0}$
the trace process on the set $\mathcal{E}_{N}\cup\{\mathfrak{d}\}$,\footnote{Mind that $\overline{\mathcal{E}}_{N}=\mathcal{E}_{N}$ in the discrete
topology.} by $\Phi_{N}:\mathcal{E}_{N}\cup\{\mathfrak{d}\}\to\mathscr{V}_{\mathfrak{d}}$
the projection given as $\Phi_{N}(\xi^{x}_{N}):=x$ and $\Phi_{N}(\mathfrak{d}):=\mathfrak{d}$,
and by $Y_{N}(t):=\Phi_{N}(\eta^{{\rm tr}}_{N}(t))$ the corresponding
order process on $\mathscr{V}_{\mathfrak{d}}$. Denote by $\mathbb{Q}^{N}_{\eta}$
(resp. ${\bf Q}^{N}_{\eta}$), $\eta\in\Omega_{N}$, the law of $\{\eta_{N}(t)\}_{t\ge0}$
(resp. $\{Y_{N}(t)\}_{t\ge0}$) on $D([0,\infty);\Omega_{N})$ (resp.
$D([0,\infty);\mathscr{V}_{\mathfrak{d}})$) starting from $\eta$.

The edge/site-weight structure defines a natural (weighted) random
walk $\{{\bf y}(t)\}_{t\ge0}$ on $\mathscr{G}=(\mathscr{V},\mathscr{E})$,
which is defined via the following infinitesimal generator $\mathfrak{L}$
acting on $C_{0}(\mathscr{V}_{\mathfrak{d}})$ as
\begin{equation}
\mathfrak{L}f(x)=\sum_{y\in\mathscr{V}}c_{xy}\alpha_{y}\left(f(y)-f(x)\right)\qquad\text{for}\quad x\in\mathscr{V},\qquad\mathfrak{L}f(\mathfrak{d})=0.\label{eq:wRW}
\end{equation}
The following is the main theorem for the inclusion process. Denote
by ${\bf Q}_{x}$ the law of $\{{\bf y}(t)\}_{t\ge0}$ on $D([0,\infty);\mathscr{V}_{\mathfrak{d}})$
starting from $x\in\mathscr{V}$.
\begin{thm}
\label{t:inc}The inclusion process $\{\eta_{N}(t)\}_{t\ge0}$ is
metastable in the following sense:
\begin{enumerate}
\item For any $x\in\mathscr{V}$, the trace laws ${\bf Q}^{N}_{\xi^{x}_{N}}$
converge weakly to the limit law ${\bf Q}_{x}$ as $N\to\infty$.
\item The excursions outside $\mathcal{E}_{N}$ is negligible, i.e., for
each $x\in\mathscr{V}$ and $T>0$,
\[
\lim_{N\to\infty}{\rm E}^{\mathbb{Q}^{N}_{\xi^{x}_{N}}}\left[\int^{T}_{0}{\bf 1}\left\{ \eta_{N}(t)\notin\mathcal{E}_{N}\right\} {\rm d}t\right]=0.
\]
\end{enumerate}
\end{thm}

The proof of Theorem \ref{t:inc} is presented in Section \ref{sec3}.

\subsection{\label{sec1.5}Application II: Non-ergodic one-dimensional diffusions}

In this subsection, we establish metastability of the one-dimensional
diffusion described by the following stochastic differential equation
(SDE):
\begin{equation}
{\rm d}\bm{x}_{\epsilon}(t)=b(\bm{x}_{\epsilon}(t))\,{\rm d}t+\sqrt{2\epsilon}\,{\rm d}\bm{w}_{t},\label{eq:SDE}
\end{equation}
where $b:\mathbb{R}\to\mathbb{R}$ is a $C^{3}$ function, $(\bm{w}_{t}:t\ge0)$
is a standard Brownian motion, and $\epsilon>0$ is a small parameter
representing the temperature of the system. We assume throughout that
this SDE admits a unique non-explosive strong solution, but does not
assume that the process is ergodic; i.e., it may be null-recurrent
or even be transient.

Set
\begin{equation}
U(x):=\begin{cases}
-\int^{x}_{0}b(y)\,{\rm d}y & \text{if}\quad x\ge0,\\
\int^{0}_{x}b(y)\,{\rm d}y & \text{if}\quad x<0,
\end{cases}\label{eq:U-def}
\end{equation}
such that $-U'=b$, and denote by $\mathcal{C}_{0}$ the set of all
critical points of $U$. We assume that $U$ is a Morse function;
that is, 
\[
U''(\mathfrak{c})\ne0\qquad\text{for every}\quad\mathfrak{c}\in\mathcal{C}_{0}.
\]
Consequently, every critical point is either a local minimum or a
local maximum. Let $\mathcal{M}_{0}$ and $\mathcal{S}_{0}$ denote
the sets of all local minima and local maxima of $U$, respectively,
such that $\mathcal{C}_{0}=\mathcal{M}_{0}\cup\mathcal{S}_{0}$.
\begin{assumption}
\label{assu:U}We assume that the following conditions hold:
\begin{enumerate}
\item We have at least two local minima: $|\mathcal{M}_{0}|\ge2$.
\item There exists $\delta>0$ such that $|\mathfrak{c}-\mathfrak{c}'|>\delta$
for all distinct $\mathfrak{c},\mathfrak{c}'\in\mathcal{C}_{0}$.
\item The second derivatives are bounded away from zero and infinity:
\[
0<\inf_{\mathfrak{c}\in\mathcal{C}_{0}}\left|U''(\mathfrak{c})\right|\le\sup_{\mathfrak{c}\in\mathcal{C}_{0}}\left|U''(\mathfrak{c})\right|<\infty.
\]
\end{enumerate}
\end{assumption}

\begin{rem}
Assumption \ref{assu:U}-(3) is imposed for technical convenience.
It ensures that the jump rates in \eqref{eq:rY} are uniformly bounded,
and consequently, that the limiting Markov chain is non-explosive.
In practice, this assumption should be readily relieved.
\end{rem}

For simplicity, assume in addition that $0\in\mathcal{M}_{0}$. We
enumerate critical points so that $\mathfrak{m}_{0}=0$ and
\[
\cdots<\mathfrak{m}_{n}<\mathfrak{s}_{n}<\mathfrak{m}_{n+1}<\mathfrak{s}_{n+1}<\cdots;\qquad n\in\mathbb{Z},
\]
where $\mathfrak{m}_{n}\in\mathcal{M}_{0}$ and $\mathfrak{s}_{n}\in\mathcal{S}_{0}$.
In particular, by Assumption \ref{assu:U}, $\lim_{n\to\infty}\mathfrak{m}_{n}=\infty$
and $\lim_{n\to-\infty}\mathfrak{m}_{n}=-\infty$. For each $n\in\mathbb{Z}$,
define the right and left barrier heights of the well $(\mathfrak{s}_{n-1},\mathfrak{s}_{n})$
by
\[
\mathfrak{h}^{+}_{n}:=U(\mathfrak{s}_{n})-U(\mathfrak{m}_{n}),\qquad\mathfrak{h}^{-}_{n}:=U(\mathfrak{s}_{n-1})-U(\mathfrak{m}_{n}),
\]
respectively, and set 
\begin{equation}
D:=\inf_{n\in\mathbb{Z}}\min\{\mathfrak{h}^{+}_{n},\mathfrak{h}^{-}_{n}\}.\label{eq:leb}
\end{equation}
Thus, $D$ is the lowest energy barrier among all the wells. We impose
the following assumptions on the barrier heights.
\begin{assumption}
\label{assu:M}
\begin{enumerate}
\item The infimum in \eqref{eq:leb} is attained, i.e., there exists $n\in\mathbb{Z}$
such that $D=\mathfrak{h}^{+}_{n}$ or $D=\mathfrak{h}^{-}_{n}$.
\item There exists $c>0$ such that for every $n\in\mathbb{Z}$, 
\[
\mathfrak{h}^{+}_{n}>D\ \Rightarrow\ \mathfrak{h}^{+}_{n}\ge D+c\qquad\text{and}\qquad\mathfrak{h}^{-}_{n}>D\ \Rightarrow\ \mathfrak{h}^{-}_{n}\ge D+c.
\]
\end{enumerate}
\end{assumption}

For $n\in\mathbb{Z}$, denote by $\nu_{n}$ and $\omega_{n}$ the
masses of $\mathfrak{m}_{n}$ and the Eyring--Kramers constant at
$\mathfrak{s}_{n}$, respectively:
\begin{equation}
\nu_{n}:=\frac{1}{\sqrt{U''(\mathfrak{m}_{n})}},\qquad\omega_{n}:=\frac{\sqrt{-U''(\mathfrak{s}_{n})}}{2\pi}.\label{eq:mass}
\end{equation}
By Assumption \ref{assu:U}-(3), the constants $\nu_{n}$ and $\omega_{n}$
are uniformly bounded away from zero and infinity.

Define a jump rate function $r_{{\bf y}}:\mathbb{Z}\times\mathbb{Z}\to[0,\infty)$
by
\begin{equation}
r_{{\bf y}}(n,k):=\begin{cases}
\frac{\omega_{n}}{\nu_{n}}{\bf 1}\{\mathfrak{h}^{+}_{n}=D\} & \text{if}\quad k=n+1,\\
\frac{\omega_{n-1}}{\nu_{n}}{\bf 1}\{\mathfrak{h}^{-}_{n}=D\} & \text{if}\quad k=n-1,\\
0 & \text{otherwise}.
\end{cases}\label{eq:rY}
\end{equation}
Let $\{{\bf y}(t)\}_{t\ge0}$ be the continuous-time Markov chain
on $\mathbb{Z}$ with jump rate $r_{{\bf y}}$.

Fix $r_{0}>0$ sufficiently small such that
\[
\mathfrak{s}_{n-1}<\mathfrak{m}_{n}-3r_{0}<\mathfrak{m}_{n}+3r_{0}<\mathfrak{s}_{n}\qquad\text{for every}\quad n\in\mathbb{Z},
\]
which is possible by Assumption \ref{assu:U}-(2). Define the metastable
sets by
\[
\mathcal{E}^{n}:=\left(\mathfrak{m}_{n}-r_{0},\mathfrak{m}_{n}+r_{0}\right),\quad\widehat{\mathcal{E}}^{n}:=\left(\mathfrak{m}_{n}-\frac{r_{0}}{2},\mathfrak{m}_{n}+\frac{r_{0}}{2}\right),\quad\text{and}\quad\mathcal{E}:=\bigcup_{n\in\mathbb{Z}}\mathcal{E}^{n}.
\]

The metastable behavior of the dynamics is observed in the time scale
$\theta_{\epsilon}:=e^{D/\epsilon}$. Define the accelerated process
$X_{\epsilon}(t):=\bm{x}_{\epsilon}(\theta_{\epsilon}t)$, the trace
process $X^{{\rm tr}}_{\epsilon}$ of $X_{\epsilon}$ on $\overline{\mathcal{E}}$
and the order process $Y_{\epsilon}$ on $\mathcal{M}_{0}$ as in
\eqref{eq:YN}.\footnote{In this one-dimensional case, it is straightforward to verify that
the inverse occupation time $S_{N}(t)$ does not blow up, and thus
the cemetery point $\mathfrak{d}$ does not have to be introduced.} Denote by $\mathbb{Q}^{\epsilon}_{x}$, ${\bf Q}^{\epsilon}_{x}$,
and $\mathbf{Q}_{n}$ the laws of $\{X_{\epsilon}(t)\}_{t\ge0}$,
$\{Y_{\epsilon}(t)\}_{t\ge0}$ starting from $x\in\mathbb{R}$, and
the law of $\{{\bf y}(t)\}_{t\ge0}$ starting from $n\in\mathbb{Z}$,
respectively.

It is well known in the one-dimensional setting, e.g. via the Girsanov
transform and Blumenthal's zero-one law, that the diffusion process
satisfies Assumption \ref{assu-tech}: For all $\epsilon>0$ and $n\in\mathbb{Z}$,
\[
\mathbb{Q}^{\epsilon}_{\mathfrak{m}_{n}-r_{0}}\left[\mathcal{H}_{\mathcal{E}^{n}}=0\right]=\mathbb{Q}^{\epsilon}_{\mathfrak{m}_{n}+r_{0}}\left[\mathcal{H}_{\mathcal{E}^{n}}=0\right]=1.
\]
The following is our last main result of the article.
\begin{thm}
\label{t:diff}Suppose that Assumptions \ref{assu:U} and \ref{assu:M}
are in force. Then, the diffusion process $\{X_{\epsilon}(t)\}_{t\ge0}$
is metastable in the following sense:
\begin{enumerate}
\item (Mixing) The process satisfies the local mixing condition: For each
$n\in\mathbb{Z}$, $\mathfrak{m}_{n}$ is the attractor, i.e., for
any collection $(x_{\epsilon})_{\epsilon>0}$ in $\mathcal{E}^{n}$,
\[
\lim_{\epsilon\to0}\mathbb{Q}^{\epsilon}_{x_{\epsilon}}\left[\mathcal{H}_{\mathfrak{m}_{n}}>\mathcal{H}_{\mathcal{E}\setminus\mathcal{E}^{n}}\right]=0.
\]
\item (Convergence in law) For any $n\in\mathbb{Z}$ and $(x_{\epsilon})_{\epsilon>0}$
in $\mathcal{E}^{n}$, the trace laws ${\bf Q}^{\epsilon}_{x_{\epsilon}}$
converge weakly to the limit law ${\bf Q}_{n}$ as $\epsilon\to0$.
\item (Delta-negligibility) The excursions outside $\mathcal{E}$ is negligible,
i.e. for each $n\in\mathbb{Z}$, $(x_{\epsilon})_{\epsilon>0}$ in
$\mathcal{E}^{n}$, and $T>0$,
\[
\lim_{\epsilon\to0}{\rm E}^{\mathbb{Q}^{\epsilon}_{x_{\epsilon}}}\left[\int^{T}_{0}{\bf 1}\left\{ X_{\epsilon}(t)\notin\mathcal{E}\right\} {\rm d}t\right]=0.
\]
\end{enumerate}
\end{thm}

The proof of Theorem \ref{t:diff} is presented in Section \ref{sec4}.
\begin{rem}
While we restrict our focus to the one-dimensional case to keep the
presentation simple and transparent, the underlying mechanisms and
proof techniques are fully manifested in this model. We strongly believe
that the main result can be generalized to higher dimensions, though
full details would require much more technical and geometric considerations.
For instance, the saddle structure is automatically simple in one
dimension, but immediately becomes a serious technical challenge in
two or higher dimensions.
\end{rem}

\begin{rem}
We have focused on the first time scale of metastable transitions,
which is determined by the lowest energy barrier $D$ defined in \eqref{eq:leb}.
The analysis can be extended, without substantial additional technical
difficulties, to all hierarchical time scales associated with the
higher energy barriers, following the approach developed in \cite{LLS25}.
\end{rem}

\section{\label{sec2}Proofs of Theorems \ref{t:main1} and \ref{t:main2}}

In Section \ref{sec2} we prove the theoretical results, Theorems
\ref{t:main1} and \ref{t:main2}. 
\begin{notation}
\label{nota:theta1}In this section, we assume that $\theta_{N}\equiv1$
for notational simplicity. This can be achieved by simply accelerating
the original process $X_{N}(t)$ \emph{a priori}, i.e., by considering
a new process $X_{N}'(t):=X_{N}(\theta_{N}t)$.
\end{notation}

Before proceeding, we present several auxiliary conditions that constitute
the building blocks of the proof.

The first condition states that the metastable jumps are asymptotically
non-instantaneous. Recall that $\mathbb{Q}^{N}_{\bm{x}}$ and ${\bf Q}^{N}_{\bm{x}}$
denote the laws of $X_{N}(\cdot)$ and $Y_{N}(\cdot)$ starting from
$\bm{x}\in\Omega_{N}$, respectively. For each $j\in S$, define $\breve{\mathcal{E}}^{j}_{N}:=\mathcal{E}_{N}\setminus\mathcal{E}^{j}_{N}$.\medskip{}

\noindent\textbf{Condition} $\mathfrak{N}$ (Non-instantaneity).
For all $j\in S$,
\begin{equation}
\lim_{t\to0}\limsup_{N\to\infty}\sup_{\bm{x}\in\overline{\mathcal{E}}^{j}_{N}}\mathbb{Q}^{N}_{\bm{x}}\left[\mathcal{H}_{\breve{\mathcal{E}}^{j}_{N}}<t\right]=0.\label{eq:N}
\end{equation}

\medskip{}
The following two conditions constitute the routine ingredients to
prove the convergence of measures in Condition $\mathfrak{C}$. The
first one corresponds to the tightness criterion due to Aldous (cf.
\cite[Theorem 16.10]{Bil99}). In below, $\mathscr{T}_{T}$ denotes
the set of all finite-range stopping times bounded by $T$.\medskip{}

\noindent\textbf{Condition} $\mathfrak{A}$ (Aldous). For any $T>0$,
$j\in S$, and $\bm{x}_{N}\in\overline{\mathcal{E}}^{j}_{N}$,
\begin{equation}
\lim_{a\to0}\limsup_{N\to\infty}\sup_{\tau\in\mathscr{T}_{T}}\sup_{\delta\in(0,a)}{\bf Q}^{N}_{\bm{x}_{N}}\left[Y_{N}(\tau+\delta)\ne Y_{N}(\tau)\right]=0.\label{eq:A}
\end{equation}

\medskip{}
The second condition corresponds to the uniqueness of limit points.\medskip{}

\noindent\textbf{Condition} $\mathfrak{U}$ (Uniqueness). For $j\in S$
and $\bm{x}_{N}\in\overline{\mathcal{E}}^{j}_{N}$, any limit point
of $\{{\bf Q}^{N}_{\bm{x}_{N}}\}_{N\ge1}$ is ${\bf Q}_{j}$.

\medskip{}
Refer to Figure \ref{fig2.1} for the overall structure of the arguments
presented in this section.

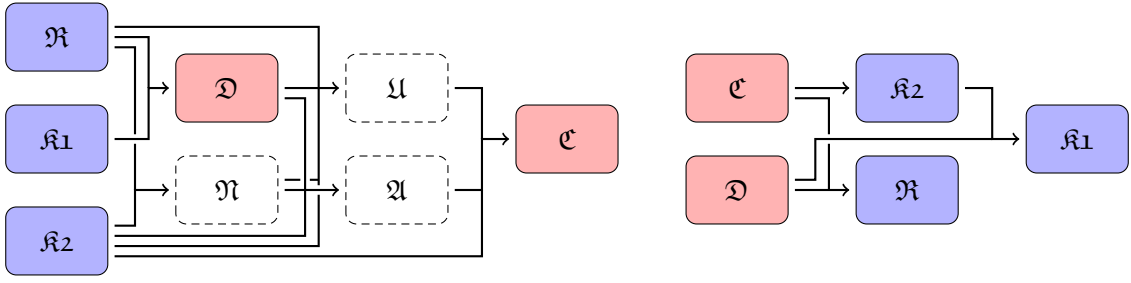
\begin{figure}
\begin{tikzpicture}[scale=0.9]
\fill[blue!30!white,rounded corners] (0,0) rectangle (1.5,1); \draw[rounded corners] (0,0) rectangle (1.5,1); \draw (0.75,0.5) node{$\mathfrak{K2}$};
\fill[blue!30!white,rounded corners] (0,1.5) rectangle (1.5,2.5); \draw[rounded corners] (0,1.5) rectangle (1.5,2.5); \draw (0.75,2) node{$\mathfrak{K1}$};
\fill[blue!30!white,rounded corners] (0,3) rectangle (1.5,4); \draw[rounded corners] (0,3) rectangle (1.5,4); \draw (0.75,3.5) node{$\mathfrak{R}$};
\draw[densely dashed,rounded corners] (2.5,0.75) rectangle (4,1.75); \draw (3.25,1.25) node{$\mathfrak{N}$};
\fill[red!30!white,rounded corners] (2.5,2.25) rectangle (4,3.25); \draw[rounded corners] (2.5,2.25) rectangle (4,3.25); \draw (3.25,2.75) node{$\mathfrak{D}$};
\draw[densely dashed,rounded corners] (5,2.25) rectangle (6.5,3.25); \draw (5.75,2.75) node{$\mathfrak{U}$};
\draw[densely dashed,rounded corners] (5,0.75) rectangle (6.5,1.75); \draw (5.75,1.25) node{$\mathfrak{A}$};
\fill[red!30!white,rounded corners] (7.5,1.5) rectangle (9,2.5); \draw[rounded corners] (7.5,1.5) rectangle (9,2.5); \draw (8.25,2) node{$\mathfrak{C}$};
\draw[thick,->] (1.6,2)--(2.1,2)--(2.1,2.75)--(2.4,2.75); \draw[thick] (1.6,3.5)--(2.1,3.5)--(2.1,2.75);
\draw[thick,->] (1.6,0.725)--(1.9,0.725)--(1.9,1.25)--(2.4,1.25); \draw[thick] (1.6,3.35)--(1.9,3.35)--(1.9,2.07); \draw[thick] (1.9,1.93)--(1.9,1.25);
\draw[thick,->] (4.1,1.25)--(4.9,1.25); \draw[thick] (1.6,0.575)--(4.4,0.575)--(4.4,1.25); \draw[thick] (4.1,2.6)--(4.4,2.6)--(4.4,1.25);
\draw[thick,->] (4.1,2.75)--(4.9,2.75); \draw[thick] (1.6,3.65)--(4.6,3.65)--(4.6,2.75); \draw[thick] (1.6,0.425)--(4.6,0.425)--(4.6,1.18); \draw[thick] (4.6,1.32)--(4.6,2.75); \draw[thick] (4.1,1.4)--(4.33,1.4); \draw[thick] (4.47,1.4)--(4.6,1.4);
\draw[thick,->] (6.6,1.25)--(7,1.25)--(7,2)--(7.4,2); \draw[thick] (6.6,2.75)--(7,2.75)--(7,2); \draw[thick] (1.6,0.275)--(7,0.275)--(7,1.25);

\begin{scope}[shift={(10,0.75)}]
\fill[red!30!white,rounded corners] (0,0) rectangle (1.5,1); \draw[rounded corners] (0,0) rectangle (1.5,1); \draw (0.75,0.5) node{$\mathfrak{D}$};
\fill[red!30!white,rounded corners] (0,1.5) rectangle (1.5,2.5); \draw[rounded corners] (0,1.5) rectangle (1.5,2.5); \draw (0.75,2) node{$\mathfrak{C}$};
\fill[blue!30!white,rounded corners] (2.5,0) rectangle (4,1); \draw[rounded corners] (2.5,0) rectangle (4,1); \draw (3.25,0.5) node{$\mathfrak{R}$};
\fill[blue!30!white,rounded corners] (2.5,1.5) rectangle (4,2.5); \draw[rounded corners] (2.5,1.5) rectangle (4,2.5); \draw (3.25,2) node{$\mathfrak{K2}$};
\fill[blue!30!white,rounded corners] (5,0.75) rectangle (6.5,1.75); \draw[rounded corners] (5,0.75) rectangle (6.5,1.75); \draw (5.75,1.25) node{$\mathfrak{K1}$};
\draw[thick,->] (1.6,2)--(2.4,2);
\draw[thick,->] (1.6,0.5)--(2.4,0.5); \draw[thick] (1.6,1.85)--(2.1,1.85)--(2.1,1.32); \draw[thick] (2.1,1.18)--(2.1,0.5);
\draw[thick,->] (4.1,2)--(4.5,2)--(4.5,1.25)--(4.9,1.25); \draw[thick] (1.6,0.65)--(1.9,0.65)--(1.9,1.25)--(4.5,1.25);
\end{scope}
\end{tikzpicture}

\caption{\label{fig2.1}Logical structure in Section \ref{sec2}. The left
diagram corresponds to the direction $\mathfrak{R}+\mathfrak{K1}+\mathfrak{K2}\Rightarrow\mathfrak{C}+\mathfrak{D}$
in Section \ref{sec2.1}, and the right diagram corresponds to the
opposite direction in Section \ref{sec2.2}.}
\end{figure}

Before going into the details, we record a technical lemma. For $A\subset S$,
define $\mathcal{E}_{N}(A):=\bigcup_{j\in A}\mathcal{E}^{j}_{N}$
and $\overline{\mathcal{E}}_{N}(A):=\bigcup_{j\in A}\overline{\mathcal{E}}^{j}_{N}$.
\begin{lem}
\label{lem:tech}
\begin{enumerate}
\item For all $\bm{x}\in\Omega_{N}$, $t\ge0$, and $A\subset S$, 
\begin{equation}
\mathbb{Q}^{N}_{\bm{x}}\left[\mathcal{H}_{\mathcal{E}_{N}(A)}\le S_{N}(t)\right]={\bf Q}^{N}_{\bm{x}}\left[H_{A\cup\{\mathfrak{d}\}}\le t\right].\label{eq:tech-1}
\end{equation}
\item Denote by $\Theta_{s}$, $s\ge0$, the time-translation operator on
$D([0,\infty);\Omega_{N})$ such that $\Theta_{s}\omega=\omega(\cdot+s)$.
Then for all $t,s\ge0$,
\[
S_{N}(t)\circ\Theta_{S_{N}(s)}+S_{N}(s)=S_{N}(t+s).
\]
\end{enumerate}
\end{lem}

\begin{proof}
(1) Note that $H_{A\cup\{\mathfrak{d}\}}\le t$ if and only if $H_{\mathfrak{d}}\le t$
or $H_{A}\le t<H_{\mathfrak{d}}$. By the definition in \eqref{eq:YN},
this is equivalent to
\begin{equation}
S_{N}(t)=\infty\qquad\text{or}\qquad X_{N}(S_{N}(s))\in\overline{\mathcal{E}}_{N}(A)\qquad\text{for some}\quad s\in[0,t].\label{eq:tech-pf-1}
\end{equation}
It suffices to prove that \eqref{eq:tech-pf-1} is equivalent to $\mathcal{H}_{\mathcal{E}_{N}(A)}\le S_{N}(t)$.
First, suppose that \eqref{eq:tech-pf-1} holds. If $S_{N}(t)=\infty$
then clearly $\mathcal{H}_{\mathcal{E}_{N}(A)}\le\infty=S_{N}(t)$.
Suppose that $X_{N}(S_{N}(s))\in\overline{\mathcal{E}}_{N}(A)$ for
some $s\in[0,t]$. Let $a:=S_{N}(s)$ such that $X_{N}(a)\in\overline{\mathcal{E}}_{N}(A)$.
Then, by \eqref{eq:SN},
\[
\int^{a+\epsilon}_{a}{\bf 1}\left\{ X_{N}(u)\in\mathcal{E}_{N}\right\} {\rm d}u>0\qquad\text{for all}\quad\epsilon>0.
\]
This, along with the fact that $X_{N}(a)\in\overline{\mathcal{E}}_{N}(A)$,
implies that $\mathcal{H}_{\mathcal{E}_{N}(A)}\le a+\epsilon$ for
all $\epsilon>0$, thus $\mathcal{H}_{\mathcal{E}_{N}(A)}\le a=S_{N}(s)\le S_{N}(t)$.

On the other hand, suppose that $b:=\mathcal{H}_{\mathcal{E}_{N}(A)}\le S_{N}(t)$.
If $S_{N}(t)=\infty$ then \eqref{eq:tech-pf-1} clearly holds, so
suppose that $S_{N}(t)<\infty$. Since the trajectory is right-continuous,
$X_{N}(b)\in\overline{\mathcal{E}}_{N}(A)$. Let
\[
s:=T_{N}(b)=\int^{b}_{0}{\bf 1}\left\{ X_{N}(u)\in\mathcal{E}_{N}\right\} {\rm d}u.
\]
By \eqref{eq:SN}, 
\begin{equation}
S_{N}(s)=S_{N}(T_{N}(b))\ge b\qquad\text{and}\qquad s=T_{N}(b)\le T_{N}(S_{N}(t))=t.\label{eq:l_tech-1}
\end{equation}
Since the trajectory is right-continuous and $\mathcal{E}_{N}(A)$
is open, for all $\epsilon>0$, there exist $0<\epsilon_{1}<\epsilon_{2}<\epsilon$
such that $X_{N}(b+u)=X_{N}(\mathcal{H}_{\mathcal{E}_{N}(A)}+u)\in\mathcal{E}_{N}(A)$
for all $u\in[\epsilon_{1},\epsilon_{2})$, so that
\[
T_{N}(b+\epsilon)>T_{N}(b)=s.
\]
Therefore, $S_{N}(s)\le b$. This, together with \eqref{eq:l_tech-1},
implies that $S_{N}(s)=b$. Thus, $S_{N}(s)\le S_{N}(t)$ which in
particular implies $s\le t$. Then, $X_{N}(S_{N}(s))=X_{N}(b)\in\overline{\mathcal{E}}_{N}(A)$
where $s\le t$. This implies that \eqref{eq:tech-pf-1} holds, concluding
the proof of part (1).

\noindent (2) If $S_{N}(s)(\omega)=\infty$ then both-hand sides are
clearly $\infty$, since $\mathfrak{d}$ is a cemetery point. Let
$S_{N}(s)(\omega)<\infty$. We calculate via \eqref{eq:SN} as
\begin{align*}
\left(S_{N}(t)\circ\Theta_{S_{N}(s)}\right) & (\omega)=\sup\left\{ u\ge0:\int^{S_{N}(s)(\omega)+u}_{S_{N}(s)(\omega)}{\bf 1}\left\{ \omega_{v}\in\mathcal{E}_{N}\right\} {\rm d}v\le t\right\} \\
=\sup & \left\{ u\ge0:\int^{S_{N}(s)(\omega)+u}_{0}{\bf 1}\left\{ \omega_{v}\in\mathcal{E}_{N}\right\} {\rm d}v\le t+\int^{S_{N}(s)(\omega)}_{0}{\bf 1}\left\{ \omega_{v}\in\mathcal{E}_{N}\right\} {\rm d}v\right\} .
\end{align*}
Since $T_{N}(S_{N}(s))=s$, the right-hand side equals
\[
\sup\left\{ u\ge0:\int^{S_{N}(s)(\omega)+u}_{0}{\bf 1}\left\{ \omega_{v}\in\mathcal{E}_{N}\right\} {\rm d}v\le t+s\right\} =S_{N}(t+s)(\omega)-S_{N}(s)(\omega),
\]
which concludes the proof of part (2).
\end{proof}

\subsection{\label{sec2.1}From resolvent to metastability}

In Section \ref{sec2.1}, we prove the only if part of Theorem \ref{t:main1},
i.e., we assume $\mathfrak{R},\mathfrak{K1},\mathfrak{K2}$ and prove
$\mathfrak{C},\mathfrak{D}$. First, we handle $\mathfrak{D}$. Recall
the notation from \eqref{eq:FN}.
\begin{lem}
\label{lem2.2}Assume that $\mathfrak{R}$ and $\mathfrak{K1}$ are
in force. Then, $\mathfrak{D}$ holds.
\end{lem}

\begin{proof}
Let $\{\mathcal{F}^{i}_{N}:i\in S\}$ and $\{\widehat{\mathcal{F}}^{i}_{N}:i\in S\}$
be such that
\[
\widehat{\mathcal{E}}^{i}_{N}\Subset\widehat{\mathcal{F}}^{i}_{N}\Subset\mathcal{F}^{i}_{N}\Subset\mathcal{E}^{i}_{N}\quad;\quad i\in S,
\]
and write
\[
\widehat{\mathcal{F}}_{N}:=\bigcup_{i\in S}\widehat{\mathcal{F}}^{i}_{N}\qquad\text{and}\qquad\mathcal{F}_{N}:=\bigcup_{i\in S}\mathcal{F}^{i}_{N}.
\]
We prove condition $\mathfrak{D}$ with $\mathcal{F}^{i}_{N}$ by
assuming condition $\mathfrak{R}$ with $\widehat{\mathcal{F}}^{i}_{N}$.

Let $j\in S$, $T>0$, and fix an arbitrary $\eta>0$. Let $\mathcal{K}_{N}=\mathcal{K}_{N}(j,T,\eta)$,
$N\ge1$, be compact subsets that satisfy $\mathfrak{K1}$. Since
$\Omega_{N}\setminus\mathcal{F}_{N}\subset(\mathcal{K}_{N}\setminus\mathcal{F}_{N})\cup(\Omega_{N}\setminus\mathcal{K}_{N})$,
the expectation in the left-hand side of \eqref{eq:D} is bounded
by
\begin{equation}
{\rm E}^{\mathbb{Q}^{N}_{\bm{x}}}\left[\int^{T}_{0}{\bf 1}\left\{ X_{N}(t)\in\mathcal{K}_{N}\setminus\mathcal{F}_{N}\right\} {\rm d}t\right]+{\rm E}^{\mathbb{Q}^{N}_{\bm{x}}}\left[\int^{T}_{0}{\bf 1}\left\{ X_{N}(t)\in\Omega_{N}\setminus\mathcal{K}_{N}\right\} {\rm d}t\right].\label{eqeq1}
\end{equation}
By the condition $\mathfrak{K1}$ (recall Notation \ref{nota:theta1}),
\begin{equation}
\limsup_{N\to\infty}\sup_{\bm{x}\in\overline{\mathcal{E}}^{j}_{N}}{\rm E}^{\mathbb{Q}^{N}_{\bm{x}}}\left[\int^{T}_{0}{\bf 1}\left\{ X_{N}(t)\in\Omega_{N}\setminus\mathcal{K}_{N}\right\} {\rm d}t\right]<\eta.\label{eqeq2}
\end{equation}
Next, we claim that for all $\lambda>0$,
\begin{equation}
\limsup_{N\to\infty}\sup_{\bm{x}\in\overline{\mathcal{E}}^{j}_{N}}{\rm E}^{\mathbb{Q}^{N}_{\bm{x}}}\left[\int^{\infty}_{0}e^{-\lambda t}\,{\bf 1}\left\{ X_{N}(t)\in\mathcal{K}_{N}\setminus\mathcal{F}_{N}\right\} {\rm d}t\right]=0.\label{eqeq3}
\end{equation}
Take $G_{N}\in C_{0}(\Omega_{N})$ such that $G_{N}\equiv1$ in $\mathcal{K}_{N}\setminus\mathcal{F}_{N}$
and $G_{N}\equiv0$ in $\widehat{\mathcal{F}}_{N}$ (in fact, $G_{N}$
is compactly supported in this way), which is possible by Urysohn's
lemma since $\widehat{\mathcal{F}}_{N}\Subset\mathcal{F}_{N}$. Let
$F_{N}\in C_{0}(\Omega_{N})$ be the solution to \eqref{eq:res} such
that
\begin{equation}
G_{N}|_{\widehat{\mathcal{F}}^{i}_{N}}={\bf g}(i)\qquad\text{for each}\quad i\in S\qquad\text{and}\qquad\sup_{N\ge1}\|G_{N}\|_{\infty}<\infty.\label{eq:lift-2}
\end{equation}
By \eqref{eq:pr},
\begin{equation}
F_{N}(\bm{x})={\rm E}^{\mathbb{Q}^{N}_{\bm{x}}}\left[\int^{\infty}_{0}e^{-\lambda t}\,G_{N}(X_{N}(t))\,{\rm d}t\right].\label{eqeq4}
\end{equation}
Since $G_{N}\ge{\bf 1}_{\mathcal{K}_{N}\setminus\mathcal{F}_{N}}$,
\begin{equation}
F_{N}(\bm{x})\ge{\rm E}^{\mathbb{Q}^{N}_{\bm{x}}}\left[\int^{\infty}_{0}e^{-\lambda t}\,{\bf 1}\left\{ X_{N}(t)\in\mathcal{K}_{N}\setminus\mathcal{F}_{N}\right\} {\rm d}t\right].\label{eqeq5}
\end{equation}
Notice that the two functions, $0\in C_{0}(S)$ and $G_{N}\in C_{0}(\Omega_{N})$,
satisfy the lifting condition \eqref{eq:lift-2}. Since ${\bf f}\equiv0$
is the unique solution to 
\[
(\lambda-\mathfrak{L}){\bf f}=0,
\]
by condition $\mathfrak{R}$ we obtain
\begin{equation}
\lim_{N\to\infty}\sup_{\bm{x}\in\overline{\mathcal{E}}^{j}_{N}}|F_{N}(\bm{x})|=0,\label{eqeq6}
\end{equation}
so that \eqref{eqeq3} is verified by \eqref{eqeq5} and \eqref{eqeq6}.

Finally, by \eqref{eqeq3},
\begin{equation}
\begin{aligned} & \limsup_{N\to\infty}\sup_{\bm{x}\in\overline{\mathcal{E}}^{j}_{N}}{\rm E}^{\mathbb{Q}^{N}_{\bm{x}}}\left[\int^{T}_{0}{\bf 1}\left\{ X_{N}(t)\in\mathcal{K}_{N}\setminus\mathcal{F}_{N}\right\} {\rm d}t\right]\\
 & \le\limsup_{N\to\infty}\sup_{\bm{x}\in\overline{\mathcal{E}}^{j}_{N}}e^{\lambda T}\,{\rm E}^{\mathbb{Q}^{N}_{\bm{x}}}\left[\int^{\infty}_{0}e^{-\lambda t}\,{\bf 1}\left\{ X_{N}(t)\in\mathcal{K}_{N}\setminus\mathcal{F}_{N}\right\} {\rm d}t\right]=0.
\end{aligned}
\label{eqeq7}
\end{equation}
Therefore, by \eqref{eqeq1}, \eqref{eqeq2}, and \eqref{eqeq7},
we have proved that
\[
\limsup_{N\to\infty}\sup_{\bm{x}\in\overline{\mathcal{E}}^{j}_{N}}{\rm E}^{\mathbb{Q}^{N}_{\bm{x}}}\left[\int^{T}_{0}{\bf 1}\left\{ X_{N}(t)\in\Omega_{N}\setminus\mathcal{F}_{N}\right\} {\rm d}t\right]<\eta,
\]
which is valid for any $\eta>0$. This concludes the proof of Condition
$\mathfrak{D}$.
\end{proof}

Next, we prove the tightness condition $\mathfrak{A}$ via the non-instantaneity
condition $\mathfrak{N}$.
\begin{lem}
\label{lem2.3}Conditions $\mathfrak{R}$ and $\mathfrak{K2}$ imply
Condition $\mathfrak{N}$.
\end{lem}

\begin{proof}
Fix $j\in S$ and an arbitrary $\eta>0$. By Condition $\mathfrak{K2}$
with any $T>0$ and Lemma \ref{lem:tech}-(1), there exists a finite
$K\subset S$ such that
\begin{equation}
\limsup_{N\to\infty}\sup_{\bm{x}\in\overline{\mathcal{E}}^{j}_{N}}\mathbb{Q}^{N}_{\bm{x}}\left[\mathcal{H}_{\mathcal{E}_{N}(S\setminus K)}\le S_{N}(T)\right]<\eta.\label{eq1}
\end{equation}
Consider the following probability for $T>0$ and $\bm{x}\in\overline{\mathcal{E}}^{j}_{N}$:
\[
\mathbb{Q}^{N}_{\bm{x}}\left[\mathcal{H}_{\overline{\mathcal{E}}_{N}(S\setminus K)}\le S_{N}(T)\right].
\]
This probability can be decomposed into
\begin{equation}
\mathbb{Q}^{N}_{\bm{x}}\left[\mathcal{H}_{\mathcal{E}_{N}(S\setminus K)}\le S_{N}(T)\right]+\mathbb{Q}^{N}_{\bm{x}}\left[\mathcal{H}_{\overline{\mathcal{E}}_{N}(S\setminus K)}\le S_{N}(T)<\mathcal{H}_{\mathcal{E}_{N}(S\setminus K)}\right].\label{eq2}
\end{equation}
By applying the strong Markov property at the stopping time $\mathcal{H}_{\overline{\mathcal{E}}_{N}(S\setminus K)}$,
the second probability in \eqref{eq2} can be bounded by
\begin{equation}
\sup_{\bm{y}\in\overline{\mathcal{E}}_{N}(S\setminus K)}\mathbb{Q}^{N}_{\bm{y}}\left[\mathcal{H}_{\mathcal{E}_{N}(S\setminus K)}>0\right]=\sup_{\bm{y}\in\partial\mathcal{E}_{N}(S\setminus K)}\mathbb{Q}^{N}_{\bm{y}}\left[\mathcal{H}_{\mathcal{E}_{N}(S\setminus K)}>0\right]=0,\label{eq3}
\end{equation}
where the second equality vanishes uniformly by Assumption \ref{assu-tech}.
Thus, collecting \eqref{eq1}, \eqref{eq2}, and \eqref{eq3}, we
have proved that
\begin{equation}
\limsup_{N\to\infty}\sup_{\bm{x}\in\overline{\mathcal{E}}^{j}_{N}}\mathbb{Q}^{N}_{\bm{x}}\left[\mathcal{H}_{\overline{\mathcal{E}}_{N}(S\setminus K)}\le S_{N}(T)\right]<\eta.\label{eq4}
\end{equation}
Substituting $T=1$ in \eqref{eq4} and applying the obvious bound
$S_{N}(1)\ge1$, we obtain that
\begin{equation}
\limsup_{N\to\infty}\sup_{\bm{x}\in\overline{\mathcal{E}}^{j}_{N}}\mathbb{Q}^{N}_{\bm{x}}\left[\mathcal{H}_{\overline{\mathcal{E}}_{N}(S\setminus K)}\le1\right]<\eta.\label{eq5}
\end{equation}

Now, abbreviate $\mathcal{H}:=\mathcal{H}_{\breve{\mathcal{E}}^{j}_{N}}$.
Observe that $\mathcal{H}<t$ implies $X_{N}(t\wedge\mathcal{H})=X_{N}(\mathcal{H})\in\overline{\breve{\mathcal{E}}}^{j}_{N}$.
Thus, to prove \eqref{eq:N} it suffices to prove that
\begin{equation}
\limsup_{t\to0}\limsup_{N\to\infty}\sup_{\bm{x}\in\overline{\mathcal{E}}^{j}_{N}}\mathbb{Q}^{N}_{\bm{x}}\left[X_{N}(t\wedge\mathcal{H})\in\overline{\breve{\mathcal{E}}}^{j}_{N}\right]=0.\label{eq6}
\end{equation}
For $t$ smaller than $1$, we may decompose the last probability
as
\begin{equation}
\begin{aligned} & \mathbb{Q}^{N}_{\bm{x}}\left[X_{N}(t\wedge\mathcal{H})\in\overline{\breve{\mathcal{E}}}^{j}_{N},\ \mathcal{H}_{\overline{\mathcal{E}}_{N}(S\setminus K)}>1\right]+\mathbb{Q}^{N}_{\bm{x}}\left[X_{N}(t\wedge\mathcal{H})\in\overline{\breve{\mathcal{E}}}^{j}_{N},\ \mathcal{H}_{\overline{\mathcal{E}}_{N}(S\setminus K)}\le1\right]\\
 & \le\mathbb{Q}^{N}_{\bm{x}}\left[X_{N}(t\wedge\mathcal{H})\in\overline{\mathcal{E}}_{N}(K\setminus\{j\})\right]+\mathbb{Q}^{N}_{\bm{x}}\left[\mathcal{H}_{\overline{\mathcal{E}}_{N}(S\setminus K)}\le1\right].
\end{aligned}
\label{eq7}
\end{equation}

Next, we handle the first probability in the right-hand side of \eqref{eq7}.
Consider ${\bf f}:=\delta_{j}\in C_{0}(S)$ and fix $\lambda>0$.
Define ${\bf g}:=(\lambda-\mathfrak{L}){\bf f}\in C_{0}(S)$. Note
that ${\bf g}(j)=\lambda-\mathfrak{L}{\bf f}(j)\ge\lambda$ and ${\bf g}(i)=-\mathfrak{L}{\bf f}(i)\le0$
for $i\ne j$, thus
\begin{equation}
\sup_{i\in S}{\bf g}(i)={\bf g}(j)=\lambda-\mathfrak{L}{\bf f}(j).\label{eq8}
\end{equation}
Define $G_{N}\in C_{0}(\Omega_{N})$ as a lift of ${\bf g}$ satisfying
\eqref{eq:lift} such that (cf. \eqref{eq:EN-DeltaN})
\begin{equation}
G_{N}=0\qquad\text{in}\quad\Delta_{N}\qquad\text{and}\qquad\sup_{\bm{x}\in\Omega_{N}}G_{N}(\bm{x})=\sup_{i\in S}{\bf g}(i).\label{eq9}
\end{equation}
Indeed, this is possible by Urysohn's lemma since $\overline{\mathcal{F}}_{N}\cap\Delta_{N}=\emptyset$.
Denote by $F_{N}$ the unique solution to the resolvent equation given
in \eqref{eq:res}. Then, the process $\{M_{N}(t)\}_{t\ge0}$ defined
by
\[
M_{N}(t):=F_{N}(X_{N}(t))-F_{N}(X_{N}(0))-\int^{t}_{0}(\mathscr{L}_{N}F_{N})(X_{N}(s))\,{\rm d}s
\]
is a $\mathbb{Q}^{N}_{\bm{x}}$-martingale. Evaluating both sides
at the stopping time $t\wedge\mathcal{H}$, by the stopping time theorem
and \eqref{eq:res}, it holds that
\begin{equation}
{\rm E}^{\mathbb{Q}^{N}_{\bm{x}}}\left[F_{N}(X_{N}(t\wedge\mathcal{H}))\right]=F_{N}(\bm{x})+{\rm E}^{\mathbb{Q}^{N}_{\bm{x}}}\left[\int^{t\wedge\mathcal{H}}_{0}(\lambda F_{N}-G_{N})(X_{N}(s))\,{\rm d}s\right].\label{eq10}
\end{equation}
By Condition $\mathfrak{R}$ and the fact that ${\bf f}(j)=1$,
\[
\lim_{N\to\infty}\sup_{\bm{x}\in\overline{\mathcal{E}}^{j}_{N}}|1-F_{N}(\bm{x})|=0.
\]
Note that $\|F_{N}\|_{\infty}\le\lambda^{-1}\|G_{N}\|_{\infty}$ by
\eqref{eq:pr}. Thus, $\|\lambda F_{N}-G_{N}\|_{\infty}\le2\|G_{N}\|_{\infty}$,
which implies along with the last displayed identity that
\begin{equation}
\sup_{\bm{x}\in\overline{\mathcal{E}}^{j}_{N}}\left|1-F_{N}(\bm{x})-{\rm E}^{\mathbb{Q}^{N}_{\bm{x}}}\left[\int^{t\wedge\mathcal{H}}_{0}(\lambda F_{N}-G_{N})(X_{N}(s))\,{\rm d}s\right]\right|\le r(N)+2t\|G\|_{\infty},\label{eq11}
\end{equation}
where $\lim_{N\to\infty}r(N)=0$. Moreover, by the fact that $\mathscr{L}_{N}$
is an infinitesimal generator (cf. \cite[Definition 3.12-(b)]{Lig10}),
\eqref{eq8}, and \eqref{eq9},
\[
\sup_{\bm{x}\in\Omega_{N}}F_{N}(\bm{x})\le\frac{1}{\lambda}\,\sup_{\bm{x}\in\Omega_{N}}G_{N}(\bm{x})=1-\frac{\mathfrak{L}{\bf f}(j)}{\lambda}.
\]
Finally, since $\lim_{N\to\infty}\sup_{\bm{x}\in\overline{\mathcal{E}}_{N}(K\setminus\{j\})}|F_{N}(\bm{x})|=0$
by Condition $\mathfrak{R}$ and the fact that $K\setminus\{j\}$
is finite, $1-F_{N}\ge\frac{1}{2}$ in $\overline{\mathcal{E}}_{N}(K\setminus\{j\})$
for $N$ big enough, which implies along with the last displayed inequality
that
\begin{equation}
{\rm E}^{\mathbb{Q}^{N}_{\bm{x}}}\left[1-F_{N}(X_{N}(t\wedge\mathcal{H}))\right]\ge\frac{1}{2}\,\mathbb{Q}^{N}_{\bm{x}}\left[X_{N}(t\wedge\mathcal{H})\in\overline{\mathcal{E}}_{N}(K\setminus\{j\})\right]+\frac{\mathfrak{L}{\bf f}(j)}{\lambda}.\label{eq12}
\end{equation}
Collecting \eqref{eq10}, \eqref{eq11}, and \eqref{eq12}, we obtain
that
\[
\sup_{\bm{x}\in\overline{\mathcal{E}}^{j}_{N}}\mathbb{Q}^{N}_{\bm{x}}\left[X_{N}(t\wedge\mathcal{H})\in\overline{\mathcal{E}}_{N}(K\setminus\{j\})\right]\le2r(N)+4t\|G\|_{\infty}-\frac{2\mathfrak{L}{\bf f}(j)}{\lambda}.
\]
This proves that
\[
\limsup_{t\to0}\limsup_{N\to\infty}\sup_{\bm{x}\in\overline{\mathcal{E}}^{j}_{N}}\mathbb{Q}^{N}_{\bm{x}}\left[X_{N}(t\wedge\mathcal{H})\in\overline{\mathcal{E}}_{N}(K\setminus\{j\})\right]\le-\frac{2\mathfrak{L}{\bf f}(j)}{\lambda}.
\]
Sending $\lambda\to\infty$ implies
\begin{equation}
\limsup_{t\to0}\limsup_{N\to\infty}\sup_{\bm{x}\in\overline{\mathcal{E}}^{j}_{N}}\mathbb{Q}^{N}_{\bm{x}}\left[X_{N}(t\wedge\mathcal{H})\in\overline{\mathcal{E}}_{N}(K\setminus\{j\})\right]=0.\label{eq13}
\end{equation}
Finally, collecting \eqref{eq5}, \eqref{eq7}, and \eqref{eq13},
we conclude that
\[
\limsup_{t\to0}\limsup_{N\to\infty}\sup_{\bm{x}\in\overline{\mathcal{E}}^{j}_{N}}\mathbb{Q}^{N}_{\bm{x}}\left[X_{N}(t\wedge\mathcal{H})\in\overline{\breve{\mathcal{E}}}^{j}_{N}\right]<\eta,
\]
which proves \eqref{eq6} since $\eta>0$ was chosen arbitrarily.
\end{proof}

\begin{lem}
\label{lem2.4}Conditions $\mathfrak{K2}$, $\mathfrak{D}$, and $\mathfrak{N}$
imply Condition $\mathfrak{A}$.
\end{lem}

\begin{proof}
Fix $T>0$, $j\in S$, and $\bm{x}_{N}\in\overline{\mathcal{E}}^{j}_{N}$.
For any $\eta>0$, recall from \eqref{eq4} that Condition $\mathfrak{K2}$
guarantees the existence of a finite index set $K\subset S$ such
that
\begin{equation}
\limsup_{N\to\infty}\sup_{\bm{x}\in\overline{\mathcal{E}}^{j}_{N}}\mathbb{Q}^{N}_{\bm{x}}\left[\mathcal{H}_{\overline{\mathcal{E}}_{N}(S\setminus K)}\le S_{N}(T)\right]<\eta.\label{eq:pf-A1}
\end{equation}
Then, we estimate
\begin{equation}
\begin{aligned} & {\bf Q}^{N}_{\bm{x}_{N}}\left[Y_{N}(\tau+\delta)\ne Y_{N}(\tau)\right]=\mathbb{Q}^{N}_{\bm{x}_{N}}\left[\Phi_{N}(X_{N}(S_{N}(\tau+\delta)))\ne\Phi_{N}(X_{N}(S_{N}(\tau)))\right]\\
 & \le\mathbb{Q}^{N}_{\bm{x}_{N}}\left[\mathcal{H}_{\overline{\mathcal{E}}_{N}(S\setminus K)}\le S_{N}(T)\right]+\mathbb{Q}^{N}_{\bm{x}_{N}}\left[\Phi_{N}(X_{N}(S_{N}(\tau+\delta)))\ne\Phi_{N}(X_{N}(S_{N}(\tau))),\ \mathcal{H}_{\overline{\mathcal{E}}_{N}(S\setminus K)}>S_{N}(T)\right].
\end{aligned}
\label{eq:pf-A2}
\end{equation}
We focus on the second probability in the right-hand side of \eqref{eq:pf-A2}.
Since $\tau\in\mathscr{T}_{T}$ and $\mathcal{H}_{\overline{\mathcal{E}}_{N}(S\setminus K)}>S_{N}(T)\ge S_{N}(\tau)$,
by the strong Markov property at $S_{N}(\tau)$ and Lemma \ref{lem:tech}-(2),
\begin{align*}
\sup_{\tau\in\mathscr{T}_{T}}\sup_{\delta\in(0,a)}\mathbb{Q}^{N}_{\bm{x}_{N}} & \left[\Phi_{N}(X_{N}(S_{N}(\tau+\delta)))\ne\Phi_{N}(X_{N}(S_{N}(\tau))),\ \mathcal{H}_{\overline{\mathcal{E}}_{N}(S\setminus K)}>S_{N}(T)\right]\\
 & \le\sup_{\delta\in(0,a)}\sup_{j\in K}\sup_{\bm{x}\in\overline{\mathcal{E}}^{k}_{N}}\mathbb{Q}^{N}_{\bm{x}}\left[\Phi_{N}(X_{N}(S_{N}(\delta)))\ne j\right].
\end{align*}
By Lemma \ref{lem:tech}-(1),
\begin{equation}
\begin{aligned}\sup_{\delta\in(0,a)}\sup_{j\in K}\sup_{\bm{x}\in\overline{\mathcal{E}}^{j}_{N}}\mathbb{Q}^{N}_{\bm{x}}\left[\Phi_{N}(X_{N}(S_{N}(\delta)))\ne j\right] & =\sup_{\delta\in(0,a)}\sup_{j\in K}\sup_{\bm{x}\in\overline{\mathcal{E}}^{j}_{N}}{\bf Q}^{N}_{\bm{x}}\left[Y_{N}(\delta)\ne j\right].\\
 & \le\sup_{j\in K}\sup_{\bm{x}\in\overline{\mathcal{E}}^{j}_{N}}{\bf Q}^{N}_{\bm{x}}\left[H_{S_{\mathfrak{d}}\setminus\{j\}}\le a\right].\\
 & \le\sup_{j\in K}\sup_{\bm{x}\in\overline{\mathcal{E}}^{j}_{N}}\mathbb{Q}^{N}_{\bm{x}}\left[\mathcal{H}_{\breve{\mathcal{E}}^{j}_{N}}\le S_{N}(a)\right].
\end{aligned}
\label{eq:pf-A3}
\end{equation}
Note that from the Markov inequality and \eqref{eq:D-Delta}, for
any $b>0$,
\begin{equation}
\begin{aligned}\sup_{\bm{x}\in\overline{\mathcal{E}}^{j}_{N}}\mathbb{Q}^{N}_{\bm{x}}\left[S_{N}(a)\ge a+b\right] & =\sup_{\bm{x}\in\overline{\mathcal{E}}^{j}_{N}}\mathbb{Q}^{N}_{\bm{x}}\left[\int^{a+b}_{0}{\bf 1}\left\{ X_{N}(t)\in\Delta_{N}\right\} {\rm d}t\ge b\right]\\
 & \le\sup_{\bm{x}\in\overline{\mathcal{E}}^{j}_{N}}\frac{1}{b}\,{\rm E}^{\mathbb{Q}^{N}_{\bm{x}}}\left[\int^{a+b}_{0}{\bf 1}\left\{ X_{N}(t)\in\Delta_{N}\right\} {\rm d}t\right]\xrightarrow{N\to\infty}0,
\end{aligned}
\label{eq:ST}
\end{equation}
where the equality in the first line holds by the definition of $S_{N}(a)$
in \eqref{eq:SN}. Applying \eqref{eq:ST} to \eqref{eq:pf-A3} for
all $j\in K$ with $b=a$, which is possible since $K$ is finite,
\[
\limsup_{N\to\infty}\sup_{j\in K}\sup_{\bm{x}\in\overline{\mathcal{E}}^{j}_{N}}\mathbb{Q}^{N}_{\bm{x}}\left[\mathcal{H}_{\breve{\mathcal{E}}^{j}_{N}}\le S_{N}(a)\right]\le\limsup_{N\to\infty}\sup_{j\in K}\sup_{\bm{x}\in\overline{\mathcal{E}}^{j}_{N}}\mathbb{Q}^{N}_{\bm{x}}\left[\mathcal{H}_{\breve{\mathcal{E}}^{j}_{N}}<2a\right].
\]
By Condition $\mathfrak{N}$, since $K$ is finite, the right-hand
side converges to $0$ as $a\to0$. Combining this result with \eqref{eq:pf-A1},
\eqref{eq:pf-A2}, and \eqref{eq:pf-A3}, we conclude that
\[
\lim_{a\to0}\limsup_{N\to\infty}\sup_{\tau\in\mathscr{T}_{T}}\sup_{\delta\in(0,a)}{\bf Q}^{N}_{\bm{x}_{N}}\left[Y_{N}(\tau+\delta)\ne Y_{N}(\tau)\right]<\eta.
\]
As this is valid for all $\eta>0$, we have proved Condition $\mathfrak{A}$.
\end{proof}

The last objective is to prove Condition $\mathfrak{U}$.
\begin{lem}
\label{lem2.5}Suppose that $\mathfrak{R}$, $\mathfrak{K2}$, $\mathfrak{D}$,
and $\mathfrak{N}$ are in force. Then, $\mathfrak{U}$ holds.
\end{lem}

\begin{proof}
Let $j\in S$ and $\bm{x}_{N}\in\overline{\mathcal{E}}^{j}_{N}$.
Take a limiting probability measure ${\bf Q}$ of $\{{\bf Q}^{N}_{\bm{x}_{N}}:N\ge1\}$.
Note that $\mathcal{H}_{\mathcal{E}^{j}_{N}}=0$ in Assumption \ref{assu-tech}.
Since the trajectory is right-continuous and $\mathcal{E}^{j}_{N}$
is open, for all $\epsilon>0$, there exist $0<\epsilon_{1}<\epsilon_{2}<\epsilon$
such that $X_{N}(u)\in\mathcal{E}^{j}_{N}$ for all $u\in[\epsilon_{1},\epsilon_{2})$,
so that
\[
T_{N}(\epsilon)>T_{N}(0)=0.
\]
This implies $S_{N}(0)=0$, so that $X_{N}(S_{N}(0))=\bm{x}_{N}\in\overline{\mathcal{E}}^{j}_{N}$.
Hence, \eqref{eq:attr} implies that
\[
\liminf_{N\to\infty}{\bf Q}^{N}_{\bm{x}_{N}}\left[Y_{N}(0)=j\right]=1,
\]
and this implies that 
\begin{equation}
{\bf Q}\,[{\bf y}(0)=j]=1.\label{eq:ini-j}
\end{equation}

Next, we prove that 
\begin{equation}
{\bf Q}\left[{\bf y}(u)\ne{\bf y}(u-)\right]=0\qquad\text{for all}\quad u>0.\label{eq:U1}
\end{equation}
Since every trajectory in $D([0,\infty);S)$ is c\`adl\`ag,
\[
\{{\bf y}(u)\ne{\bf y}(u-)\}=\bigcup^{\infty}_{n=1}\left\{ {\bf y}(u-a)\ne{\bf y}(u),\enskip\forall a\in\left(0,\frac{1}{n}\right),\enskip{\bf y}(u-)\ne{\bf y}(u)\right\} ,
\]
where each set in the right-hand side is open. Thus, it suffices to
prove for each $n\ge1$ that
\[
\lim_{N\to\infty}{\bf Q}^{N}_{\bm{x}_{N}}\left\{ Y_{N}(u-a)\ne Y_{N}(u),\enskip\forall a\in\left(0,\frac{1}{n}\right),\enskip Y_{N}(u-)\ne Y_{N}(u)\right\} =0.
\]
Since the event in the left-hand side implies $Y_{N}(u-\frac{1}{2^{m}n})\ne Y_{N}(u)$
for all $m\ge1$, the last display holds if we can verify that
\begin{equation}
\limsup_{m\to\infty}\limsup_{N\to\infty}\mathbb{Q}^{N}_{\bm{x}_{N}}\left\{ \Phi_{N}\left(X_{N}\left(S_{N}\left(u-\frac{1}{2^{m}n}\right)\right)\right)\ne\Phi_{N}\left(X_{N}(S_{N}(u))\right)\right\} =0.\label{eq:U2}
\end{equation}
For any $\eta>0$, recall from \eqref{eq4} that Condition $\mathfrak{K2}$
provides a finite set $K$ such that
\[
\limsup_{N\to\infty}\sup_{\bm{x}\in\overline{\mathcal{E}}^{j}_{N}}\mathbb{Q}^{N}_{\bm{x}}\left[\mathcal{H}_{\overline{\mathcal{E}}_{N}(S\setminus K)}\le S_{N}(u)\right]<\eta.
\]
Thus, the left-hand side in \eqref{eq:U2} is bounded by
\[
\eta+\limsup_{m\to\infty}\limsup_{N\to\infty}\mathbb{Q}^{N}_{\bm{x}_{N}}\left[\Phi_{N}\left(X_{N}\left(S_{N}\left(u-\frac{1}{2^{m}n}\right)\right)\right)\ne\Phi_{N}\left(X_{N}(S_{N}(u))\right),\enskip\mathcal{H}_{\overline{\mathcal{E}}_{N}(S\setminus K)}>S_{N}(u)\right].
\]
By the strong Markov property at $S_{N}(u-\frac{1}{2^{m}n})$, along
with Lemma \ref{lem:tech}-(2), the above probability is bounded by
\[
\sum_{j\in K}\sup_{\bm{x}\in\overline{\mathcal{E}}^{j}_{N}}\mathbb{Q}^{N}_{\bm{x}}\left[\Phi_{N}\left(X_{N}\left(S_{N}\left(\frac{1}{2^{m}n}\right)\right)\right)\ne j\right].
\]
By Lemma \ref{lem:tech}-(1),
\[
\begin{aligned}\sum_{j\in K}\sup_{\bm{x}\in\overline{\mathcal{E}}^{j}_{N}}\mathbb{Q}^{N}_{\bm{x}}\left[\Phi_{N}\left(X_{N}\left(S_{N}\left(\frac{1}{2^{m}n}\right)\right)\right)\ne j\right] & =\sum_{j\in K}\sup_{\bm{x}\in\overline{\mathcal{E}}^{j}_{N}}{\bf Q}^{N}_{\bm{x}}\left[Y_{N}\left(\frac{1}{2^{m}n}\right)\ne j\right]\\
 & \le\sum_{j\in K}\sup_{\bm{x}\in\overline{\mathcal{E}}^{j}_{N}}{\bf Q}^{N}_{\bm{x}}\left[H_{S_{\mathfrak{d}}\setminus\{j\}}\le\frac{1}{2^{m}n}\right].\\
 & \le\sum_{j\in K}\sup_{\bm{x}\in\overline{\mathcal{E}}^{j}_{N}}\mathbb{Q}^{N}_{\bm{x}}\left[\mathcal{H}_{\breve{\mathcal{E}}^{j}_{N}}\le S_{N}\left(\frac{1}{2^{m}n}\right)\right].
\end{aligned}
\]
In addition, recall from \eqref{eq:ST} that \eqref{eq:D} implies
\[
\limsup_{N\to\infty}\sum_{j\in K}\sup_{\bm{x}\in\overline{\mathcal{E}}^{j}_{N}}\mathbb{Q}^{N}_{\bm{x}}\left[\mathcal{H}_{\breve{\mathcal{E}}^{j}_{N}}\le S_{N}\left(\frac{1}{2^{m}n}\right)\right]\le\sum_{j\in K}\limsup_{N\to\infty}\sup_{\bm{x}\in\overline{\mathcal{E}}^{j}_{N}}\mathbb{Q}^{N}_{\bm{x}}\left[\mathcal{H}_{\breve{\mathcal{E}}^{j}_{N}}<\frac{1}{2^{m-1}n}\right].
\]
By condition $\mathfrak{N}$, the right-hand side converges to $0$
as $m\to\infty$. Since $\eta>0$ was arbitrary, we have proved \eqref{eq:U2}
thus \eqref{eq:U1}.

Finally, take any $\lambda>0$ and ${\bf f}:S_{\mathfrak{d}}\to\mathbb{R}$
in the domain of $\mathfrak{L}$. Let ${\bf g}:=(\lambda-\mathfrak{L}){\bf f}$,
let $G_{N}$ be any lift of ${\bf g}|_{S}$ such that $G_{N}|_{\mathcal{F}^{j}_{N}}={\bf g}(j)$
for $j\in S$ (which exists by Urysohn's lemma), and let $F_{N}$
be the solution to \eqref{eq:res}. To prove that ${\bf Q}={\bf Q}_{j}$,
by the martingale characterization and \eqref{eq:ini-j}, it suffices
to prove that
\[
{\rm E}^{{\bf Q}}\left[\mathfrak{M}^{s,t}_{{\bf f}}({\bf y})\,\mathfrak{H}({\bf y})\right]=0,
\]
where $\mathfrak{M}^{s,t}_{{\bf f}}$, $s<t$, are the martingale
differences,
\[
\mathfrak{M}^{s,t}_{{\bf f}}({\bf y}):={\bf f}({\bf y}(t))-{\bf f}({\bf y}(s))+\int^{t}_{s}e^{-\lambda u}\,{\bf g}({\bf y}(u))\,{\rm d}u,
\]
and $\mathfrak{H}:D([0,\infty);S_{\mathfrak{d}})\to\mathbb{R}$ is
an arbitrary bounded continuous function that is measurable up to
time $s$, i.e., $\mathfrak{H}\in\sigma({\bf y}(u):0\le u\le s)$.
Note that \eqref{eq:U1} guarantees the convergence
\[
\lim_{N\to\infty}{\rm E}^{{\bf Q}^{N}_{\bm{x}_{N}}}\left[\mathfrak{M}^{s,t}_{{\bf f}}(Y_{N})\,\mathfrak{H}(Y_{N})\right]={\rm E}^{{\bf Q}}\left[\mathfrak{M}^{s,t}_{{\bf f}}({\bf y})\,\mathfrak{H}({\bf y})\right],
\]
since outside the set $\{{\bf y}(u-)\ne{\bf y}(u)\}$ the projection
function at time $u$ is continuous. Thus, to prove the lemma we are
left to verify that
\[
\lim_{N\to\infty}{\rm E}^{{\bf Q}^{N}_{\bm{x}_{N}}}\left[\mathfrak{M}^{s,t}_{{\bf f}}(Y_{N})\,\mathfrak{H}(Y_{N})\right]=0.
\]
By \eqref{eq4} and Condition $\mathfrak{K2}$, for any $\eta>0$
there exists finite $K\subset S$ such that 
\[
\limsup_{N\to\infty}\mathbb{Q}^{N}_{\bm{x}_{N}}\left[\mathcal{H}_{\overline{\mathcal{E}}_{N}(S\setminus K)}\le S_{N}(t)\right]<\eta.
\]
Then,
\begin{align*}
 & {\rm E}^{{\bf Q}^{N}_{\bm{x}_{N}}}\left[\mathfrak{M}^{s,t}_{{\bf f}}(Y_{N})\,\mathfrak{H}(Y_{N})\right]\\
 & ={\rm E}^{\mathbb{Q}^{N}_{\bm{x}_{N}}}\left[\mathfrak{M}^{s,t}_{{\bf f}}(Y_{N})\,\mathfrak{H}(Y_{N})\left({\bf 1}\left\{ \mathcal{H}_{\overline{\mathcal{E}}_{N}(S\setminus K)}\le S_{N}(t)\right\} +{\bf 1}\left\{ \mathcal{H}_{\overline{\mathcal{E}}_{N}(S\setminus K)}>S_{N}(t)\right\} \right)\right]\\
 & \le C\eta+{\rm E}^{\mathbb{Q}^{N}_{\bm{x}_{N}}}\left[\mathfrak{M}^{s,t}_{{\bf f}}(Y_{N})\,\mathfrak{H}(Y_{N})\,{\bf 1}\left\{ \mathcal{H}_{\overline{\mathcal{E}}_{N}(S\setminus K)}>S_{N}(t)\right\} \right],
\end{align*}
where $C>0$ is a finite constant that depends only on $\|{\bf f}\|_{\infty}$,
$\|{\bf g}\|_{\infty}$, and $\|\mathfrak{H}\|_{\infty}$. Next, we
may apply Condition $\mathfrak{R}$ on the set $\overline{\mathcal{E}}_{N}(K)$
to rewrite the last term above as
\begin{align*}
 & {\rm E}^{\mathbb{Q}^{N}_{\bm{x}_{N}}}\left[\mathfrak{M}^{s,t}_{{\bf f}}(Y_{N})\,\mathfrak{H}(Y_{N})\,{\bf 1}\left\{ \mathcal{H}_{\overline{\mathcal{E}}_{N}(S\setminus K)}>S_{N}(t)\right\} \right]\\
 & ={\rm E}^{\mathbb{Q}^{N}_{\bm{x}_{N}}}\left[\left(F_{N}(X^{{\rm tr}}_{N}(t))-F_{N}(X^{{\rm tr}}_{N}(s))+\int^{t}_{s}e^{-\lambda u}\,{\bf g}(Y_{N}(u))\,{\rm d}u\right)\mathfrak{H}(Y_{N})\right]+o_{N}(1),
\end{align*}
where $o_{N}(1)$ denotes a term vanishing as $N\to\infty$. Moreover,
combining \eqref{eq:ST} (from \eqref{eq:D}) and Condition $\mathfrak{D}$,
we may approximate ${\bf g}(Y_{N}(u))$ by $G_{N}(X^{{\rm tr}}_{N}(u))$
in the right-hand side to obtain that the last expectation equals
\[
{\rm E}^{\mathbb{Q}^{N}_{\bm{x}_{N}}}\left[\left(F_{N}(X^{{\rm tr}}_{N}(t))-F_{N}(X^{{\rm tr}}_{N}(s))+\int^{t}_{s}e^{-\lambda u}\,G_{N}(X^{{\rm tr}}_{N}(u))\,{\rm d}u\right)\mathfrak{H}(Y_{N})\right]+o_{N}(1).
\]
Finally, to conclude the proof of the lemma, it suffices to apply
the martingale characterization of the original process $X_{N}(\cdot)$
in $\Omega_{N}$, along with the time-changing formulas
\begin{equation}
\limsup_{N\to\infty}\sup_{\bm{x}\in\overline{\mathcal{E}}^{j}_{N}}{\rm E}^{\mathbb{Q}^{N}_{\bm{x}}}\left[e^{-\lambda t}-e^{-\lambda S_{N}(t)}\right]=0=\limsup_{N\to\infty}\sup_{\bm{x}\in\overline{\mathcal{E}}^{j}_{N}}{\rm E}^{\mathbb{Q}^{N}_{\bm{x}}}\left[\int^{t}_{0}\left(e^{-\lambda u}-e^{-\lambda S_{N}(u)}\right){\rm d}u\right],\label{eq:tc}
\end{equation}
which follows by Condition $\mathfrak{D}$. Rigorous verifications
of the final two assertions can be found in \cite[Proposition 4.5 and Lemma 4.4]{LMS25},
respectively, thus we omit the detail (see from \eqref{eq:ST} that
$S_{N}(t)<\infty$ uniformly starting from $\overline{\mathcal{E}}^{j}_{N}$
with high probability).
\end{proof}

Finally, we formally state that Condition $\mathfrak{C}$ follows.
\begin{lem}
\label{lem2.6}Conditions $\mathfrak{K2}$, $\mathfrak{A}$, and $\mathfrak{U}$
imply Condition $\mathfrak{C}$.
\end{lem}

\begin{proof}
Two conditions $\mathfrak{K2}$ and $\mathfrak{A}$ guarantee that
$\{{\bf Q}^{N}_{\bm{x}_{N}}:N\ge1\}$ is tight by Aldous' tightness
criterion (see, e.g., \cite[Theorem 16.10]{Bil99}). Condition $\mathfrak{U}$
implies that any limit point must be ${\bf Q}_{j}$. These two ingredients
prove Condition $\mathfrak{C}$.
\end{proof}

\begin{proof}[Proof of the only if part of Theorem \ref{t:main1}]
 Lemmas \ref{lem2.2}, \ref{lem2.3}, \ref{lem2.4}, \ref{lem2.5},
and \ref{lem2.6} complete the proof of the only if part. See also
Figure \ref{fig2.1}-left.
\end{proof}

\subsection{\label{sec2.2}From metastability to resolvent}

In this subsection, we prove the other direction, i.e., $\mathfrak{C}+\mathfrak{D}\Rightarrow\mathfrak{R}+\mathfrak{K1}+\mathfrak{K2}$.
\begin{lem}
\label{lem2.7}Conditions $\mathfrak{C}$ and $\mathfrak{D}$ imply
Condition $\mathfrak{R}$.
\end{lem}

\begin{proof}
First, note that the convergence of measures ${\bf Q}^{N}_{\bm{x}_{N}}$
to ${\bf Q}_{j}$ is uniform over all possible $\bm{x}_{N}\in\overline{\mathcal{E}}^{j}_{N}$.
Indeed, if the uniformity fails, then we may subtract a sequence $\{\bm{x}_{N}':N\ge1\}$
with $\bm{x}_{N}'\in\overline{\mathcal{E}}^{j}_{N}$ such that ${\bf Q}^{N}_{\bm{x}_{N}'}$
stays bounded away from ${\bf Q}_{j}$, thus does not converge to
${\bf Q}_{j}$. This contradicts Condition $\mathfrak{C}$.

Fix $\lambda>0$, ${\bf g}\in C_{0}(S)$, and $G_{N}\in C_{0}(\Omega_{N})$
that satisfies \eqref{eq:lift} and $\|G_{N}\|_{\infty}=M<\infty$.
Fix $j\in S$. By Condition $\mathfrak{C}$, along with the uniformity
stated in the last paragraph,
\[
\lim_{N\to\infty}\sup_{\bm{x}\in\overline{\mathcal{E}}^{j}_{N}}\left|{\rm E}^{{\bf Q}^{N}_{\bm{x}}}\left[\int^{\infty}_{0}e^{-\lambda t}\,{\bf g}(Y_{N}(t))\,{\rm d}t\right]-{\rm E}^{{\bf Q}_{j}}\left[\int^{\infty}_{0}e^{-\lambda t}\,{\bf g}({\bf y}(t))\,{\rm d}t\right]\right|=0,
\]
where the second expectation equals ${\bf f}(j)$ by \eqref{eq:pr}.
Thus, it suffices to demonstrate that
\begin{equation}
\lim_{N\to\infty}\sup_{\bm{x}\in\overline{\mathcal{E}}^{j}_{N}}\left|{\rm E}^{\mathbb{Q}^{N}_{\bm{x}}}\left[\int^{\infty}_{0}e^{-\lambda t}\,G_{N}(X_{N}(t))\,{\rm d}t\right]-{\rm E}^{{\bf Q}^{N}_{\bm{x}}}\left[\int^{\infty}_{0}e^{-\lambda t}\,{\bf g}(Y_{N}(t))\,{\rm d}t\right]\right|=0.\label{eq:2.7-wts}
\end{equation}
We first claim that
\begin{equation}
\lim_{N\to\infty}\sup_{\bm{x}\in\overline{\mathcal{E}}^{j}_{N}}{\rm E}^{\mathbb{Q}^{N}_{\bm{x}}}\left[\int^{\infty}_{0}e^{-\lambda t}\,G_{N}(X_{N}(t))\,{\bf 1}\left\{ X_{N}(t)\notin\mathcal{F}_{N}\right\} {\rm d}t\right]=0.\label{eq:2.7-1}
\end{equation}
Fix $\eta>0$ and take $T>0$ such that $\lambda^{-1}e^{-\lambda T}M\le\eta$.
By Condition $\mathfrak{D}$, 
\[
\lim_{N\to\infty}\sup_{\bm{x}\in\overline{\mathcal{E}}^{j}_{N}}{\rm E}^{\mathbb{Q}^{N}_{\bm{x}}}\left[\int^{T}_{0}{\bf 1}\left\{ X_{N}(t)\notin\mathcal{F}_{N}\right\} {\rm d}t\right]=0.
\]
Thus, we may calculate the expectation in \eqref{eq:2.7-1} as
\begin{align*}
 & {\rm E}^{\mathbb{Q}^{N}_{\bm{x}}}\left[\left(\int^{T}_{0}+\int^{\infty}_{T}\right)e^{-\lambda t}\,G_{N}(X_{N}(t))\,{\bf 1}\left\{ X_{N}(t)\notin\mathcal{F}_{N}\right\} {\rm d}t\right]\\
 & \le M\left({\rm E}^{\mathbb{Q}^{N}_{\bm{x}}}\left[\int^{T}_{0}{\bf 1}\left\{ X_{N}(t)\notin\mathcal{F}_{N}\right\} {\rm d}t\right]+\lambda^{-1}e^{-\lambda T}\right)\xrightarrow{\lim_{N\to\infty}\sup_{\bm{x}\in\overline{\mathcal{E}}^{j}_{N}}}\lambda^{-1}e^{-\lambda T}M\le\eta.
\end{align*}
This proves \eqref{eq:2.7-1} since $\eta>0$ was chosen arbitrarily.

Next, by time changing we get
\begin{align*}
 & {\rm E}^{\mathbb{Q}^{N}_{\bm{x}}}\left[\int^{\infty}_{0}e^{-\lambda t}\,G_{N}(X_{N}(t))\,{\bf 1}\left\{ X_{N}(t)\in\mathcal{F}_{N}\right\} {\rm d}t\right]\\
 & ={\rm E}^{\mathbb{Q}^{N}_{\bm{x}}}\left[\int^{\infty}_{0}e^{-\lambda S_{N}(t)}\,G_{N}(X^{{\rm tr}}_{N}(t))\,{\bf 1}\left\{ X^{{\rm tr}}_{N}(t)\in\mathcal{F}_{N}\right\} {\rm d}t\right]\\
 & ={\rm E}^{\mathbb{Q}^{N}_{\bm{x}}}\left[\int^{\infty}_{0}e^{-\lambda S_{N}(t)}\,{\bf g}(Y_{N}(t))\,{\bf 1}\left\{ X^{{\rm tr}}_{N}(t)\in\mathcal{F}_{N}\right\} {\rm d}t\right],
\end{align*}
where the second equality holds by \eqref{eq:YN} and \eqref{eq:lift}.
As done in \eqref{eq:2.7-1}, it also holds that
\[
\lim_{N\to\infty}\sup_{\bm{x}\in\overline{\mathcal{E}}^{j}_{N}}{\rm E}^{\mathbb{Q}^{N}_{\bm{x}}}\left[\int^{\infty}_{0}e^{-\lambda S_{N}(t)}\,{\bf g}(Y_{N}(t))\,{\bf 1}\left\{ X^{{\rm tr}}_{N}(t)\notin\mathcal{F}_{N}\right\} {\rm d}t\right]=0,
\]
thus the $\lim_{N\to\infty}\sup_{\bm{x}\in\overline{\mathcal{E}}^{j}_{N}}$
of
\begin{equation}
\left|{\rm E}^{\mathbb{Q}^{N}_{\bm{x}}}\left[\int^{\infty}_{0}e^{-\lambda t}\,G_{N}(X_{N}(t))\,{\bf 1}\left\{ X_{N}(t)\in\mathcal{F}_{N}\right\} {\rm d}t\right]-{\rm E}^{\mathbb{Q}^{N}_{\bm{x}}}\left[\int^{\infty}_{0}e^{-\lambda S_{N}(t)}\,{\bf g}(Y_{N}(t))\,{\rm d}t\right]\right|\label{eq:2.7-2}
\end{equation}
equals $0$. By a time-changing argument via \eqref{eq:tc} (guaranteed
by Condition $\mathfrak{D}$), we know that
\begin{equation}
\lim_{N\to\infty}\sup_{\bm{x}\in\overline{\mathcal{E}}^{j}_{N}}\left|{\rm E}^{\mathbb{Q}^{N}_{\bm{x}}}\left[\int^{\infty}_{0}e^{-\lambda S_{N}(t)}\,{\bf g}(Y_{N}(t))\,{\rm d}t\right]-{\rm E}^{{\bf Q}^{N}_{\bm{x}}}\left[\int^{\infty}_{0}e^{-\lambda t}\,{\bf g}(Y_{N}(t))\,{\rm d}t\right]\right|=0.\label{eq:2.7-3}
\end{equation}
Combining \eqref{eq:2.7-1}, \eqref{eq:2.7-2}, and \eqref{eq:2.7-3}
proves \eqref{eq:2.7-wts}, and thus the lemma.
\end{proof}

\begin{lem}
\label{lem2.8}Condition $\mathfrak{C}$ implies Condition $\mathfrak{K2}$.
\end{lem}

\begin{proof}
Condition $\mathfrak{C}$ automatically implies tightness, which then
implies that (cf. \cite[eq. (16.22)]{Bil99}) given any $j\in S$,
$T>0$, and $\bm{x}_{N}\in\overline{\mathcal{E}}^{j}_{N}$,
\begin{equation}
\limsup_{n\to\infty}\limsup_{N\to\infty}{\bf Q}^{N}_{\bm{x}_{N}}\left[H_{S_{\mathfrak{d}}\setminus K_{n}}\le T\right]=0,\label{eq:2.8}
\end{equation}
where $K_{1},K_{2},\dots$ are increasing finite sets such that $\bigcup^{\infty}_{n=1}K_{n}=S_{\mathfrak{d}}$.
Now suppose the contrary of $\mathfrak{K2}$ such that there exist
$j\in S$ and $\eta,T>0$ such that for all finite $K\subset S$,
\[
\limsup_{N\to\infty}\sup_{\bm{x}\in\overline{\mathcal{E}}^{j}_{N}}{\bf Q}^{N}_{\bm{x}}\left[H_{S_{\mathfrak{d}}\setminus K}\le T\right]>\eta.
\]
Then, there exist subsequences $N_{1}<N_{2}<\cdots$ and $\bm{x}_{N_{1}},\bm{x}_{N_{2}},\dots\in\overline{\mathcal{E}}^{j}_{N}$
such that
\[
{\bf Q}^{N_{i}}_{\bm{x}_{N_{i}}}\left[H_{S_{\mathfrak{d}}\setminus K_{n}}\le T\right]\ge\eta\qquad\text{for all}\quad i,n\ge1.
\]
This contradicts to \eqref{eq:2.8}, which concludes the proof.
\end{proof}

\begin{lem}
\label{lem2.9}Conditions $\mathfrak{D}$ and $\mathfrak{K2}$ imply
Condition $\mathfrak{K1}$.
\end{lem}

\begin{proof}
Fix $j\in S$ and $T,\eta>0$. Since $S_{N}(T)\ge T$, it follows
from Lemma \ref{lem:tech}-(1) that Condition $\mathfrak{K2}$ provides
us a finite $K$ with
\begin{equation}
\limsup_{N\to\infty}\sup_{\bm{x}\in\overline{\mathcal{E}}^{j}_{N}}\mathbb{Q}^{N}_{\bm{x}}\left[\mathcal{H}_{\mathcal{E}_{N}(S\setminus K)}\le T\right]\le\limsup_{N\to\infty}\sup_{\bm{x}\in\overline{\mathcal{E}}^{j}_{N}}\mathbb{Q}^{N}_{\bm{x}}\left[\mathcal{H}_{\mathcal{E}_{N}(S\setminus K)}\le S_{N}(T)\right]<\frac{\eta}{T}.\label{eq2.9-1}
\end{equation}
Let $\mathcal{K}_{N}:=\bigcup_{i\in K}\overline{\mathcal{E}}^{i}_{N}$,
which is compact since each $\overline{\mathcal{E}}^{i}_{N}$ is compact
and $K$ is finite. Then the set in \eqref{eq:K1-1} equals $K$,
thus \eqref{eq:K1-1} holds. Next, since $\Omega_{N}\setminus\mathcal{K}_{N}\subset(\Omega_{N}\setminus\mathcal{E}_{N})\cup(\mathcal{E}_{N}\setminus\mathcal{K}_{N})$,
\begin{align*}
 & {\rm E}^{\mathbb{Q}^{N}_{\bm{x}}}\left[\int^{T}_{0}{\bf 1}\left\{ X_{N}(t)\in\Omega_{N}\setminus\mathcal{K}_{N}\right\} {\rm d}t\right]\\
 & \le{\rm E}^{\mathbb{Q}^{N}_{\bm{x}}}\left[\int^{T}_{0}{\bf 1}\left\{ X_{N}(t)\in\Omega_{N}\setminus\mathcal{E}_{N}\right\} {\rm d}t\right]+{\rm E}^{\mathbb{Q}^{N}_{\bm{x}}}\left[\int^{T}_{0}{\bf 1}\left\{ X_{N}(t)\in\mathcal{E}_{N}\setminus\mathcal{K}_{N}\right\} {\rm d}t\right].
\end{align*}
The first expectation in the right-hand side vanishes in the limit
$N\to\infty$ by \eqref{eq:D-Delta}. For the second expectation in
the right-hand side, note that $\mathcal{E}_{N}\setminus\mathcal{K}_{N}=\mathcal{E}_{N}(S\setminus K)$.
Thus,
\begin{equation}
{\rm E}^{\mathbb{Q}^{N}_{\bm{x}}}\left[\int^{T}_{0}{\bf 1}\left\{ X_{N}(t)\in\mathcal{E}_{N}\setminus\mathcal{K}_{N}\right\} {\rm d}t\right]\le T\,\mathbb{Q}^{N}_{\bm{x}}\left[\mathcal{H}_{\mathcal{E}_{N}(S\setminus K)}\le T\right].\label{eq:2.9-2}
\end{equation}
Thus, by \eqref{eq2.9-1},
\[
\limsup_{N\to\infty}\sup_{\bm{x}\in\overline{\mathcal{E}}^{j}_{N}}{\rm E}^{\mathbb{Q}^{N}_{\bm{x}}}\left[\int^{T}_{0}{\bf 1}\left\{ X_{N}(t)\in\Omega_{N}\setminus\mathcal{K}_{N}\right\} {\rm d}t\right]\le\limsup_{N\to\infty}\sup_{\bm{x}\in\overline{\mathcal{E}}^{j}_{N}}T\,\mathbb{Q}^{N}_{\bm{x}}\left[\mathcal{H}_{\mathcal{E}_{N}(S\setminus K)}\le T\right]<\eta.
\]
This finishes the proof.
\end{proof}

\begin{proof}[Proof of the if part of Theorem \ref{t:main1}]
 Lemmas \ref{lem2.7}, \ref{lem2.8}, and \ref{lem2.9} finish the
verification. See also Figure \ref{fig2.1}-right.
\end{proof}

\subsection{\label{sec2.3}Proof of Theorem \ref{t:main2}}

In this last subsection of Section \ref{sec2}, we prove Theorem \ref{t:main2},
i.e., we verify that conditions $\widehat{\mathfrak{R}}$ and $\mathfrak{K}$
are sufficient to verify Conditions $\mathfrak{C}$ and $\mathfrak{D}$.
By Theorem \ref{t:main1}, it suffices to prove that $\mathfrak{R}+\mathfrak{K}\Rightarrow\mathfrak{K1}+\mathfrak{K2}$
and $\widehat{\mathfrak{R}}\Rightarrow\mathfrak{R}$. First, we demonstrate
that $\mathfrak{K}$ implies $\mathfrak{K1}$.
\begin{lem}
\label{lem2.10}Condition $\mathfrak{K}$ implies condition $\mathfrak{K1}$.
\end{lem}

\begin{proof}
The idea is exactly the same as in \eqref{eq:2.9-2}; we bound
\[
{\rm E}^{\mathbb{Q}^{N}_{\bm{x}}}\left[\int^{T}_{0}{\bf 1}\left\{ X_{N}(t)\in\Omega_{N}\setminus\mathcal{K}_{N}\right\} {\rm d}t\right]\le T\,\mathbb{Q}^{N}_{\bm{x}}\left[\mathcal{H}_{\Omega_{N}\setminus\mathcal{K}_{N}}\le T\right],
\]
and then apply Condition $\mathfrak{K}$.
\end{proof}

Now, by Lemma \ref{lem2.2} we know that $\mathfrak{R}+\mathfrak{K1}\Rightarrow\mathfrak{D}$,
thus we can make use of $\mathfrak{D}$ as well.
\begin{lem}
\label{lem2.11}Conditions $\mathfrak{K}$ and $\mathfrak{D}$ imply
condition $\mathfrak{K2}$.
\end{lem}

\begin{proof}
Recall from \eqref{eq:ST} that \eqref{eq:D-Delta} gives
\[
\limsup_{N\to\infty}\sup_{\bm{x}\in\overline{\mathcal{E}}^{j}_{N}}\mathbb{Q}^{N}_{\bm{x}}\left[S_{N}(a)\ge2a\right]=0\qquad\text{for all}\quad a>0.
\]
Combining this with Condition $\mathfrak{K}$, we obtain that for
any $j\in S$ and $T,\eta>0$, there exist a compact $\mathcal{K}_{N}\subset\Omega_{N}$,
$N\ge1$, such that
\[
K:=\left\{ i\in S:\mathcal{E}^{i}_{N}\cap\mathcal{K}_{N}\ne\emptyset\quad\text{for some}\enspace N\ge1\right\} 
\]
is finite and
\[
\limsup_{N\to\infty}\sup_{\bm{x}\in\overline{\mathcal{E}}^{j}_{N}}\mathbb{Q}^{N}_{\bm{x}}\left[\mathcal{H}_{\mathcal{K}^{c}_{N}}\le S_{N}(T)\right]<\eta.
\]
Then, since $\mathcal{H}_{\mathcal{E}_{N}(S\setminus K)}\ge\mathcal{H}_{\Omega_{N}\setminus\mathcal{K}_{N}}$,
via Lemma \ref{lem:tech}-(1), we conclude that
\begin{align*}
\limsup_{N\to\infty}\sup_{\bm{x}\in\overline{\mathcal{E}}^{j}_{N}}{\bf Q}^{N}_{\bm{x}}\left[H_{S_{\mathfrak{d}}\setminus K}\le T\right] & =\limsup_{N\to\infty}\sup_{\bm{x}\in\overline{\mathcal{E}}^{j}_{N}}\mathbb{Q}^{N}_{\bm{x}}\left[\mathcal{H}_{\mathcal{E}_{N}(S\setminus K)}\le S_{N}(T)\right]\\
 & \le\limsup_{N\to\infty}\sup_{\bm{x}\in\overline{\mathcal{E}}^{j}_{N}}\mathbb{Q}^{N}_{\bm{x}}\left[\mathcal{H}_{\Omega_{N}\setminus\mathcal{K}_{N}}\le S_{N}(T)\right]<\eta.
\end{align*}
This is exactly Condition $\mathfrak{K2}$.
\end{proof}

\begin{lem}
\label{lem:2.13}Condition $\widehat{\mathfrak{R}}$ implies condition
$\mathfrak{R}$.
\end{lem}

\begin{proof}
Let $\mathcal{F}_{j}$, $j\in S$, be any sets such that $\widehat{\mathcal{E}}^{j}\Subset\mathcal{F}^{j}\Subset\mathcal{E}^{j}$.
Fix $\lambda>0$ and ${\bf g}:S\to\mathbb{R}$. Let $G_{N}\in C_{0}(\Omega_{N})$
be the lift of ${\bf g}$ that satisfies the hypothesis \eqref{eq:lift}
of condition $\mathfrak{R}$. Since $\widehat{\mathcal{E}}^{j}\subset\mathcal{F}^{j}$,
$G_{N}$ also satisfies the hypothesis \eqref{eq:lift-1} of $\widehat{\mathfrak{R}}$.
Therefore, by condition $\widehat{\mathfrak{R}}$, the unique solution
$F_{N}\in C_{0}(\Omega_{N})$ to \eqref{eq:res} satisfies \eqref{eq:R},
which is condition $\mathfrak{R}$.
\end{proof}

\begin{proof}[Proof of Theorem \ref{t:main2}]
This is now immediate by Lemmas \ref{lem2.2}, \ref{lem2.10}, \ref{lem2.11},
and \ref{lem:2.13}, along with Theorem \ref{t:main1}.
\end{proof}

\section{\label{sec3}Metastability of the Inclusion Process}

In this section, we prove the metastable behavior of the inclusion
process on $\mathscr{G}=(\mathscr{V},\mathscr{E})$, i.e., we prove
Theorem \ref{t:inc}. Mind that we take $\widehat{\mathcal{E}}^{x}_{N}\equiv\mathcal{F}^{x}_{N}\equiv\mathcal{E}^{x}_{N}$
for any $x\in\mathscr{V}$.

By Theorem \ref{t:main2}, it suffices to check the following two
propositions corresponding to Conditions $\widehat{\mathfrak{R}}$
and $\mathfrak{K}$, respectively, in Section \ref{sec1.3}.
\begin{prop}
\label{p:inc1}For all $\lambda>0$, ${\bf g}\in C_{0}(\mathscr{V}_{\mathfrak{d}})$,
and $G_{N}\in C_{0}(\Omega_{N})$ such that $G_{N}(\xi^{x}_{N})={\bf g}(x)$
for all $x\in\mathscr{V}$ and $\sup_{N\ge1}\|G_{N}\|_{\infty}<\infty$,
the unique solution $F_{N}\in C_{0}(\Omega_{N})$ to $(\lambda-\mathscr{L}_{N})F_{N}=G_{N}$
satisfies
\begin{equation}
\lim_{N\to\infty}F_{N}(\xi^{x}_{N})={\bf f}(x)\qquad\text{for each}\quad x\in\mathscr{V},\label{eq:R-inc}
\end{equation}
where ${\bf f}\in C_{0}(\mathscr{V}_{\mathfrak{d}})$ solves $(\lambda-\mathfrak{L}){\bf f}={\bf g}$
(cf. \eqref{eq:wRW}).
\end{prop}

\begin{prop}
\label{p:inc2}For all $x\in\mathscr{V}$ and $T,\epsilon>0$, there
exist compact sets $\mathcal{K}_{N}$, $N\ge1$ such that
\begin{equation}
\left\{ y\in\mathscr{V}:\xi^{y}_{N}\in\mathcal{K}_{N}\quad\text{for some}\enspace N\ge1\right\} \quad\text{is a finite set},\label{eq:K1-1-inc}
\end{equation}
and
\begin{equation}
\limsup_{N\to\infty}\mathbb{Q}^{N}_{\xi^{x}_{N}}\left[\mathcal{H}_{\Omega_{N}\setminus\mathcal{K}_{N}}\le T\right]<\epsilon.\label{eq:K-inc}
\end{equation}
\end{prop}

\begin{proof}[Proof of Theorem \ref{t:inc}]
 The theorem follows from Propositions \ref{p:inc1}, \ref{p:inc2},
and Theorem \ref{t:main2}.
\end{proof}

The rest of the section is devoted to proving the two propositions.

For $x,y\in\mathscr{V}$, denote by $\xi^{xy}_{i,N-i}\in\Omega_{N}$
the configuration that satisfies
\[
(\xi^{xy}_{i,N-i})_{x}=i\qquad\text{and}\qquad(\xi^{xy}_{i,N-i})_{y}=N-i.
\]
Note that $\xi^{xy}_{N,0}=\xi^{x}_{N}$ and $\xi^{xy}_{0,N}=\xi^{y}_{N}$.
In addition, define
\[
\mathcal{A}^{xy}_{N}:=\left\{ \eta\in\Omega_{N}:\eta_{x}+\eta_{y}=N\right\} =\left\{ \xi^{xy}_{i,N-i}:0\le i\le N\right\} .
\]

\subsection{\label{sec3.1}Resolvent condition}

Here, we prove Proposition \ref{p:inc1}. Fix $\lambda>0$, ${\bf g}\in C_{0}(\mathscr{V}_{\mathfrak{d}})$,
and $G_{N}\in C_{0}(\Omega_{N})$ as in Proposition \ref{p:inc1},
and pick $x\in\mathscr{V}$. Define a test measure $\mu_{N,x}$ on
$\Omega_{N}$ as follows (recall \eqref{eq:loc-fin}):
\begin{itemize}
\item $\mu_{N,x}(\xi^{xy}_{i,N-i}):=\frac{\Gamma(\alpha_{x}\epsilon_{N}+i)}{\Gamma(\alpha_{x}\epsilon_{N})i!}\frac{\Gamma(\alpha_{y}\epsilon_{N}+N-i)}{\Gamma(\alpha_{y}\epsilon_{N})(N-i)!}$
for $y\in\mathscr{N}_{x}$ and $0\le i\le N$;
\item $\mu_{N,x}(\eta):=0$ for all $\eta\notin\bigcup_{y\in\mathscr{N}_{x}}\mathcal{A}^{xy}_{N}$.
\end{itemize}
Above, $\Gamma(\cdot)$ denotes the usual gamma function. The definition
is consistent since $\mu_{N,x}(\xi^{x}_{N})=\mu_{N,x}(\xi^{xy}_{N,0})=\frac{\Gamma(\alpha_{x}\epsilon_{N}+N)}{\Gamma(\alpha_{x}\epsilon_{N})N!}$
for any $y\in\mathscr{N}_{x}$. The following estimate is useful:
for $i\ge1$ and $\alpha>0$,
\begin{equation}
\frac{\alpha}{i}\le\frac{\Gamma(\alpha+i)}{\Gamma(\alpha)i!}=\frac{\alpha}{i}\prod^{i-1}_{j=1}\left(1+\frac{\alpha}{j}\right)\le\frac{\alpha}{i}e^{\sum^{i-1}_{j=1}\frac{\alpha}{j}}\le\frac{\alpha}{i}e^{\alpha(1+\log i)}.\label{eq:G/G}
\end{equation}
In addition, define a test function $h_{N,x}:\Omega_{N}\to[0,1]$
as follows:
\begin{itemize}
\item $h_{N,x}(\xi^{xy}_{i,N-i}):=\frac{i}{N}$ for $y\in\mathscr{N}_{x}$
and $0\le i\le N$;
\item $h_{N,x}(\eta):=0$ for all $\eta\notin\bigcup_{y\in\mathscr{N}_{x}}\mathcal{A}^{xy}_{N}$.
\end{itemize}
It is again consistent since $h_{N,x}(\xi^{x}_{N})=h_{N,x}(\xi^{xy}_{N,0})=1$
for any $y\in\mathscr{N}_{x}$. Since $(\lambda-\mathscr{L}_{N})F_{N}=G_{N}$,
we obtain that
\begin{equation}
\sum_{\eta\in\Omega_{N}}\mu_{N,x}(\eta)h_{N,x}(\eta)\left(\lambda F_{N}(\eta)-\mathscr{L}_{N}F_{N}(\eta)\right)=\sum_{\eta\in\Omega_{N}}\mu_{N,x}(\eta)h_{N,x}(\eta)G_{N}(\eta).\label{eq:lFG}
\end{equation}
Let us abbreviate $\mu:=\mu_{N,x}$ and $h:=h_{N,x}$. We start with
the right-hand side of \eqref{eq:lFG}. Since $\mu$ vanishes outside
$\bigcup_{y\in\mathcal{N}_{x}}\mathcal{A}^{xy}_{N}$ and $h$ vanishes
at $\xi^{y}_{N}$, $y\in\mathscr{N}_{x}$, the right-hand side becomes
\begin{equation}
\mu(\xi^{x}_{N})h(\xi^{x}_{N})G_{N}(\xi^{x}_{N})+\sum_{y\in\mathscr{N}_{x}}\sum^{N-1}_{i=1}\mu(\xi^{xy}_{i,N-i})h(\xi^{xy}_{i,N-i})G_{N}(\xi^{xy}_{i,N-i}).\label{eq:inc1-1}
\end{equation}
The first term in \eqref{eq:inc1-1} equals $\frac{\Gamma(\alpha_{x}\epsilon_{N}+N)}{\Gamma(\alpha_{x}\epsilon_{N})N!}\,{\bf g}(x)$
by definition. Since $0\le h\le1$, by \eqref{eq:loc-fin}, the absolute
value of the second term in \eqref{eq:inc1-1} is bounded by
\begin{equation}
\mathfrak{n}_{\mathscr{G}}\sup_{N\ge1}\|G_{N}\|_{\infty}\sup_{y\in\mathscr{N}_{x}}\sum^{N-1}_{i=1}\frac{\Gamma(\alpha_{x}\epsilon_{N}+i)}{\Gamma(\alpha_{x}\epsilon_{N})i!}\,\frac{\Gamma(\alpha_{y}\epsilon_{N}+N-i)}{\Gamma(\alpha_{y}\epsilon_{N})(N-i)!}.\label{eq:inc1-2}
\end{equation}
By \eqref{eq:G/G}, the summand in \eqref{eq:inc1-2} is bounded as
\[
\frac{\Gamma(\alpha_{x}\epsilon_{N}+i)}{\Gamma(\alpha_{x}\epsilon_{N})i!}\,\frac{\Gamma(\alpha_{y}\epsilon_{N}+N-i)}{\Gamma(\alpha_{y}\epsilon_{N})(N-i)!}\le\frac{\alpha^{2}_{\sup}\epsilon^{2}_{N}}{i(N-i)}\,e^{2\alpha_{\sup}\epsilon_{N}(1+\log N)}.
\]
Thus, via \eqref{eq:eN-cond}, the absolute value of the second term
in \eqref{eq:inc1-1} is bounded by
\[
\mathfrak{n}_{\mathscr{G}}\sup_{N\ge1}\|G_{N}\|_{\infty}\,e^{2\alpha_{\sup}\epsilon_{N}(1+\log N)}\,\sum^{N-1}_{i=1}\frac{\alpha^{2}_{\sup}\epsilon^{2}_{N}}{i(N-i)}\le C_{\mathscr{G},G}\,\frac{\epsilon^{2}_{N}\log N}{N},
\]
where $C_{\mathscr{G},G}$ is a positive constant that depends on
the graph structure $\mathscr{G}$ and $\sup_{N\ge1}\|G_{N}\|_{\infty}$
but does not depend on $N\ge1$. Since $\frac{\Gamma(\alpha_{x}\epsilon_{N}+N)}{\Gamma(\alpha_{x}\epsilon_{N})N!}\ge\frac{\alpha_{x}\epsilon_{N}}{N}$
by \eqref{eq:G/G} and $\lim_{N\to\infty}\epsilon_{N}\log N=0$ by
\eqref{eq:eN-cond}, we may summarize the above discussions as
\begin{equation}
\sum_{\eta\in\Omega_{N}}\mu(\eta)h(\eta)G_{N}(\eta)=\frac{\Gamma(\alpha_{x}\epsilon_{N}+N)}{\Gamma(\alpha_{x}\epsilon_{N})N!}\left({\bf g}(x)+o_{N}(1)\right).\label{eq:lFG-2}
\end{equation}
For the first term in the left-hand side of \eqref{eq:lFG}, by the
same computations, via \eqref{eq:FG},
\begin{equation}
\lambda\sum_{\eta\in\Omega_{N}}\mu(\eta)h(\eta)F_{N}(\eta)=\frac{\Gamma(\alpha_{x}\epsilon_{N}+N)}{\Gamma(\alpha_{x}\epsilon_{N})N!}\left(\lambda F_{N}(\xi^{x}_{N})+o_{N}(1)\right).\label{eq:lFG-3}
\end{equation}

Finally, let us compute the second term in the left-hand side of \eqref{eq:lFG},
which is
\begin{equation}
\sum_{\eta\in\Omega_{N}}\mu(\eta)h(\eta)\sum_{z,w\in\mathscr{V}}c_{zw}\eta_{z}\left(\alpha_{w}+\epsilon^{-1}_{N}\eta_{w}\right)\left(F_{N}(\eta)-F_{N}(\eta-\delta_{z}+\delta_{w})\right).\label{eq:lFG-3.2}
\end{equation}
We may restrict the first summation for $\eta=\xi^{x}_{N}$ or $\eta=\xi^{xy}_{i,N-i}$
where $y\in\mathscr{N}_{x}$ and $1\le i\le N-1$, since otherwise
it vanishes by the definition of $h$. Thus, it becomes
\begin{equation}
\begin{aligned}\mu(\xi^{x}_{N})h(\xi^{x}_{N})\sum_{y\in\mathscr{N}_{x}} & c_{xy}N\alpha_{y}\left(F_{N}(\xi^{x}_{N})-F_{N}(\xi^{xy}_{N-1,1})\right)\\
+\sum_{y\in\mathscr{N}_{x}}\sum^{N-1}_{i=1}\mu(\xi^{xy}_{i,N-i}) & h(\xi^{xy}_{i,N-i})\,\Bigg(c_{xy}i\left(\alpha_{y}+\epsilon^{-1}_{N}(N-i)\right)\left(F_{N}(\xi^{xy}_{i,N-i})-F_{N}(\xi^{xy}_{i-1,N-i+1})\right)\\
 & +c_{xy}(N-i)\left(\alpha_{x}+\epsilon^{-1}_{N}i\right)\left(F_{N}(\xi^{xy}_{i,N-i})-F_{N}(\xi^{xy}_{i+1,N-i-1})\right)\\
 & +\sum_{z\in\mathscr{N}_{x}\setminus\{y\}}c_{xz}i\alpha_{z}\left(F_{N}(\xi^{xy}_{i,N-i})-F_{N}(\xi^{xy}_{i,N-i}-\delta_{x}+\delta_{z})\right)\\
 & +\sum_{z\in\mathscr{N}_{y}\setminus\{x\}}c_{yz}(N-i)\alpha_{z}\left(F_{N}(\xi^{xy}_{i,N-i})-F_{N}(\xi^{xy}_{i,N-i}-\delta_{y}+\delta_{z})\right)\Bigg).
\end{aligned}
\label{eq:lFG-3.5}
\end{equation}
First, we calculate the last two lines in \eqref{eq:lFG-3.5}. By
\eqref{eq:G/G} and \eqref{eq:FG}, we have
\[
\left|\mu(\xi^{xy}_{i,N-i})h(\xi^{xy}_{i,N-i})\sum_{z\in\mathscr{N}_{x}\setminus\{y\}}c_{xz}i\alpha_{z}\left(F_{N}(\xi^{xy}_{i,N-i})-F_{N}(\xi^{xy}_{i,N-i}-\delta_{x}+\delta_{z})\right)\right|\le\frac{C_{\mathscr{G},G}\,\epsilon^{2}_{N}}{N-i},
\]
and
\[
\left|\mu(\xi^{xy}_{i,N-i})h(\xi^{xy}_{i,N-i})\sum_{z\in\mathscr{N}_{y}\setminus\{x\}}c_{yz}(N-i)\alpha_{z}\left(F_{N}(\xi^{xy}_{i,N-i})-F_{N}(\xi^{xy}_{i,N-i}-\delta_{y}+\delta_{z})\right)\right|\le\frac{C_{\mathscr{G},G}\,\epsilon^{2}_{N}}{i}.
\]
Thus, by \eqref{eq:eN-cond}, \eqref{eq:lFG-3.2} is equal to
\begin{equation}
\begin{aligned}\mu(\xi^{x}_{N})h(\xi^{x}_{N})\sum_{y\in\mathscr{N}_{x}} & c_{xy}N\alpha_{y}\left(F_{N}(\xi^{x}_{N})-F_{N}(\xi^{xy}_{N-1,1})\right)\\
+\sum_{y\in\mathscr{N}_{x}}\sum^{N-1}_{i=1}\mu(\xi^{xy}_{i,N-i}) & h(\xi^{xy}_{i,N-i})\,\bigg(c_{xy}i\left(\alpha_{y}+\epsilon^{-1}_{N}(N-i)\right)\left(F_{N}(\xi^{xy}_{i,N-i})-F_{N}(\xi^{xy}_{i-1,N-i+1})\right)\\
+c_{xy} & (N-i)\left(\alpha_{x}+\epsilon^{-1}_{N}i\right)\left(F_{N}(\xi^{xy}_{i,N-i})-F_{N}(\xi^{xy}_{i+1,N-i-1})\right)\bigg)+o_{N}\left(\frac{\epsilon_{N}}{N}\right).
\end{aligned}
\label{eq:lFG-4}
\end{equation}
Let us calculate the coefficient of each $F_{N}(\xi^{xy}_{i,N-i})$
in \eqref{eq:lFG-4}. For $i=N$, the coefficient of $F_{N}(\xi^{x}_{N})$
becomes
\begin{equation}
\begin{aligned}\sum_{y\in\mathscr{N}_{x}}\bigg(\mu(\xi^{x}_{N})h(\xi^{x}_{N})c_{xy}N\alpha_{y} & -\mu(\xi^{xy}_{N-1,1})h(\xi^{xy}_{N-1,1})c_{xy}\left(\alpha_{x}+\epsilon^{-1}_{N}(N-1)\right)\bigg)\\
 & =\sum_{y\in\mathscr{N}_{x}}\frac{\Gamma(\alpha_{x}\epsilon_{N}+N)}{\Gamma(\alpha_{x}\epsilon_{N})N!}c_{xy}\alpha_{y}.
\end{aligned}
\label{eq:lFG-5}
\end{equation}
For $i=0$, the coefficient of $F_{N}(\xi^{y}_{N})$ becomes
\begin{equation}
\mu(\xi^{xy}_{1,N-1})h(\xi^{xy}_{1,N-1})\,c_{xy}\left(\alpha_{y}+\epsilon^{-1}_{N}(N-1)\right)-\mu(\xi^{x}_{N})h(\xi^{x}_{N})c_{xy}N\alpha_{y}=-\frac{\Gamma(\alpha_{x}\epsilon_{N}+N)}{\Gamma(\alpha_{x}\epsilon_{N})N!}c_{xy}\alpha_{y}.\label{eq:lFG-6}
\end{equation}
For $1\le i\le N-1$, the coefficient of $F_{N}(\xi^{xy}_{i,N-i})$
becomes
\begin{equation}
\begin{aligned}\mu(\xi^{xy}_{i,N-i}) & h(\xi^{xy}_{i,N-i})c_{xy}i\left(\alpha_{y}+\epsilon^{-1}_{N}(N-i)\right)+\mu(\xi^{xy}_{i,N-i})h(\xi^{xy}_{i,N-i})c_{xy}(N-i)\left(\alpha_{x}+\epsilon^{-1}_{N}i\right)\\
 & -\mu(\xi^{xy}_{i+1,N-i-1})h(\xi^{xy}_{i+1,N-i-1})c_{xy}(i+1)\left(\alpha_{y}+\epsilon^{-1}_{N}(N-i-1)\right)\\
 & -\mu(\xi^{xy}_{i-1,N-i+1})h(\xi^{xy}_{i-1,N-i+1})c_{xy}(N-i+1)\left(\alpha_{x}+\epsilon^{-1}_{N}(i-1)\right).
\end{aligned}
\label{eq:lFG-7}
\end{equation}
Merging the first/fourth terms, the second/third terms and simplifying,
the absolute value of \eqref{eq:lFG-7} equals
\begin{equation}
\left|\frac{c_{xy}\epsilon^{-1}_{N}}{N}\,\frac{\Gamma(\alpha_{x}\epsilon_{N}+i)}{\Gamma(\alpha_{x}\epsilon_{N})i!}\,\frac{\Gamma(\alpha_{y}\epsilon_{N}+N-i)}{\Gamma(\alpha_{y}\epsilon_{N})(N-i)!}\left(i(\alpha_{y}\epsilon_{N}+N-i)-(\alpha_{x}\epsilon_{N}+i)(N-i)\right)\right|\le\frac{C_{\mathscr{G}}\epsilon^{2}_{N}}{i(N-i)}.\label{eq:lFG-8}
\end{equation}
Gathering \eqref{eq:lFG-5}, \eqref{eq:lFG-6}, \eqref{eq:lFG-8}
and applying $\sum^{N-1}_{i=1}\frac{1}{i(N-i)}\le\frac{C\log N}{N}$,
\eqref{eq:lFG-4} can be written as\footnote{Here, $f_{N}=O_{N}(g_{N})$ indicates $|f_{N}|\le Cg_{N}$ where $C$
is independent of $N$.}
\begin{equation}
\begin{aligned}\sum_{y\in\mathscr{N}_{x}}\frac{\Gamma(\alpha_{x}\epsilon_{N}+N)}{\Gamma(\alpha_{x}\epsilon_{N})N!} & c_{xy}\alpha_{y}\left(F_{N}(\xi^{x}_{N})-F_{N}(\xi^{y}_{N})\right)+O_{N}\left(\frac{\epsilon^{2}_{N}\log N}{N}\right)+o_{N}\left(\frac{\epsilon_{N}}{N}\right)\\
 & =\frac{\Gamma(\alpha_{x}\epsilon_{N}+N)}{\Gamma(\alpha_{x}\epsilon_{N})N!}\,(-\mathfrak{L}f_{N}(x)+o_{N}(1)),
\end{aligned}
\label{eq:lFG-9}
\end{equation}
where $f_{N}(z):=F_{N}(\xi^{z}_{N})$. Gathering \eqref{eq:lFG},
\eqref{eq:lFG-2}, \eqref{eq:lFG-3}, and \eqref{eq:lFG-9}, we obtain
\[
\lambda f_{N}(x)-\mathfrak{L}f_{N}(x)={\bf g}(x)+o_{N}(1)\quad\text{for all}\quad x\in\mathscr{V},
\]
where the error $o_{N}(1)$ is uniform over all $x\in\mathscr{V}$.
Appealing to \eqref{eq:pr}, this implies that
\[
F_{N}(\xi^{x}_{N})=f_{N}(x)={\rm E}^{{\bf Q}_{x}}\left[\int^{\infty}_{0}e^{-\lambda t}\,({\bf g}({\bf y}(t))+o_{N}(1))\,{\rm d}t\right]={\bf f}(x)+o_{N}(1).
\]
This proves Proposition \ref{p:inc1}.

\subsection{\label{sec3.2}Compactness condition}

In this subsection, we prove Proposition \ref{p:inc2}. If $\mathscr{G}$
is finite, then $\mathcal{K}_{N}\equiv\Omega_{N}$ suffices and there
is nothing to prove. From now on, assume that $\mathscr{G}$ is an
infinite graph.

\begin{figure}
\begin{tikzpicture}
\foreach \i in {0,...,5} {
\fill[red!30!white,rotate around={60*\i:(0,0)}] (3/2,{3*sqrt(3)/2}) circle (0.15);
\fill[red!30!white,rotate around={60*\i:(0,0)}] (-3/2,{3*sqrt(3)/2}) circle (0.15);
\fill[red!30!white,rotate around={60*\i:(0,0)}] (3/2,{3*sqrt(3)/2+0.15}) rectangle (-3/2,{3*sqrt(3)/2-0.15});
\fill[blue!30!white,rotate around={60*\i:(0,0)}] (1,{sqrt(3)}) circle (0.15);
\fill[blue!30!white,rotate around={60*\i:(0,0)}] (-1,{sqrt(3)}) circle (0.15);
\fill[blue!30!white,rotate around={60*\i:(0,0)}] (1,{sqrt(3)+0.15}) rectangle (-1,{sqrt(3)-0.15});
\fill[teal!30!white,rotate around={60*\i:(0,0)}] (1/2,{sqrt(3)/2}) circle (0.15);
\fill[teal!30!white,rotate around={60*\i:(0,0)}] (-1/2,{sqrt(3)/2}) circle (0.15);
\fill[teal!30!white,rotate around={60*\i:(0,0)}] (1/2,{sqrt(3)/2+0.15}) rectangle (-1/2,{sqrt(3)/2-0.15});
}
\draw (3,0)--(3/2,{3*sqrt(3)/2})--(-3/2,{3*sqrt(3)/2})--(-3,0)--(-3/2,{-3*sqrt(3)/2})--(3/2,{-3*sqrt(3)/2})--(3,0);
\draw (3,0)--(-3,0); \draw (3/2,{3*sqrt(3)/2})--(-3/2,{-3*sqrt(3)/2}); \draw (-3/2,{3*sqrt(3)/2})--(3/2,{-3*sqrt(3)/2});
\draw (1/2,{3*sqrt(3)/2})--(-2,{-sqrt(3)})--(2,{-sqrt(3)})--(-1/2,{3*sqrt(3)/2})--(-5/2,{-sqrt(3)/2})--(5/2,{-sqrt(3)/2})--(1/2,{3*sqrt(3)/2});
\draw (2,{sqrt(3)})--(-2,{sqrt(3)})--(1/2,{-3*sqrt(3)/2})--(5/2,{sqrt(3)/2})--(-5/2,{sqrt(3)/2})--(-1/2,{-3*sqrt(3)/2})--(2,{sqrt(3)});
\fill (0,0) circle (0.07);
\draw (0,-0.1) node[below]{$x$};
\draw (1/3,{-sqrt(3)/2-0.1}) node[below]{\color{teal} $\mathscr{N}_{x,1}$};
\draw (5/6,{-sqrt(3)-0.1}) node[below]{\color{blue} $\mathscr{N}_{x,2}$};
\draw (4/3,{-3*sqrt(3)/2-0.1}) node[below]{\color{red} $\mathscr{N}_{x,3}$};

\begin{scope}[shift={(7,0)}]
\fill[blue!15!white] (3,0)--(3/2,{3*sqrt(3)/2})--(-3/2,{3*sqrt(3)/2})--(-3,0)--(-3/2,{-3*sqrt(3)/2})--(3/2,{-3*sqrt(3)/2})--(3,0);
\foreach \i in {0,...,5} {
\fill[red!15!white,rotate around={60*\i:(0,0)}] (3/2,{3*sqrt(3)/2}) circle (0.15);
\fill[red!15!white,rotate around={60*\i:(0,0)}] (-3/2,{3*sqrt(3)/2}) circle (0.15);
\fill[red!15!white,rotate around={60*\i:(0,0)}] (3/2,{3*sqrt(3)/2+0.15}) rectangle (-3/2,{3*sqrt(3)/2-0.15});
}
%\draw[densely dotted] (3,0)--(3/2,{3*sqrt(3)/2})--(-3/2,{3*sqrt(3)/2})--(-3,0)--(-3/2,{-3*sqrt(3)/2})--(3/2,{-3*sqrt(3)/2})--(3,0);
\draw[densely dotted] (3,0)--(-3,0); \draw[densely dotted] (3/2,{3*sqrt(3)/2})--(-3/2,{-3*sqrt(3)/2}); \draw[densely dotted] (-3/2,{3*sqrt(3)/2})--(3/2,{-3*sqrt(3)/2});
\draw[densely dotted] (1/2,{3*sqrt(3)/2})--(-2,{-sqrt(3)})--(2,{-sqrt(3)})--(-1/2,{3*sqrt(3)/2})--(-5/2,{-sqrt(3)/2})--(5/2,{-sqrt(3)/2})--(1/2,{3*sqrt(3)/2});
\draw[densely dotted] (2,{sqrt(3)})--(-2,{sqrt(3)})--(1/2,{-3*sqrt(3)/2})--(5/2,{sqrt(3)/2})--(-5/2,{sqrt(3)/2})--(-1/2,{-3*sqrt(3)/2})--(2,{sqrt(3)});
\foreach \i in {-3,...,3} { \fill (\i,0) circle (0.07); }
\foreach \i in {-2.5,...,2.5} { \fill (\i,{-sqrt(3)/2}) circle (0.07); \fill (\i,{sqrt(3)/2}) circle (0.07); }
\foreach \i in {-2,...,2} { \fill (\i,{-sqrt(3)}) circle (0.07); \fill (\i,{sqrt(3)}) circle (0.07); }
\foreach \i in {-1.5,...,1.5} { \fill (\i,{-3*sqrt(3)/2}) circle (0.07); \fill (\i,{3*sqrt(3)/2}) circle (0.07); }
\foreach \j in {0,...,2} {
\foreach \i in {-3,...,2} { \draw[blue,rotate around={60*\j:(0,0)}] (\i+1/3-0.05,0-0.05)--(\i+1/3+0.05,0+0.05); \draw[blue,rotate around={60*\j:(0,0)}] (\i+1/3-0.05,0+0.05)--(\i+1/3+0.05,0-0.05); }
\foreach \i in {-3,...,2} { \draw[blue,rotate around={60*\j:(0,0)}] (\i+2/3-0.05,0-0.05)--(\i+2/3+0.05,0+0.05); \draw[blue,rotate around={60*\j:(0,0)}] (\i+2/3-0.05,0+0.05)--(\i+2/3+0.05,0-0.05); }
\foreach \i in {-2.5,...,1.5} { \draw[blue,rotate around={60*\j:(0,0)}] (\i+1/3-0.05,{-sqrt(3)/2-0.05})--(\i+1/3+0.05,{-sqrt(3)/2+0.05}); \draw[blue,rotate around={60*\j:(0,0)}] (\i+1/3-0.05,{-sqrt(3)/2+0.05})--(\i+1/3+0.05,{-sqrt(3)/2-0.05}); }
\foreach \i in {-2.5,...,1.5} { \draw[blue,rotate around={60*\j:(0,0)}] (\i+2/3-0.05,{-sqrt(3)/2-0.05})--(\i+2/3+0.05,{-sqrt(3)/2+0.05}); \draw[blue,rotate around={60*\j:(0,0)}] (\i+2/3-0.05,{-sqrt(3)/2+0.05})--(\i+2/3+0.05,{-sqrt(3)/2-0.05}); }
\foreach \i in {-2.5,...,1.5} { \draw[blue,rotate around={60*\j:(0,0)}] (\i+1/3-0.05,{sqrt(3)/2-0.05})--(\i+1/3+0.05,{sqrt(3)/2+0.05}); \draw[blue,rotate around={60*\j:(0,0)}] (\i+1/3-0.05,{sqrt(3)/2+0.05})--(\i+1/3+0.05,{sqrt(3)/2-0.05}); }
\foreach \i in {-2.5,...,1.5} { \draw[blue,rotate around={60*\j:(0,0)}] (\i+2/3-0.05,{sqrt(3)/2-0.05})--(\i+2/3+0.05,{sqrt(3)/2+0.05}); \draw[blue,rotate around={60*\j:(0,0)}] (\i+2/3-0.05,{sqrt(3)/2+0.05})--(\i+2/3+0.05,{sqrt(3)/2-0.05}); }
\foreach \i in {-2,...,1} { \draw[blue,rotate around={60*\j:(0,0)}] (\i+1/3-0.05,{-sqrt(3)-0.05})--(\i+1/3+0.05,{-sqrt(3)+0.05}); \draw[blue,rotate around={60*\j:(0,0)}] (\i+1/3-0.05,{-sqrt(3)+0.05})--(\i+1/3+0.05,{-sqrt(3)-0.05}); }
\foreach \i in {-2,...,1} { \draw[blue,rotate around={60*\j:(0,0)}] (\i+2/3-0.05,{-sqrt(3)-0.05})--(\i+2/3+0.05,{-sqrt(3)+0.05}); \draw[blue,rotate around={60*\j:(0,0)}] (\i+2/3-0.05,{-sqrt(3)+0.05})--(\i+2/3+0.05,{-sqrt(3)-0.05}); }
\foreach \i in {-2,...,1} { \draw[blue,rotate around={60*\j:(0,0)}] (\i+1/3-0.05,{sqrt(3)-0.05})--(\i+1/3+0.05,{sqrt(3)+0.05}); \draw[blue,rotate around={60*\j:(0,0)}] (\i+1/3-0.05,{sqrt(3)+0.05})--(\i+1/3+0.05,{sqrt(3)-0.05}); }
\foreach \i in {-2,...,1} { \draw[blue,rotate around={60*\j:(0,0)}] (\i+2/3-0.05,{sqrt(3)-0.05})--(\i+2/3+0.05,{sqrt(3)+0.05}); \draw[blue,rotate around={60*\j:(0,0)}] (\i+2/3-0.05,{sqrt(3)+0.05})--(\i+2/3+0.05,{sqrt(3)-0.05}); }
}
\draw (0,-0.1) node[below]{$\xi_N^x$};
\draw (4/3,{-3*sqrt(3)/2-0.1}) node[below]{\color{red} $\mathcal{X}_{N,3}$};
\draw (11/6+0.3,{-sqrt(3)-0.1}) node[below]{\color{blue} $\mathcal{K}_{N,3}$};
\foreach \i in {0,...,5} {
\fill[red,rotate around={60*\i:(0,0)}] (3,0) circle (0.07);
\fill[red,rotate around={60*\i:(0,0)}] (5/2,{sqrt(3)/2}) circle (0.07);
\fill[red,rotate around={60*\i:(0,0)}] (2,{sqrt(3)}) circle (0.07);
}
\end{scope}
\end{tikzpicture}\caption{\label{fig3.1}(Left) Collection $\widehat{\mathscr{N}}_{x,3}=\{x\}\cup\mathscr{N}_{x,1}\cup\mathscr{N}_{x,2}\cup\mathscr{N}_{x,3}$
in the triangular lattice $\mathscr{G}$. Teal, blue and red regions
denote $\mathscr{N}_{x,1}$, $\mathscr{N}_{x,2}$, and $\mathscr{N}_{x,3}$,
respectively. (Right) Corresponding local configuration space near
$\xi^{x}_{N}$. The circles denote the configurations $\xi^{y}_{N}$
for $y\in\widehat{\mathscr{N}}_{x,3}$. The crosses denote the elements
$\xi^{yz}_{i,N-i}$ for $1\le i\le N-1$ and $y,z\in\widehat{\mathscr{N}}_{x,3}$
with $\{y,z\}\in\mathscr{E}$. The blue region represents $\mathcal{K}_{N,3}$
and the red region represents one boundary $\mathcal{X}_{N,3}$. The
other boundary $\mathcal{Y}_{N,3}$, which is not illustrated, can
be reached from the crosses.}
\end{figure}
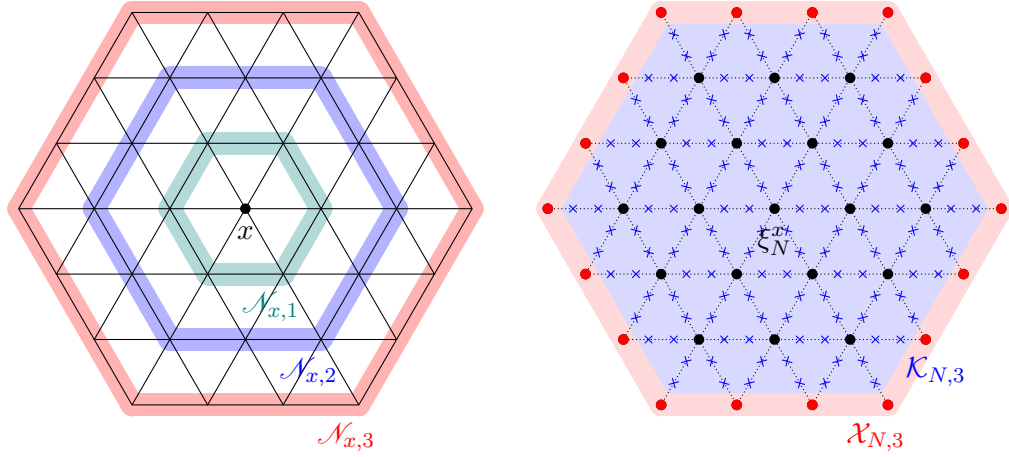

Fix $x\in\mathscr{V}$ and $T,\epsilon>0$. For each integer $m\ge0$,
define (see Figure \ref{fig3.1}-left)
\[
\mathscr{N}_{x,m}:=\left\{ y\in\mathscr{V}:d_{\mathscr{G}}(x,y)=m\right\} ,
\]
where $d_{\mathscr{G}}$ denotes the graph distance in $\mathscr{G}$.
In particular, $\mathscr{N}_{x,0}=\{x\}$ and $\mathscr{N}_{x,1}=\mathscr{N}_{x}$
(cf. \eqref{eq:loc-fin}). In addition, define
\[
\widehat{\mathscr{N}}_{x,m}:=\bigcup^{m}_{j=0}\mathscr{N}_{x,j}.
\]
Then, define (see Figure \ref{fig3.1}-right)
\[
\mathcal{K}_{N,m}:=\bigcup_{\substack{y,z\in\widehat{\mathscr{N}}_{x,m}\\
\{y,z\}\in\mathscr{E}
}
}\mathcal{A}^{yz}_{N}\,\Big\backslash\,\bigcup_{y\in\mathscr{N}_{x,m}}\mathcal{E}^{y}_{N}.
\]
Write
\[
\mathcal{X}_{N,m}:=\bigcup_{y\in\mathscr{N}_{x,m}}\mathcal{E}^{y}_{N}
\]
and
\[
\mathcal{Y}_{N,m}:=\bigcup_{\substack{y,z\in\widehat{\mathscr{N}}_{x,m}\\
\{y,z\}\in\mathscr{E}
}
}\bigcup^{N-1}_{i=1}\left\{ \xi^{yz}_{i,N-i}-\delta_{w}+\delta_{u}:w\in\{y,z\},\enskip u\notin\{y,z\},\enskip\{w,u\}\in\mathscr{E}\right\} .
\]
It is clear that $\partial\mathcal{K}_{N,m}=\mathcal{X}_{N,m}\cup\mathcal{Y}_{N,m}$,
where $\partial\mathcal{A}:=\{\eta\notin\mathcal{A}:r_{N}(\eta,\xi)>0\enskip\text{for some}\enskip\xi\in\mathcal{A}\}$.
First, we prove that the process escapes $\mathcal{K}_{N,m}$ via
$\mathcal{X}_{N,m}$ with probability asymptotically one.

We need some potential theory notation. For $\eta\in\Omega_{N}$,
define
\[
\mu_{N}(\eta):=\prod_{x\in\mathscr{V}}\frac{\Gamma(\alpha_{x}\epsilon_{N}+\eta_{x})}{\Gamma(\alpha_{x}\epsilon_{N})\eta_{x}!}.
\]
Then, $\mu_{N}$ is a stationary measure of $\{\eta_{N}(t)\}_{t\ge0}$
with detailed balance conditions:
\[
\mu_{N}(\eta)c_{xy}\eta_{x}\left(\alpha_{y}+\epsilon^{-1}_{N}\eta_{y}\right)=\mu_{N}(\eta-\delta_{x}+\delta_{y})c_{yx}(\eta_{y}+1)\left(\alpha_{x}+\epsilon^{-1}_{N}(\eta_{x}-1)\right).
\]
Define $\Omega_{N,m}:=\mathcal{K}_{N,m}\cup\mathcal{X}_{N,m}\cup\mathcal{Y}_{N,m}$.
Denote by
\[
\mathcal{D}_{N,m}(f):=\frac{1}{2}\sum_{\eta,\xi\in\Omega_{N,m}}\mu_{N}(\eta)r_{N}(\eta,\xi)\left(f(\xi)-f(\eta)\right)^{2}
\]
the \emph{Dirichlet form} of the process restricted to $\Omega_{N,m}$.
For disjoint $\mathcal{A},\mathcal{B}\subset\Omega_{N,m}$, the \emph{equilibrium
potential} between $\mathcal{A}$ and $\mathcal{B}$ is given as
\[
h_{\mathcal{A},\mathcal{B}}(\eta):=\mathbb{Q}^{N,m}_{\eta}\left[\mathcal{H}_{\mathcal{A}}<\mathcal{H}_{\mathcal{B}}\right],
\]
where $\mathbb{Q}^{N,m}_{\eta}$ denotes the law of the restricted
process in $\Omega_{N,m}$ starting from $\eta$. Then, the \emph{capacity}
between $\mathcal{A}$ and $\mathcal{B}$ is defined as
\[
{\rm cap}_{N,m}(\mathcal{A},\mathcal{B}):=\mathcal{D}_{N,m}(h_{\mathcal{A},\mathcal{B}}).
\]
A function $\phi:\Omega_{N,m}\times\Omega_{N,m}\to\mathbb{R}$ is
a \emph{flow} if $\phi(\eta,\xi)=-\phi(\xi,\eta)$ and $\phi(\eta,\xi)=0$
if $r_{N}(\eta,\xi)=0$. Given a flow $\phi$, its flow norm is defined
as
\[
\|\phi\|^{2}_{N,m}:=\frac{1}{2}\sum_{\substack{\eta,\xi\in\Omega_{N,m}\\
r_{N}(\eta,\xi)>0
}
}\frac{\phi(\eta,\xi)^{2}}{\mu_{N}(\eta)r_{N}(\eta,\xi)}.
\]
A flow $\phi$ is a \emph{unit flow} from $\mathcal{A}$ to $\mathcal{B}$
if $\sum_{\eta\in\mathcal{A}}\sum_{\xi\in\Omega_{N,m}}\phi(\eta,\xi)=1$,
$\sum_{\eta\in\mathcal{B}}\sum_{\xi\in\Omega_{N,m}}\phi(\eta,\xi)=-1$,
and
\[
\sum_{\xi\in\Omega_{N,m}}\phi(\eta,\xi)=0\quad\text{for all}\quad\eta\notin\mathcal{A}\cup\mathcal{B}.
\]

\begin{lem}
\label{lem1}We have
\[
\lim_{N\to\infty}\mathbb{Q}^{N}_{\xi^{x}_{N}}\left[\mathcal{H}_{\mathcal{X}_{N,m}}>\mathcal{H}_{\mathcal{Y}_{N,m}}\right]=0.
\]
\end{lem}

\begin{proof}
To calculate the left-hand side, we may restrict the original inclusion
process to $\Omega_{N,m}:=\mathcal{K}_{N,m}\cup\mathcal{X}_{N,m}\cup\mathcal{Y}_{N,m}$.
Then, a renewal estimate (cf. \cite[Lemma 8.4]{BdH15}) gives
\begin{equation}
\mathbb{Q}^{N}_{\xi^{x}_{N}}\left[\mathcal{H}_{\mathcal{X}_{N,m}}>\mathcal{H}_{\mathcal{Y}_{N,m}}\right]\le\frac{{\rm cap}_{N,m}(\xi^{x}_{N},\mathcal{Y}_{N,m})}{{\rm cap}_{N,m}(\xi^{x}_{N},\mathcal{X}_{N,m})}.\label{eq:lem1-1}
\end{equation}
For the numerator in \eqref{eq:lem1-1}, we define a test object $f:\Omega_{N,m}\to\mathbb{R}$
as $f\equiv1$ on $\mathcal{K}_{N,m}\cup\mathcal{X}_{N,m}$ and $f\equiv0$
on $\mathcal{Y}_{N,m}$. Via \eqref{eq:G/G} and \eqref{eq:eN-cond},
the Dirichlet form is estimated as
\begin{align*}
 & \mathcal{D}_{N,m}(f)=\sum_{\eta\in\mathcal{K}_{N,m}}\sum_{\xi\in\mathcal{Y}_{N,m}}\mu_{N}(\eta)r_{N}(\eta,\xi)\\
 & =\frac{1}{2}\sum_{\substack{y,z\in\widehat{\mathscr{N}}_{x,m}\\
\{y,z\}\in\mathscr{E}
}
}\sum^{N-1}_{i=1}\frac{\Gamma(\alpha_{y}\epsilon_{N}+i)}{\Gamma(\alpha_{y}\epsilon_{N})i!}\frac{\Gamma(\alpha_{z}\epsilon_{N}+N-i)}{\Gamma(\alpha_{z}\epsilon_{N})(N-i)!}\left(\sum_{\substack{u\in\mathscr{V}\setminus\{z\}\\
\{y,u\}\in\mathscr{E}
}
}c_{yu}i\alpha_{u}+\sum_{\substack{u\in\mathscr{V}\setminus\{y\}\\
\{z,u\}\in\mathscr{E}
}
}c_{zu}(N-i)\alpha_{u}\right)\\
 & \le\alpha^{3}_{\sup}\mathfrak{n}_{\mathscr{G}}c_{\sup}\epsilon^{2}_{N}\sum_{\substack{y,z\in\widehat{\mathscr{N}}_{x,m}\\
\{y,z\}\in\mathscr{E}
}
}\sum^{N-1}_{i=1}\frac{N}{i(N-i)}e^{2\alpha_{\sup}\epsilon_{N}(1+\log N)}\le C_{\mathscr{G}}\epsilon^{2}_{N}\log N.
\end{align*}
Thus, the Dirichlet principle (cf. \cite[Theorem 7.33]{BdH15}) implies
that
\begin{equation}
{\rm cap}_{N,m}(\xi^{x}_{N},\mathcal{Y}_{N,m})\le\mathcal{D}_{N,m}(f)\le C_{\mathscr{G}}\epsilon^{2}_{N}\log N.\label{eq:lem1-2}
\end{equation}
For the denominator in \eqref{eq:lem1-1}, fix $y\in\mathscr{N}_{x,m}$
and a path $x=x_{0},x_{1},\dots,x_{m}=y$ in $\mathscr{G}$. Define
a flow $\phi:\Omega_{N,m}\times\Omega_{N,m}\to\mathbb{R}$ as
\[
\phi(\xi^{x_{j-1}x_{j}}_{i,N-i},\xi^{x_{j-1}x_{j}}_{i-1,N-i+1}):=1,\quad\phi(\xi^{x_{j-1}x_{j}}_{i-1,N-i+1},\xi^{x_{j-1}x_{j}}_{i,N-i}):=-1
\]
for all $1\le j\le m$ and $1\le i\le N$, and $\phi:=0$ otherwise.
Then, it is clear that $\phi$ is a unit flow from $\xi^{x}_{N}$
to $\mathcal{X}_{N,m}\ni\xi^{y}_{N}$. In addition, by \eqref{eq:G/G},
\begin{align*}
\|\phi\|^{2}_{N,m} & =\sum^{m}_{j=1}\sum^{N}_{i=1}\frac{1}{\frac{\Gamma(\alpha_{x_{j-1}}\epsilon_{N}+i)}{\Gamma(\alpha_{x_{j-1}}\epsilon_{N})i!}\frac{\Gamma(\alpha_{x_{j}}\epsilon_{N}+N-i)}{\Gamma(\alpha_{x_{j}}\epsilon_{N})(N-i)!}c_{x_{j-1}x_{j}}(N-i)(\alpha_{x_{j}}+\epsilon^{-1}_{N}i)}\\
 & \le\sum^{m}_{j=1}\sum^{N}_{i=1}\frac{1}{\alpha_{x_{j-1}}\epsilon_{N}\alpha_{x_{j}}c_{x_{j-1}x_{j}}}\le C_{\mathscr{G}}N\epsilon^{-1}_{N}.
\end{align*}
The Thomson principle (\cite[Theorem 7.37]{BdH15}) then implies that
\begin{equation}
{\rm cap}_{N,m}(\xi^{x}_{N},\mathcal{X}_{N,m})\ge\frac{1}{\|\phi\|^{2}_{N,m}}\ge\frac{\epsilon_{N}}{NC_{\mathscr{G}}}.\label{eq:lem1-3}
\end{equation}
Combining \eqref{eq:lem1-1}, \eqref{eq:lem1-2}, and \eqref{eq:lem1-3},
we conclude that
\[
\mathbb{Q}^{N}_{\xi^{x}_{N}}\left[\mathcal{H}_{\mathcal{X}_{N,m}}>\mathcal{H}_{\mathcal{Y}_{N,m}}\right]\le C^{2}_{\mathscr{G}}\epsilon_{N}N\log N\xrightarrow{N\to\infty}0,
\]
via \eqref{eq:eN-cond}, which was what we wanted.
\end{proof}

Now, to study the restricted dynamics in $\Omega_{N,m}$ conditioned
on the event that $\mathcal{H}_{\mathcal{X}_{N,m}}<\mathcal{H}_{\mathcal{Y}_{N,m}}$,
it suffices to consider the (further) restricted dynamics $\{\widetilde{\eta}_{N}(t)\}_{t\ge0}$
in $\widetilde{\Omega}_{N,m}:=\mathcal{K}_{N,m}\cup\mathcal{X}_{N,m}$.
Let us denote the corresponding law on the path space starting from
$\eta$ as $\widetilde{\mathbb{Q}}^{N,m}_{\eta}$.
\begin{lem}
\label{lem2}For each fixed $T>0$,
\[
\lim_{m\to\infty}\limsup_{N\to\infty}\widetilde{\mathbb{Q}}^{N,m}_{\xi^{x}_{N}}\left[\mathcal{H}_{\mathcal{X}_{N,m}}\le T\right]=0.
\]
\end{lem}

\begin{proof}
Denote by $\widehat{\mathbb{Q}}^{N,m}_{\xi^{x}_{N}}$ the law of the
trace process $\{\widehat{\eta}_{N}(t)\}_{t\ge0}$ in $\mathcal{E}_{N}(\widehat{\mathscr{N}}_{x,m})$
obtained from $\{\widetilde{\eta}_{N}(t)\}_{t\ge0}$. Then, since
the trace process is obtained by freezing the clock if necessary,
\begin{equation}
\widetilde{\mathbb{Q}}^{N,m}_{\xi^{x}_{N}}\left[\mathcal{H}_{\mathcal{X}_{N,m}}\le T\right]\le\widehat{\mathbb{Q}}^{N,m}_{\xi^{x}_{N}}\left[\mathcal{H}_{\mathcal{X}_{N,m}}\le T\right].\label{eq:trace-change}
\end{equation}
For this new process, the jump rate function $\widehat{r}_{N}(\cdot,\cdot)$
can be calculated via \cite[Corollary 6.2]{BL10}:
\begin{align*}
\widehat{r}_{N}(\xi^{y}_{N},\xi^{z}_{N}) & =\sum_{\eta\in\widetilde{\Omega}_{N,m}}r_{N}(\xi^{y}_{N},\eta)\widetilde{\mathbb{Q}}^{N,m}_{\eta}\left[\mathcal{H}_{\xi^{z}_{N}}=\mathcal{H}_{\mathcal{E}_{N}(\widehat{\mathscr{N}}_{x,m})}\right]\\
 & =\begin{cases}
c_{yz}N\alpha_{z}\widetilde{\mathbb{Q}}^{N,m}_{\xi^{yz}_{N-1,1}}\left[\mathcal{H}_{\xi^{z}_{N}}<\mathcal{H}_{\xi^{y}_{N}}\right] & \text{if}\quad\{y,z\}\in\mathscr{E},\\
0 & \text{otherwise}.
\end{cases}
\end{align*}
Since $\mathcal{A}^{yz}_{N}$ is a one-dimensional subspace, the final
probability can be easily calculated as (cf. \cite[Lemma 4.10]{KS21})
\[
\widetilde{\mathbb{Q}}^{N,m}_{\xi^{yz}_{N-1,1}}\left[\mathcal{H}_{\xi^{z}_{N}}<\mathcal{H}_{\xi^{y}_{N}}\right]=\frac{1}{N}+O_{N}\left(d_{N}\log N\right).
\]
Therefore, by \eqref{eq:eN-cond},
\[
\widehat{r}_{N}(\xi^{y}_{N},\xi^{z}_{N})=\begin{cases}
c_{yz}\alpha_{z}(1+o_{N}(1)) & \text{if}\quad\{y,z\}\in\mathscr{E},\\
0 & \text{otherwise}.
\end{cases}
\]
This implies that, for sufficiently large $N$, $\mathcal{H}_{\mathcal{X}_{N,m}}$
is stochastically dominated from below by a random variable $\mathcal{H}_{1}+\cdots+\mathcal{H}_{m}$,
where $\mathcal{H}_{i}$, $i\ge1$, are i.i.d. exponential random
variables with rate $2c_{\sup}\alpha_{\sup}\in(0,\infty)$. Thus,
\[
\limsup_{N\to\infty}\widehat{\mathbb{Q}}^{N,m}_{\xi^{x}_{N}}\left[\mathcal{H}_{\mathcal{X}_{N,m}}\le T\right]\le\mathbb{Q}\left[\mathcal{H}_{1}+\cdots+\mathcal{H}_{m}\le T\right]\xrightarrow{m\to\infty}0.
\]
Via \eqref{eq:trace-change}, this proves the lemma.
\end{proof}

We are ready to prove Proposition \ref{p:inc2}. Via Lemma \ref{lem2},
take $m=m(T,\epsilon)\ge1$ such that
\[
\limsup_{N\to\infty}\widetilde{\mathbb{Q}}^{N,m}_{\xi^{x}_{N}}\left[\mathcal{H}_{\mathcal{X}_{N,m}}\le T\right]<\epsilon,
\]
and let $\mathcal{K}_{N}:=\mathcal{K}_{N,m}$. Then, $\mathcal{K}_{N}$
satisfies Condition $\mathfrak{K}$. Indeed, since
\[
\left\{ y\in\mathscr{V}:\xi^{y}_{N}\in\mathcal{K}_{N}\enspace\text{for some}\enspace N\ge1\right\} =\widehat{\mathscr{N}}_{x,m}\setminus\mathscr{N}_{x,m}=\widehat{\mathscr{N}}_{x,m-1},
\]
\eqref{eq:K1-1-inc} is obvious. Moreover, by Lemma \ref{lem1},
\begin{align*}
 & \limsup_{N\to\infty}\mathbb{Q}^{N}_{\xi^{x}_{N}}\left[\mathcal{H}_{\Omega_{N}\setminus\mathcal{K}_{N}}\le T\right]\\
 & \le\limsup_{N\to\infty}\mathbb{Q}^{N}_{\xi^{x}_{N}}\left[\mathcal{H}_{\Omega_{N}\setminus\mathcal{K}_{N}}\le T,\enskip\mathcal{H}_{\mathcal{X}_{N,m}}<\mathcal{H}_{\mathcal{Y}_{N,m}}\right]+\limsup_{N\to\infty}\mathbb{Q}^{N}_{\xi^{x}_{N}}\left[\mathcal{H}_{\mathcal{Y}_{N,m}}<\mathcal{H}_{\mathcal{X}_{N,m}}\right]\\
 & \le\limsup_{N\to\infty}\widetilde{\mathbb{Q}}^{N,m}_{\xi^{x}_{N}}\left[\mathcal{H}_{\mathcal{X}_{N,m}}\le T\right]<\epsilon.
\end{align*}
This proves \eqref{eq:K-inc}.

\section{\label{sec4}Metastability of Langevin dynamics}

In this section, we prove Theorem \ref{t:diff}. Throughout the section,
we assume that Assumptions \ref{assu:U} and \ref{assu:M} hold. Let
$\mathscr{L}_{\epsilon}:D(\mathscr{L}_{\epsilon})\subset C_{0}(\mathbb{R}^{d})\to C_{0}(\mathbb{R}^{d})$
denote the infinitesimal generator corresponding to $\{\bm{x}_{\epsilon}(t)\}_{t\ge0}$.
It is easy to obtain that for $u\in C^{2}(\mathbb{R}^{d})\cap C_{0}(\mathbb{R}^{d})\subset D(\mathscr{L}_{\epsilon})$,
\begin{equation}
\mathscr{L}_{\epsilon}u=-U'u'+\epsilon u''=\epsilon e^{U/\epsilon}\left[e^{-U/\epsilon}u'\right]',\label{eq:gen}
\end{equation}
where $U$ is given in \eqref{eq:U-def}. We also denote by $\mathfrak{L}$
the infinitesimal generator of the Markov chain $\{{\bf y}(t)\}_{t\ge0}$
defined by the jump rates \eqref{eq:rY}.

The following statement establishes condition $\widehat{\mathfrak{R}}$.
\begin{prop}
\label{p:R}For every $\lambda>0$, ${\bf g}\in C_{0}(\mathbb{Z})$,
and every lift $G\in C_{0}(\mathbb{R})$ satisfying
\[
G|_{\widehat{\mathcal{E}}^{n}}={\bf g}(n);\qquad n\in\mathbb{Z},
\]
the unique solution $F_{\epsilon}\in C_{0}(\mathbb{R})$ of the resolvent
equation 
\begin{equation}
(\lambda-\theta_{\epsilon}\mathscr{L}_{\epsilon})F_{\epsilon}=G,\label{eq:res-Lan}
\end{equation}
satisfies
\[
\lim_{\epsilon\to0}\sup_{x\in\mathcal{E}^{n}}\left|F_{\epsilon}(x)-{\bf f}(n)\right|=0,
\]
where ${\bf f}\in C_{0}(\mathbb{Z})$ solves $(\lambda-\mathfrak{L}){\bf f}={\bf g}$.
\end{prop}

The next proposition is condition $\mathfrak{K}$ for the accelerated
process $X_{\epsilon}(t):=\bm{x}_{\epsilon}(\theta_{\epsilon}t)$.
\begin{prop}
\label{p:K} For every $n\in\mathbb{Z}$ and $T,\eta>0$, there exists
a compact set $\mathcal{K}=\mathcal{K}(n,T,\eta)\subset\mathbb{R}$
such that 
\begin{equation}
\left\{ k\in\mathbb{Z}:\mathcal{E}^{k}\cap\mathcal{K}\ne\emptyset\right\} \enspace\text{is a finite set},\label{eq:K1-1-Lan}
\end{equation}
and
\[
\limsup_{\epsilon\to0}\sup_{x\in\mathcal{E}^{n}}\mathbb{Q}^{\epsilon}_{x}\left[\mathcal{H}_{\mathcal{K}^{c}}\le T\right]<\eta.
\]
\end{prop}

The next proposition is the mixing condition $\mathfrak{M}$.
\begin{prop}
\label{p:M}For every $n\in\mathbb{Z}$ and any $(x_{\epsilon})_{\epsilon>0}$
in $\mathcal{E}^{n}$,
\[
\lim_{\epsilon\to0}\mathbb{Q}^{\epsilon}_{x_{\epsilon}}\left[\mathcal{H}_{\mathfrak{m}_{n}}>\mathcal{H}_{\mathcal{E}\setminus\mathcal{E}^{n}}\right]=0.
\]
\end{prop}

\begin{proof}[Proof of Theorem \ref{t:diff}]
 Proposition \ref{p:M} is the first assertion of the theorem. Proposition
\ref{p:R} yields condition $\widehat{\mathfrak{R}}$ for the generators
$\theta_{\epsilon}\mathscr{L}_{\epsilon}$ and $\mathfrak{L}$. Moreover,
Proposition \ref{p:K} establishes condition $\mathfrak{K}$. It therefore
follows from Theorem \ref{t:main2} that condition $\mathfrak{C}$
and \eqref{eq:D-Delta} hold, and they are the second and third assertions
of the theorem.
\end{proof}

\subsection{Hitting time}

For $a<b$, define 
\[
\varphi^{a,b}_{\epsilon}(x):={\rm E}^{\mathbb{Q}^{\epsilon}_{x}}\left[\mathcal{H}_{\{a,b\}}\right];\qquad x\in[a,b].
\]
It is well known that $\varphi^{a,b}_{\epsilon}$ solves the boundary-value
problem
\[
\begin{cases}
\theta_{\epsilon}\mathscr{L}_{\epsilon}\varphi^{a,b}_{\epsilon}(x)=-1 & \text{if}\quad a<x<b,\\
\varphi^{a,b}_{\epsilon}(a)=\varphi^{a,b}_{\epsilon}(b)=0.
\end{cases}
\]

\begin{lem}
\label{l:hit1}For every $n\in\mathbb{Z}$,
\[
\sup_{\mathfrak{m}_{n}\le x\le\mathfrak{m}_{n+1}}{\rm E}^{\mathbb{Q}^{\epsilon}_{x}}\left[\mathcal{H}_{\{\mathfrak{m}_{n},\mathfrak{m}_{n+1}\}}\right]\le\frac{(\mathfrak{m}_{n+1}-\mathfrak{m}_{n})^{2}}{2\epsilon\theta_{\epsilon}}.
\]
\end{lem}

\begin{proof}
Since $\mathfrak{s}_{n}$ is the unique local maximum in $[\mathfrak{m}_{n},\mathfrak{m}_{n+1}]$,
we have
\[
U'(x)>0\qquad\text{for}\quad\mathfrak{m}_{n}<x<\mathfrak{s}_{n},\qquad U'(x)<0\qquad\text{for}\quad\mathfrak{s}_{n}<x<\mathfrak{m}_{n+1}.
\]
Consequently,
\[
(x-\mathfrak{s}_{n})U'(x)\le0\qquad\text{for}\quad\mathfrak{m}_{n}\le x\le\mathfrak{m}_{n+1}.
\]
Define
\[
w_{\epsilon}(x):=\frac{R^{2}-(x-\mathfrak{s}_{n})^{2}}{2\epsilon\theta_{\epsilon}};\qquad x\in[\mathfrak{m}_{n},\mathfrak{m}_{n+1}],
\]
where $R=\max\{\mathfrak{s}_{n}-\mathfrak{m}_{n},\mathfrak{m}_{n+1}-\mathfrak{s}_{n}\}$.
Then,
\[
\varphi^{\mathfrak{m}_{n},\mathfrak{m}_{n+1}}_{\epsilon}(\mathfrak{m}_{n})=0\le w_{\epsilon}(\mathfrak{m}_{n})\qquad\text{and}\qquad\varphi^{\mathfrak{m}_{n},\mathfrak{m}_{n+1}}_{\epsilon}(\mathfrak{m}_{n+1})=0\le w_{\epsilon}(\mathfrak{m}_{n+1}).
\]
Moreover, according to \eqref{eq:gen},
\[
\theta_{\epsilon}\mathscr{L}_{\epsilon}w_{\epsilon}(x)=\frac{(x-\mathfrak{s}_{n})U'(x)}{\epsilon}-1\le-1,
\]
and hence
\[
\theta_{\epsilon}\mathscr{L}_{\epsilon}(\varphi^{\mathfrak{m}_{n},\mathfrak{m}_{n+1}}_{\epsilon}-w_{\epsilon})\ge0.
\]
The maximum principle therefore yields
\[
\varphi^{\mathfrak{m}_{n},\mathfrak{m}_{n+1}}_{\epsilon}(x)\le w_{\epsilon}(x)\qquad\text{for}\quad x\in[\mathfrak{m}_{n},\mathfrak{m}_{n+1}].
\]
Finally,
\[
\sup_{\mathfrak{m}_{n}\le x\le\mathfrak{m}_{n+1}}w_{\epsilon}(x)=\sup_{\mathfrak{m}_{n}\le x\le\mathfrak{m}_{n+1}}\frac{R^{2}-(x-\mathfrak{s}_{n})^{2}}{2\epsilon}\le\frac{(\mathfrak{m}_{n+1}-\mathfrak{m}_{n})^{2}}{2\epsilon},
\]
thus the last two displays complete the proof.
\end{proof}

\begin{cor}
\label{c:hit1}For every $a<b$ and $C>0$,
\[
\limsup_{\epsilon\to0}\sup_{a\le x\le b}\mathbb{Q}^{\epsilon}_{x}\left[\mathcal{H}_{\mathcal{M}_{0}}>\frac{C}{\epsilon^{2}\theta_{\epsilon}}\right]=0.
\]
\end{cor}

\begin{proof}
Choose $n_{1}<n_{2}$ such that $[a,b]\subset(\mathfrak{m}_{n_{1}},\mathfrak{m}_{n_{2}})$.
By Lemma \ref{l:hit1}, for all $a\le x\le b$,
\[
{\rm E}^{\mathbb{Q}^{\epsilon}_{x}}\left[\mathcal{H}_{\mathcal{M}_{0}}\right]\le\frac{R^{2}}{2\epsilon\theta_{\epsilon}},
\]
where
\[
R:=\max\left\{ \mathfrak{m}_{n_{1}+1}-\mathfrak{m}_{n_{1}},\dots,\mathfrak{m}_{n_{2}}-\mathfrak{m}_{n_{2}-1}\right\} .
\]
Therefore, by the Markov inequality,
\[
\sup_{a\le x\le b}\mathbb{Q}^{\epsilon}_{x}\left[\mathcal{H}_{\mathcal{M}_{0}}>\frac{C}{\epsilon^{2}\theta_{\epsilon}}\right]\le\frac{\epsilon^{2}\theta_{\epsilon}}{C}{\rm E}^{\mathbb{Q}^{\epsilon}_{x}}\left[\mathcal{H}_{\mathcal{M}_{0}}\right]\le\frac{R^{2}}{2C}\epsilon,
\]
which completes the proof.
\end{proof}

For $a<b$, let $h^{a,b}_{\epsilon}:[a,b]\to\mathbb{R}$ be the equilibrium
potential defined by
\begin{equation}
h^{a,b}_{\epsilon}(x):=\mathbb{Q}^{\epsilon}_{x}\left[\mathcal{H}_{b}<\mathcal{H}_{a}\right].\label{eq:eq-pot}
\end{equation}
It solves the boundary-value problem
\[
\begin{cases}
\theta_{\epsilon}\mathscr{L}_{\epsilon}h^{a,b}_{\epsilon}(x)=0 & \text{if}\quad a<x<b,\\
h^{a,b}_{\epsilon}(a)=0,\\
h^{a,b}_{\epsilon}(b)=1.
\end{cases}
\]
Solving this equation explicitly, we obtain
\[
h^{a,b}_{\epsilon}(x)=\frac{\int^{x}_{a}e^{U(y)/\epsilon}\,{\rm d}y}{\int^{b}_{a}e^{U(y)/\epsilon}\,{\rm d}y}.
\]

\begin{lem}
\label{l:hit2}Fix $n\in\mathbb{Z}$. If $\mathfrak{m}_{n}\le x\le\mathfrak{s}_{n}$,\footnote{In this section, we write $a_{\epsilon}=o_{\epsilon}(1)$ if the collection
$(a_{\epsilon})_{\epsilon>0}$ satisfies $\lim_{\epsilon\to0}|a_{\epsilon}|=0$
uniformly over the corresponding range of $x$.}
\[
\mathbb{Q}^{\epsilon}_{x}\left[\mathcal{H}_{\mathfrak{m}_{n+1}}<\mathcal{H}_{\mathfrak{m}_{n}}\right]\le[1+o_{\epsilon}(1)]\sqrt{\frac{-U''(\mathfrak{s}_{n})}{2\pi\epsilon}}(x-\mathfrak{m}_{n})e^{\frac{U(x)-U(\mathfrak{s}_{n})}{\epsilon}},
\]
and if $\mathfrak{s}_{n-1}\le x\le\mathfrak{m}_{n}$,
\[
\mathbb{Q}^{\epsilon}_{x}\left[\mathcal{H}_{\mathfrak{m}_{n-1}}<\mathcal{H}_{\mathfrak{m}_{n}}\right]\le[1+o_{\epsilon}(1)]\sqrt{\frac{-U''(\mathfrak{s}_{n-1})}{2\pi\epsilon}}(\mathfrak{m}_{n}-x)e^{\frac{U(x)-U(\mathfrak{s}_{n-1})}{\epsilon}},
\]
\end{lem}

\begin{proof}
By the Laplace method,
\[
\int^{\mathfrak{m}_{n+1}}_{\mathfrak{m}_{n}}e^{U(y)/\epsilon}\,{\rm d}y=[1+o_{\epsilon}(1)]\sqrt{\frac{2\pi\epsilon}{-U''(\mathfrak{s}_{n})}}e^{U(\mathfrak{s}_{n})/\epsilon}.
\]
Since $U$ is increasing on $[\mathfrak{m}_{n},x]$, we have
\[
\int^{x}_{\mathfrak{m}_{n}}e^{U(y)/\epsilon}dy\le(x-\mathfrak{m}_{n})e^{U(x)/\epsilon}.
\]
Therefore, for $\mathfrak{m}_{n}\le x\le\mathfrak{s}_{n}$,
\[
h^{\mathfrak{m}_{n},\mathfrak{m}_{n+1}}_{\epsilon}(x)\le[1+o_{\epsilon}(1)]\sqrt{\frac{-U''(\mathfrak{s}_{n})}{2\pi\epsilon}}(x-\mathfrak{m}_{n})e^{\frac{U(x)-U(\mathfrak{s}_{n})}{\epsilon}}.
\]
This proves the first assertion. The second one follows similarly.
\end{proof}

Set
\[
\delta:=\sqrt{\epsilon\log\frac{1}{\epsilon}}.
\]

\begin{cor}
\label{c:hit2}Fix $n\in\mathbb{Z}$. For every $a>\sqrt{-1/U''(\mathfrak{s}_{n-1})}$
and $b>\sqrt{-1/U''(\mathfrak{s}_{n})}$,
\[
\liminf_{\epsilon\to0}\inf_{\mathfrak{s}_{n-1}+a\delta\le x\le\mathfrak{s}_{n}-b\delta}\mathbb{Q}^{\epsilon}_{x}\left[\mathcal{H}_{\mathcal{M}_{0}}=\mathcal{H}_{\mathfrak{m}_{n}}\right]=1.
\]
\end{cor}

\begin{proof}
We first consider $\mathfrak{s}_{n-1}+a\delta\le x\le\mathfrak{m}_{n}$.
Starting from $\mathfrak{s}_{n-1}+a\delta\le x\le\mathfrak{m}_{n}$,
the first local minimum visited by the process is either $\mathfrak{m}_{n}$
or $\mathfrak{m}_{n-1}$. Thus,
\[
\mathbb{Q}^{\epsilon}_{x}\left[\mathcal{H}_{\mathcal{M}_{0}}=\mathcal{H}_{\mathfrak{m}_{n}}\right]=\mathbb{Q}^{\epsilon}_{x}\left[\mathcal{H}_{\mathfrak{m}_{n}}<\mathcal{H}_{\mathfrak{m}_{n-1}}\right].
\]
By Lemma \ref{l:hit2}, there exists a constant $C>0$ such that for
all $\mathfrak{s}_{n-1}+a\delta\le x\le\mathfrak{m}_{n}$,
\[
\mathbb{Q}^{\epsilon}_{x}\left[\mathcal{H}_{\mathfrak{m}_{n}}<\mathcal{H}_{\mathfrak{m}_{n-1}}\right]\ge1-\frac{C}{\sqrt{\epsilon}}e^{\frac{U(x)-U(\mathfrak{s}_{n-1})}{\epsilon}}\ge1-\frac{C}{\sqrt{\epsilon}}e^{\frac{U(\mathfrak{s}_{n-1}+a\delta)-U(\mathfrak{s}_{n-1})}{\epsilon}}.
\]
Taylor expansion at $\mathfrak{s}_{n-1}$ yields
\[
U(\mathfrak{s}_{n-1}+a\delta)-U(\mathfrak{s}_{n-1})=\frac{U''(\mathfrak{s}_{n-1})}{2}a^{2}\delta^{2}+O(\delta^{3}).
\]
Since $\frac{\delta^{3}}{\epsilon}\to0$, we obtain that
\[
\epsilon^{-1/2}e^{\frac{U(\mathfrak{s}_{n-1}+a\delta)-U(\mathfrak{s}_{n-1})}{\epsilon}}=[1+o_{\epsilon}(1)]\epsilon^{-1/2}e^{\frac{U''(\mathfrak{s}_{n-1})}{2}a^{2}\log\frac{1}{\epsilon}}=[1+o_{\epsilon}(1)]\epsilon^{\frac{-a^{2}U''(\mathfrak{s}_{n-1})-1}{2}}.
\]
The exponent is strictly positive since $a>\sqrt{-1/U''(\mathfrak{s}_{n-1})}$.
Therefore, 
\[
\liminf_{\epsilon\to0}\inf_{\mathfrak{s}_{n-1}+a\delta\le x\le\mathfrak{m}_{n}}\mathbb{Q}^{\epsilon}_{x}\left[\mathcal{H}_{\mathcal{M}_{0}}=\mathcal{H}_{\mathfrak{m}_{n}}\right]\ge1-C\limsup_{\epsilon\to0}\epsilon^{\frac{-a^{2}U''(\mathfrak{s}_{n-1})-1}{2}}=1.
\]
The same argument applied to the interval $\mathfrak{m}_{n}\le x\le\mathfrak{s}_{n}-b\delta$
gives
\[
\liminf_{\epsilon\to0}\inf_{\mathfrak{m}_{n}\le x\le\mathfrak{s}_{n}-b\delta}\mathbb{Q}^{\epsilon}_{x}\left[\mathcal{H}_{\mathcal{M}_{0}}=\mathcal{H}_{\mathfrak{m}_{n}}\right]\ge1-C'\lim_{\epsilon\to0}\epsilon^{\frac{-b^{2}U''(\mathfrak{s}_{n})-1}{2}}=1,
\]
for some another constant $C'>0$. Combining the two estimates proves
the claim.
\end{proof}

\begin{prop}
\label{p:hit}Fix $n\in\mathbb{Z}$. For every $a>\sqrt{-1/U''(\mathfrak{s}_{n-1})}$,
$b>\sqrt{-1/U''(\mathfrak{s}_{n})}$, and $C>0$,
\[
\limsup_{\epsilon\to0}\sup_{\mathfrak{s}_{n-1}+a\delta\le x\le\mathfrak{s}_{n}-b\delta}\mathbb{Q}^{\epsilon}_{x}\left[\mathcal{H}_{\mathfrak{m}_{n}}>\frac{C}{\epsilon^{2}\theta_{\epsilon}}\right]=0.
\]
\end{prop}

\begin{proof}
Decompose the probability as 
\[
\begin{aligned}\mathbb{Q}^{\epsilon}_{x}\left[\mathcal{H}_{\mathfrak{m}_{n}}>\frac{C}{\epsilon^{2}\theta_{\epsilon}}\right] & =\mathbb{Q}^{\epsilon}_{x}\left[\mathcal{H}_{\mathfrak{m}_{n}}>\frac{C}{\epsilon^{2}\theta_{\epsilon}},\enspace\mathcal{H}_{\mathcal{M}_{0}}=\mathcal{H}_{\mathfrak{m}_{n}}\right]+\mathbb{Q}^{\epsilon}_{x}\left[\mathcal{H}_{\mathfrak{m}_{n}}>\frac{C}{\epsilon^{2}\theta_{\epsilon}},\enspace\mathcal{H}_{\mathcal{M}_{0}}\ne\mathcal{H}_{\mathfrak{m}_{n}}\right]\end{aligned}
.
\]
The last two probabilities vanish uniformly for $\mathfrak{s}_{n-1}+a\delta\le x\le\mathfrak{s}_{n}-b\delta$
by Corollaries \ref{c:hit1} and \ref{c:hit2}, respectively.
\end{proof}

\subsection{\label{sec4.1}Resolvent condition}

\subsubsection{\label{sec4.1.2}Test functions}

For $n\in\mathbb{Z}$, set
\begin{equation}
\begin{aligned}\mathcal{W}^{n}_{\epsilon} & :=\left[\mathfrak{s}_{n-1}+\frac{2\delta}{\sqrt{-U''(\mathfrak{s}_{n-1})}},\mathfrak{s}_{n}-\frac{2\delta}{\sqrt{-U''(\mathfrak{s}_{n})}}\right],\\
\mathcal{C}^{n}_{\epsilon} & :=\left[\mathfrak{s}_{n}-\frac{2\delta}{\sqrt{-U''(\mathfrak{s}_{n})}},\mathfrak{s}_{n}+\frac{2\delta}{\sqrt{-U''(\mathfrak{s}_{n})}}\right],
\end{aligned}
\label{eq:WeN-CeN}
\end{equation}
and define $p^{n}_{\epsilon}:\mathcal{C}^{n}_{\epsilon}\to\mathbb{R}$
as
\begin{equation}
p^{n}_{\epsilon}(x):=\frac{1}{M^{n}_{\epsilon}}\int^{x-\mathfrak{s}_{n}}_{-2\delta/\sqrt{-U''(\mathfrak{s}_{n})}}e^{\frac{U''(\mathfrak{s}_{n})}{2\epsilon}t^{2}}\,{\rm d}t,\label{eq:def_p}
\end{equation}
where the normalizing constant is given by
\[
M^{n}_{\epsilon}:=\int^{2\delta/\sqrt{-U''(\mathfrak{s}_{n})}}_{-2\delta/\sqrt{-U''(\mathfrak{s}_{n})}}e^{\frac{U''(\mathfrak{s}_{n})}{2\epsilon}t^{2}}\,{\rm d}t=\sqrt{\frac{2\pi\epsilon}{-U''(\mathfrak{s}_{n})}}(1+o_{\epsilon}(1)).
\]
By the following lemma, which is proved in \cite[Proposition 8.5]{LS22a}
more generally for all dimensions $d$, the function $p^{n}_{\epsilon}$
constructed above approximates $h^{\mathfrak{m}_{n},\mathfrak{m}_{n+1}}_{\epsilon}$
(cf. \eqref{eq:eq-pot}).
\begin{lem}
\label{lem:test0}For $n\in\mathbb{Z}$,
\[
\int_{\mathcal{C}^{n}_{\epsilon}}\left|\mathscr{L}_{\epsilon}p^{n}_{\epsilon}(x)\right|e^{-U(x)/\epsilon}\,{\rm d}x=o_{\epsilon}(1)\sqrt{\epsilon}e^{-U(\mathfrak{s}_{n})/\epsilon}.
\]
\end{lem}

For $n\in\mathbb{Z}$, define $\phi^{n}_{\epsilon}:\mathbb{R}\to\mathbb{R}$
as
\begin{equation}
\phi^{n}_{\epsilon}(x):=\begin{cases}
1 & \text{if}\quad x\in\mathcal{W}^{n}_{\epsilon},\\
1-p^{n}_{\epsilon}(x) & \text{if}\quad x\in\mathcal{C}^{n}_{\epsilon},\\
p^{n-1}_{\epsilon}(x) & \text{if}\quad x\in\mathcal{C}^{n-1}_{\epsilon},\\
0 & \text{otherwise}.
\end{cases}\label{eq:def_phi}
\end{equation}
Then, $\phi^{n}_{\epsilon}$ is continuous, supported on $\mathcal{C}^{n-1}_{\epsilon}\cup\mathcal{W}^{n}_{\epsilon}\cup\mathcal{C}^{n}_{\epsilon}$,
$0\le\phi^{n}_{\epsilon}\le1$, $\|\phi^{n}_{\epsilon}\|_{L^{\infty}(\mathbb{R})}=1$,
and there exists $C>0$ such that
\[
\|(\phi^{n}_{\epsilon})'\|_{L^{\infty}(\mathbb{R})}\le C\epsilon^{-1/2},\qquad\|(\phi^{n}_{\epsilon})''\|_{L^{\infty}(\mathbb{R})}\le C\epsilon^{-3/2}.
\]

\begin{lem}
\label{lem:test1}Recall $\nu_{N}$ from \eqref{eq:mass}. For each
$n\in\mathbb{Z}$, 
\begin{equation}
\int_{\mathcal{E}^{n}}\phi^{n}_{\epsilon}(x)e^{-U(x)/\epsilon}\,{\rm d}x=\nu_{n}\sqrt{2\pi\epsilon}e^{-U(\mathfrak{m}_{n})/\epsilon}+o_{\epsilon}(1)\sqrt{\epsilon}e^{-U(\mathfrak{m}_{n})/\epsilon}\label{eq:l_test1-1}
\end{equation}
and
\begin{equation}
\int_{\mathbb{R}\setminus\mathcal{E}^{n}}\phi^{n}_{\epsilon}(x)e^{-U(x)/\epsilon}\,{\rm d}x=o_{\epsilon}(1)\sqrt{\epsilon}e^{-U(\mathfrak{m}_{n})/\epsilon}.\label{eq:l_test1-2}
\end{equation}
\end{lem}

\begin{proof}
Fix $n\in\mathbb{Z}$. Since $\phi^{n}_{\epsilon}(x)=1$ for $x\in\mathcal{E}^{n}$,
the Laplace method proves \eqref{eq:l_test1-1}. Note that $0\le\phi^{n}_{\epsilon}\le1$
and $\phi^{n}_{\epsilon}$ is supported on $\mathcal{C}^{n-1}_{\epsilon}\cup\mathcal{W}^{n}_{\epsilon}\cup\mathcal{C}^{n}_{\epsilon}$.
Since $\bigcap_{\epsilon>0}\mathcal{C}^{k}_{\epsilon}=\{\mathfrak{s}_{k}\}$
and $U(\mathfrak{s}_{n-1}),U(\mathfrak{s}_{n})\ge U(\mathfrak{m}_{n})+D>U(\mathfrak{m}_{n})$,
\[
\int_{\mathcal{C}^{n-1}_{\epsilon}\cup\mathcal{C}^{n}_{\epsilon}}\phi^{n}_{\epsilon}(x)e^{-U(x)/\epsilon}\,{\rm d}x=o_{\epsilon}(1)\sqrt{\epsilon}e^{-U(\mathfrak{m}_{n})/\epsilon}.
\]
Also, since, for some $c>0$, $U(x)>U(\mathfrak{m}_{n})+c$ for all
$x\in\mathcal{W}^{n}_{\epsilon}\setminus\mathcal{E}^{n}$,
\[
\int_{\mathcal{W}^{n}_{\epsilon}\setminus\mathcal{E}^{n}}\phi^{n}_{\epsilon}(x)e^{-U(x)/\epsilon}\,{\rm d}x=o_{\epsilon}(1)\sqrt{\epsilon}e^{-U(\mathfrak{m}_{n})/\epsilon}.
\]
The two last displayed equations yield \eqref{eq:l_test1-2}.
\end{proof}

\subsubsection{Proof of Proposition \ref{p:R}}

Fix $\lambda>0$, ${\bf g}\in C_{0}(\mathbb{Z})$, and let $G\in C_{0}(\mathbb{R})$
satisfy
\begin{equation}
G|_{\widehat{\mathcal{E}}^{n}}={\bf g}(n);\qquad n\in\mathbb{Z}.\label{eq:G}
\end{equation}
Let $F_{\epsilon}\in C_{0}(\mathbb{R})$ be the unique solution to
\eqref{eq:res-Lan}. By the next result, $F_{\epsilon}$ is flat on
$\mathcal{W}^{n}_{\epsilon}$ (cf. \eqref{eq:WeN-CeN}).
\begin{lem}
\label{l:flat_ext}For all $n\in\mathbb{Z}$,
\[
\lim_{\epsilon\to0}\sup_{x\in\mathcal{W}^{n}_{\epsilon}}\left|F_{\epsilon}(x)-F_{\epsilon}(\mathfrak{m}_{n})\right|=0.
\]
\end{lem}

\begin{proof}
By \eqref{eq:pr},
\[
\begin{aligned}F_{\epsilon}(x) & ={\rm E}^{\mathbb{Q}^{\epsilon}_{x}}\left[\int^{\infty}_{0}e^{-\lambda t}\,G(X_{\epsilon}(t))\,{\rm d}t\right]\\
 & ={\rm E}^{\mathbb{Q}^{\epsilon}_{x}}\left[\int^{\infty}_{0}e^{-\lambda t}\,G(X_{\epsilon}(t))\,{\rm d}t\,{\bf 1}\left\{ \mathcal{H}_{\mathfrak{m}_{n}}\le\frac{1}{\epsilon^{2}\theta_{\epsilon}}\right\} \right]\\
 & \quad+{\rm E}^{\mathbb{Q}^{\epsilon}_{x}}\left[\int^{\infty}_{0}e^{-\lambda t}\,G(X_{\epsilon}(t))\,{\rm d}t\,{\bf 1}\left\{ \mathcal{H}_{\mathfrak{m}_{n}}>\frac{1}{\epsilon^{2}\theta_{\epsilon}}\right\} \right].
\end{aligned}
\]
Since $G$ is bounded, by Proposition \ref{p:hit}, the last expectation
vanishes uniformly on $x\in\mathcal{W}^{n}_{\epsilon}$. Hence,
\[
F_{\epsilon}(x)={\rm E}^{\mathbb{Q}^{\epsilon}_{x}}\left[\int^{\infty}_{0}e^{-\lambda t}\,G(X_{\epsilon}(t))\,{\rm d}t\,{\bf 1}\left\{ \mathcal{H}_{\mathfrak{m}_{n}}\le\frac{1}{\epsilon^{2}\theta_{\epsilon}}\right\} \right]+o_{\epsilon}(1).
\]
By the strong Markov property, the last expectation is equal to
\[
\begin{aligned} & {\rm E}^{\mathbb{Q}^{\epsilon}_{x}}\left[{\rm E}^{\mathbb{Q}^{\epsilon}_{X_{\epsilon}(\mathcal{H}_{\mathfrak{m}_{n}})}}\left[\int^{\infty}_{0}e^{-\lambda t}\,G(X_{\epsilon}(t))\,{\rm d}t\right]e^{-\lambda\mathcal{H}_{\mathfrak{m}_{n}}}\,{\bf 1}\left\{ \mathcal{H}_{\mathfrak{m}_{n}}\le\frac{1}{\epsilon^{2}\theta_{\epsilon}}\right\} \right]\\
 & ={\rm E}^{\mathbb{Q}^{\epsilon}_{x}}\left[{\rm E}^{\mathbb{Q}^{\epsilon}_{\mathfrak{m}_{n}}}\left[\int^{\infty}_{0}e^{-\lambda t}\,G(X_{\epsilon}(t))\,{\rm d}t\right]e^{-\lambda\mathcal{H}_{\mathfrak{m}_{n}}}\,{\bf 1}\left\{ \mathcal{H}_{\mathfrak{m}_{n}}\le\frac{1}{\epsilon^{2}\theta_{\epsilon}}\right\} \right]\\
 & ={\rm E}^{\mathbb{Q}^{\epsilon}_{x}}\left[F_{\epsilon}(\mathfrak{m}_{n})\,e^{-\lambda\mathcal{H}_{\mathfrak{m}_{n}}}\,{\bf 1}\left\{ \mathcal{H}_{\mathfrak{m}_{n}}\le\frac{1}{\epsilon^{2}\theta_{\epsilon}}\right\} \right].
\end{aligned}
\]
Since $G$ is bounded, the last expectation is equal to
\[
F_{\epsilon}(\mathfrak{m}_{n})\,\mathbb{Q}^{\epsilon}_{x}\left[\mathcal{H}_{\mathfrak{m}_{n}}\le\frac{1}{\epsilon^{2}\theta_{\epsilon}}\right]+o_{\epsilon}(1).
\]
Hence, the proof is completed by applying Proposition \ref{p:hit}
again.
\end{proof}

\begin{lem}
\label{l:p_R-1}We have
\[
\begin{aligned}\int_{\mathbb{R}}\phi^{n}_{\epsilon}(x)F_{\epsilon}(x)e^{-U(x)/\epsilon}\,{\rm d}x & =\nu_{n}F_{\epsilon}(\mathfrak{m}_{n})\sqrt{2\pi\epsilon}e^{-U(\mathfrak{m}_{n})/\epsilon}+o_{\epsilon}(1)\sqrt{\epsilon}e^{-U(\mathfrak{m}_{n})/\epsilon},\\
\int_{\mathbb{R}}\phi^{n}_{\epsilon}(x)G(x)e^{-U(x)/\epsilon}\,{\rm d}x & =\nu_{n}{\bf g}(n)\sqrt{2\pi\epsilon}e^{-U(\mathfrak{m}_{n})/\epsilon}+o_{\epsilon}(1)\sqrt{\epsilon}e^{-U(\mathfrak{m}_{n})/\epsilon}.
\end{aligned}
\]
\end{lem}

\begin{proof}
The boundedness of $F_{\epsilon}$ and $G$, Lemmas \ref{lem:test1}
and \ref{l:flat_ext}, and \eqref{eq:G} prove the lemma.
\end{proof}

Define ${\bf f}_{\epsilon}:\mathbb{Z}\to\mathbb{R}$ as ${\bf f}_{\epsilon}(n):=F_{\epsilon}(\mathfrak{m}_{n})$.
\begin{lem}
\label{l:p_R-2}For $n\in\mathbb{Z}$,
\[
\theta_{\epsilon}\int_{\mathbb{R}}\phi^{n}_{\epsilon}(x)(-\mathscr{L}_{\epsilon}F_{\epsilon})(x)e^{-U(x)/\epsilon}\,{\rm d}x=-\sqrt{2\pi\epsilon}e^{-U(\mathfrak{m}_{n})/\epsilon}\nu_{n}\mathfrak{L}{\bf f}_{\epsilon}(n)+o_{\epsilon}(1)\sqrt{\epsilon}e^{-U(\mathfrak{m}_{n})/\epsilon}.
\]
\end{lem}

\begin{proof}
Since $F_{\epsilon}$ may not be differentiable, we cannot use the
integration by parts directly. Nevertheless, since $C^{\infty}_{c}(\mathbb{R})$
is a core of $\mathscr{L}_{\epsilon}$, for each $\epsilon>0$, there
exists $\hat{F}_{\epsilon}\in C^{\infty}_{c}(\mathbb{R})$ such that
\begin{equation}
\|\hat{F}_{\epsilon}-F_{\epsilon}\|_{L^{\infty}(\mathbb{R})}<\epsilon\theta^{-1}_{\epsilon}\qquad\text{and}\qquad\|\mathscr{L}_{\epsilon}\hat{F}_{\epsilon}-\mathscr{L}_{\epsilon}F_{\epsilon}\|_{L^{\infty}(\mathbb{R})}<\epsilon\theta^{-1}_{\epsilon}.\label{eq:Feps-approx}
\end{equation}
Since $\hat{F}_{\epsilon}\in C^{\infty}_{c}(\mathbb{R})$,
\[
\mathscr{L}_{\epsilon}\hat{F}_{\epsilon}=\epsilon e^{U/\epsilon}\left[e^{-U/\epsilon}\left(\hat{F}_{\epsilon}\right)'\right]'=-U'\left(\hat{F}_{\epsilon}\right)'+\epsilon\left(\hat{F}_{\epsilon}\right)''.
\]
Hence, since $\phi^{n}_{\epsilon}$ is bounded, by Lemma \ref{lem:test1},
\[
\begin{aligned} & \left|\theta_{\epsilon}\int_{\mathbb{R}}\phi^{n}_{\epsilon}(x)(-\mathscr{L}_{\epsilon}F_{\epsilon})(x)e^{-U(x)/\epsilon}\,{\rm d}x-\theta_{\epsilon}\int_{\mathbb{R}^{d}}\phi^{n}_{\epsilon}(x)(-\mathscr{L}_{\epsilon}\hat{F}_{\epsilon})(x)e^{-U(x)/\epsilon}\,{\rm d}x\right|\\
 & \le\epsilon\left|\int_{\mathbb{R}}\phi^{n}_{\epsilon}(x)e^{-U(x)/\epsilon}\,{\rm d}x\right|=o_{\epsilon}(1)\sqrt{\epsilon}e^{-U(\mathfrak{m}_{n})/\epsilon}.
\end{aligned}
\]
Also,
\[
\left|\sqrt{2\pi\epsilon}e^{-U(\mathfrak{m}_{n})/\epsilon}\nu_{n}\mathfrak{L}F_{\epsilon}(\mathfrak{m}_{n})-\sqrt{2\pi\epsilon}e^{-U(\mathfrak{m}_{n})/\epsilon}\nu_{n}\mathfrak{L}\hat{F}_{\epsilon}(\mathfrak{m}_{n})\right|=o_{\epsilon}(1)\sqrt{\epsilon}e^{-U(\mathfrak{m}_{n})/\epsilon}.
\]
Therefore, it suffices to prove that
\begin{equation}
\theta_{\epsilon}\int_{\mathbb{R}}\phi^{n}_{\epsilon}(x)(-\mathscr{L}_{\epsilon}\hat{F}_{\epsilon})(x)e^{-U(x)/\epsilon}\,{\rm d}x=-\sqrt{2\pi\epsilon}e^{-U(\mathfrak{m}_{n})/\epsilon}\nu_{n}\mathfrak{L}{\bf f}_{\epsilon}(n)+o_{\epsilon}(1)\sqrt{\epsilon}e^{-U(\mathfrak{m}_{n})/\epsilon}.\label{eq:l_pR2-0}
\end{equation}

Since $\phi^{n}_{\epsilon}$ is supported on $\mathcal{C}^{n-1}_{\epsilon}\cup\mathcal{W}^{n}_{\epsilon}\cup\mathcal{C}^{n}_{\epsilon}$
and $(\phi^{n}_{\epsilon})'=0$ on $\mathcal{W}^{n}_{\epsilon}$,
by the integration by parts, 
\[
\begin{aligned}\int_{\mathbb{R}}\phi^{n}_{\epsilon}(x)(-\mathscr{L}_{\epsilon}\hat{F}_{\epsilon})(x)e^{-U(x)/\epsilon}\,{\rm d}x & =\epsilon\int_{\mathcal{C}^{n-1}_{\epsilon}}\left(\phi^{n}_{\epsilon}\right)'(x)\left(\hat{F}_{\epsilon}\right)'(x)e^{-U(x)/\epsilon}\,{\rm d}x\\
 & \quad+\epsilon\int_{\mathcal{C}^{n}_{\epsilon}}\left(\phi^{n}_{\epsilon}\right)'(x)\left(\hat{F}_{\epsilon}\right)'(x)e^{-U(x)/\epsilon}\,{\rm d}x.
\end{aligned}
\]
By the definition \eqref{eq:def_phi}, another integration by parts
yields
\[
\begin{aligned}\epsilon\int_{\mathcal{C}^{n}_{\epsilon}}\left(\phi^{n}_{\epsilon}\right)'(x)\left(\hat{F}_{\epsilon}\right)'(x)e^{-U(x)/\epsilon}\,{\rm d}x & =-\epsilon\int_{\mathcal{C}^{n}_{\epsilon}}\left(p^{n}_{\epsilon}\right)'(x)\left(\hat{F}_{\epsilon}\right)'(x)e^{-U(x)/\epsilon}\,{\rm d}x\\
 & =-\epsilon\left[\hat{F}_{\epsilon}(x)\left(p^{n}_{\epsilon}\right)'(x)e^{-U(x)/\epsilon}\right]^{\mathfrak{s}_{n}+2\delta/\sqrt{-U''(\mathfrak{s}_{n})}}_{\mathfrak{s}_{n}-2\delta/\sqrt{-U''(\mathfrak{s}_{n})}}\\
 & \quad+\int_{\mathcal{C}^{n}_{\epsilon}}\hat{F}_{\epsilon}(x)\mathscr{L}_{\epsilon}p^{n}_{\epsilon}(x)e^{-U(x)/\epsilon}\,{\rm d}x.
\end{aligned}
\]
Since $\widehat{F}_{\epsilon}$ is bounded, by Lemma \ref{lem:test0},
the last integral is $o_{\epsilon}(1)\sqrt{\epsilon}e^{-U(\mathfrak{s}_{n})/\epsilon}$.
By the definition \eqref{eq:def_p},
\[
\left(p^{n}_{\epsilon}\right)'(\mathfrak{s}_{n}\pm2\delta/\sqrt{-U''(\mathfrak{s}_{n})})=(1+o_{\epsilon}(1))\sqrt{\frac{-U''(\mathfrak{s}_{n})}{2\pi\epsilon}}\,e^{-2\log\frac{1}{\epsilon}}.
\]
 For $x\in\mathcal{C}^{n}_{\epsilon}$, by the Taylor expansion,
\[
U(x)=U(\mathfrak{s}_{n})+\frac{U''(\mathfrak{s}_{n})}{2}(x-\mathfrak{s}_{n})^{2}+O((x-\mathfrak{s}_{n})^{3})
\]
so that
\[
e^{-U(\mathfrak{s}_{n}\pm2\delta/\sqrt{-U''(\mathfrak{s}_{n})})/\epsilon}=e^{-U(\mathfrak{s}_{n})/\epsilon}e^{2\log\frac{1}{\epsilon}}[1+o_{\epsilon}(1)].
\]
Hence,
\[
\begin{aligned} & -\epsilon\left[\hat{F}_{\epsilon}(x)\left(p^{n}_{\epsilon}\right)'(x)e^{-U(x)/\epsilon}\right]^{\mathfrak{s}_{n}+2\delta/\sqrt{-U''(\mathfrak{s}_{n})}}_{\mathfrak{s}_{n}-2\delta/\sqrt{-U''(\mathfrak{s}_{n})}}\\
 & =\sqrt{\frac{-U''(\mathfrak{s}_{n})\epsilon}{2\pi}}\left[\hat{F}_{\epsilon}\left(\mathfrak{s}_{n}-\frac{2\delta}{\sqrt{-U''(\mathfrak{s}_{n})}}\right)-\hat{F}_{\epsilon}\left(\mathfrak{s}_{n}+\frac{2\delta}{\sqrt{-U''(\mathfrak{s}_{n})}}\right)\right]e^{-U(\mathfrak{s}_{n})/\epsilon}[1+o_{\epsilon}(1)].
\end{aligned}
\]
By Lemma \ref{l:flat_ext}, the last expression is equal to
\[
\sqrt{\frac{-U''(\mathfrak{s}_{n})\epsilon}{2\pi}}\left[\hat{F}_{\epsilon}(\mathfrak{m}_{n})-\hat{F}_{\epsilon}(\mathfrak{m}_{n+1})\right]e^{-U(\mathfrak{s}_{n})/\epsilon}+o_{\epsilon}(1)\sqrt{\epsilon}e^{-U(\mathfrak{s}_{n})/\epsilon}.
\]
In summary, by \eqref{eq:mass},
\[
\begin{aligned} & \epsilon\int_{\mathcal{C}^{n}_{\epsilon}}\left(\phi^{n}_{\epsilon}\right)'(x)\left(\hat{F}_{\epsilon}\right)'(x)e^{-U(x)/\epsilon}\,{\rm d}x\\
 & =\sqrt{\frac{-U''(\mathfrak{s}_{n})\epsilon}{2\pi}}\left[\hat{F}_{\epsilon}(\mathfrak{m}_{n})-\hat{F}_{\epsilon}(\mathfrak{m}_{n+1})\right]e^{-U(\mathfrak{s}_{n})/\epsilon}+o_{\epsilon}(1)\sqrt{\epsilon}e^{-U(\mathfrak{s}_{n})/\epsilon}\\
 & =\sqrt{2\pi\epsilon}\nu_{n}\frac{\omega_{n}}{\nu_{n}}\left[\hat{F}_{\epsilon}(\mathfrak{m}_{n})-\hat{F}_{\epsilon}(\mathfrak{m}_{n+1})\right]e^{-U(\mathfrak{s}_{n})/\epsilon}+o_{\epsilon}(1)\sqrt{\epsilon}e^{-U(\mathfrak{s}_{n})/\epsilon}.
\end{aligned}
\]
By the same computation,
\[
\begin{aligned} & \epsilon\int_{\mathcal{C}^{n-1}_{\epsilon}}\left(\phi^{n}_{\epsilon}\right)'(x)\left(\hat{F}_{\epsilon}\right)'(x)e^{-U(x)/\epsilon}\,{\rm d}x\\
 & =\sqrt{2\pi\epsilon}\nu_{n}\frac{\omega_{n-1}}{\nu_{n}}\left[\hat{F}_{\epsilon}(\mathfrak{m}_{n})-\hat{F}_{\epsilon}(\mathfrak{m}_{n-1})\right]e^{-U(\mathfrak{s}_{n})/\epsilon}+o_{\epsilon}(1)\sqrt{\epsilon}e^{-U(\mathfrak{s}_{n-1})/\epsilon}.
\end{aligned}
\]
Hence,
\[
\begin{aligned} & \theta_{\epsilon}\int_{\mathbb{R}}\phi^{n}_{\epsilon}(x)(-\mathscr{L}_{\epsilon}\hat{F}_{\epsilon})(x)e^{-U(x)/\epsilon}\,{\rm d}x\\
 & =\sqrt{2\pi\epsilon}\nu_{n}\frac{\omega_{n-1}}{\nu_{n}}\left[\hat{F}_{\epsilon}(\mathfrak{m}_{n})-\hat{F}_{\epsilon}(\mathfrak{m}_{n-1})\right]\theta_{\epsilon}e^{-U(\mathfrak{s}_{n-1})/\epsilon}\\
 & \quad+\sqrt{2\pi\epsilon}\nu_{n}\frac{\omega_{n}}{\nu_{n}}\left[\hat{F}_{\epsilon}(\mathfrak{m}_{n})-\hat{F}_{\epsilon}(\mathfrak{m}_{n+1})\right]\theta_{\epsilon}e^{-U(\mathfrak{s}_{n})/\epsilon}\\
 & \qquad+o_{\epsilon}(1)\sqrt{\epsilon}\theta_{\epsilon}e^{-U(\mathfrak{s}_{n-1})/\epsilon}+o_{\epsilon}(1)\sqrt{\epsilon}\theta_{\epsilon}e^{-U(\mathfrak{s}_{n})/\epsilon}.
\end{aligned}
\]
Note that $\theta_{\epsilon}=e^{D/\epsilon}$ and $U(\mathfrak{s}_{n-1})-U(\mathfrak{m}_{n}),U(\mathfrak{s}_{n})-U(\mathfrak{m}_{n})\ge D$.
Hence,
\[
\begin{aligned} & \theta_{\epsilon}\int_{\mathbb{R}}\phi^{n}_{\epsilon}(x)(-\mathscr{L}_{\epsilon}\hat{F}_{\epsilon})(x)e^{-U(x)/\epsilon}dx\\
 & =\sqrt{2\pi\epsilon}\nu_{n}\frac{\omega_{n-1}}{\nu_{n}}\left[F_{\epsilon}(\mathfrak{m}_{n})-F_{\epsilon}(\mathfrak{m}_{n-1})\right]e^{-U(\mathfrak{m}_{n})/\epsilon}{\bf 1}\{\mathfrak{h}^{-}_{n}=D\}\\
 & \quad+\sqrt{2\pi\epsilon}\nu_{n}\frac{\omega_{n}}{\nu_{n}}\left[F_{\epsilon}(\mathfrak{m}_{n})-F_{\epsilon}(\mathfrak{m}_{n+1})\right]e^{-U(\mathfrak{m}_{n})/\epsilon}{\bf 1}\{\mathfrak{h}^{+}_{n}=D\}+o_{\epsilon}(1)\sqrt{\epsilon}e^{-U(\mathfrak{m}_{n})/\epsilon}\\
 & =-\sqrt{2\pi\epsilon}e^{-U(\mathfrak{m}_{n})/\epsilon}\nu_{n}\mathfrak{L}{\bf f}_{\epsilon}(n)+o_{\epsilon}(1)\sqrt{\epsilon}e^{-U(\mathfrak{m}_{n})/\epsilon},
\end{aligned}
\]
which is \eqref{eq:l_pR2-0}.
\end{proof}

Now, we are ready to prove Proposition \ref{p:R}.
\begin{proof}[Proof of Proposition \ref{p:R}]
 Fix $n\in\mathbb{Z}$ and let $\phi^{n}_{\epsilon}$ be defined
as in \eqref{eq:def_phi}. By multiplying both sides of \eqref{eq:res-Lan}
by $\phi^{n}_{\epsilon}$ and integrating over $\mathbb{R}$ with
respect to the measure $e^{-U(x)/\epsilon}\,{\rm d}x$, Lemmas \ref{l:p_R-1}
and \ref{l:p_R-2} yield
\[
\lim_{\epsilon\to0}(\lambda-\mathfrak{L}){\bf f}_{\epsilon}(n)={\bf g}(n)=(\lambda-\mathfrak{L}){\bf f}(n),
\]
which implies
\[
\lim_{\epsilon\to0}(\lambda-\mathfrak{L})({\bf f}_{\epsilon}-{\bf f})(n)=0.
\]
Since $\mathfrak{L}$ is the infinitesimal generator of the Markov
chain $\{{\bf y}(t)\}_{t\ge0}$, i.e., it generates a contraction
semigroup that is the probability semigroup of the chain, by the Hille--Yosida
theorem, $(\lambda-\mathfrak{L})$ is invertible and $(\lambda-\mathfrak{L})^{-1}$
is bounded. Therefore, $(\lambda-\mathfrak{L})^{-1}$ is continuous
so that
\[
\lim_{\epsilon\to0}F_{\epsilon}(\mathfrak{m}_{n})={\bf f}(n)\qquad\text{for all}\quad n\in\mathbb{Z}.
\]
Along with Lemma \ref{l:flat_ext}, this implies condition $\mathfrak{R}$
as desired.
\end{proof}

\subsection{\label{sec4.2}Compactness condition}

For all $n\in\mathbb{Z}$ and $T,\eta>0$, there exists $\mathfrak{k}=\mathfrak{k}(n,T,\eta)\in\mathbb{N}$
such that
\begin{equation}
{\bf Q}_{n}\left[H_{\{-\mathfrak{k},\mathfrak{k}\}}<T\right]<\eta.\label{eq:escape_K}
\end{equation}
Let
\[
\mathcal{K}=\mathcal{K}(n,T,\eta):=\left[\mathfrak{m}_{-\mathfrak{k}}-2r_{0},\mathfrak{m}_{\mathfrak{k}}+2r_{0}\right].
\]
Then, $\mathcal{K}$ satisfies \eqref{eq:K1-1-Lan} since
\[
\left\{ k\in\mathbb{Z}:\mathcal{E}^{k}\cap\mathcal{K}\ne\emptyset\right\} =[-\mathfrak{k},\mathfrak{k}]\cap\mathbb{Z}.
\]
In addition, the hitting time of $\mathcal{E}^{-\mathfrak{k}}\cup\mathcal{E}^{\mathfrak{k}}$
of the diffusion process $\{X_{\epsilon}(t)\}_{t\ge0}$ can be controlled
with the limiting Markov chain $\{{\bf y}(t)\}_{t\ge0}$ as follows.
\begin{lem}
\label{l:K}Fix $n\in\mathbb{Z}$ and $T,\eta>0$. Let $\mathfrak{k}=\mathfrak{k}(n,T,\eta)\in\mathbb{N}$
satisfy \eqref{eq:escape_K}. Then,
\[
\limsup_{\epsilon\to0}\sup_{x\in\mathcal{E}^{n}}\mathbb{Q}^{\epsilon}_{x}\left[\mathcal{H}_{\mathcal{E}^{-\mathfrak{k}}\cup\mathcal{E}^{\mathfrak{k}}}<T\right]<\eta.
\]
\end{lem}

We prove Proposition \ref{p:K} assuming the above lemma.
\begin{proof}[Proof of Proposition \ref{p:K}]
 Fix $n\in\mathbb{Z}$ and $T,\eta>0$. Let $\mathfrak{k}$ and $\mathcal{K}$
be defined as above. Note that starting from $\mathcal{E}^{n}$, the
diffusion process visits $\mathcal{E}^{-\mathfrak{k}}\cup\mathcal{E}^{\mathfrak{k}}$
before exiting from $\mathcal{K}$. Therefore, by Lemma \ref{l:K},
\[
\limsup_{\epsilon\to0}\sup_{x\in\mathcal{E}^{n}}\mathbb{Q}^{\epsilon}_{x}\left[\mathcal{H}_{\mathcal{K}^{c}}\le T\right]\le\limsup_{\epsilon\to0}\sup_{x\in\mathcal{E}^{n}}\mathbb{Q}^{\epsilon}_{x}\left[\mathcal{H}_{\mathcal{E}^{-\mathfrak{k}}\cup\mathcal{E}^{\mathfrak{k}}}\le T\right]<\eta,
\]
which completes the proof.
\end{proof}

In the remainder of the section, we prove Lemma \ref{l:K}.

\subsubsection{Restriction on $\mathcal{K}$}

Define a new function $\tilde{b}:\mathbb{R}\to\mathbb{R}$ as
\[
\tilde{b}(x):=\begin{cases}
b(x) & \text{if}\quad x\in\mathcal{K},\\
b(\mathfrak{m}_{\mathfrak{k}}+2r_{0})+b'(\mathfrak{m}_{\mathfrak{k}}+2r_{0})(x-\mathfrak{m}_{\mathfrak{k}}-2r_{0}) & \text{if}\quad x>\mathfrak{m}_{\mathfrak{k}}+2r_{0},\\
b(\mathfrak{m}_{-\mathfrak{k}}-2r_{0})+b'(\mathfrak{m}_{-\mathfrak{k}}-2r_{0})(x-\mathfrak{m}_{-\mathfrak{k}}+2r_{0}) & \text{if}\quad x\le\mathfrak{m}_{-\mathfrak{k}}-2r_{0},
\end{cases}
\]
and consider a new SDE
\[
{\rm d}\tilde{\bm{x}}_{\epsilon}(t)=\tilde{b}(\tilde{\bm{x}}_{\epsilon}(t))\,{\rm d}t+\sqrt{2\epsilon}\,{\rm d}\bm{w}_{t}.
\]
Then, the process $\{\tilde{\bm{x}}_{\epsilon}(t)\}_{t\ge0}$ is a
diffusion process with finite stable states $\{\mathfrak{m}_{n}:n\in\tilde{S}\}$
where $\tilde{S}:=[-\mathfrak{k},\mathfrak{k}]\cap\mathbb{Z}$. The
infinitesimal generator, denoted by $\tilde{\mathscr{L}}_{\epsilon}:D(\tilde{\mathscr{L}}_{\epsilon})\subset C_{0}(\mathbb{R})\to C_{0}(\mathbb{R})$,
acts on $C^{2}(\mathbb{R})\cap C_{0}(\mathbb{R})$ as (cf. \eqref{eq:gen})
\begin{equation}
\tilde{\mathscr{L}}_{\epsilon}u=\tilde{b}u'+\epsilon u''.\label{eq:gen-1}
\end{equation}

Let $\{\tilde{{\bf y}}(t)\}_{t\ge0}$ be the $\tilde{S}$-Markov chain
obtained by $\{{\bf y}(t)\}_{t\ge0}$ reflected at $\pm\mathfrak{k}$,
i.e., whose jump rates $r_{\tilde{{\bf y}}}:\tilde{S}\times\tilde{S}\to[0,\infty)$
are given by
\[
r_{\tilde{{\bf y}}}(n,k):=r_{{\bf y}}(n,k)\qquad\text{if}\quad n,k\in\tilde{S}.
\]
Denote by $\tilde{\mathfrak{L}}$ the infinitesimal generator of $\{\tilde{{\bf y}}(t)\}_{t\ge0}$.
Then, the same proof of Proposition \ref{p:R} gives the following
result.
\begin{prop}
\label{p:R-ref}For every $\lambda>0$, $\tilde{{\bf g}}:\tilde{S}\to\mathbb{R}$,
and every lift $\tilde{G}\in C_{0}(\mathbb{R})$ such that
\[
\tilde{G}|_{\widehat{\mathcal{E}}^{n}}=\tilde{{\bf g}}(n);\qquad n\in\tilde{S},
\]
the unique solution $\tilde{F}_{\epsilon}\in C_{0}(\mathbb{R})$ of
the resolvent equation
\[
(\lambda-\theta_{\epsilon}\tilde{\mathscr{L}}_{\epsilon})\tilde{F}_{\epsilon}=\tilde{G},
\]
satisfies
\[
\lim_{\epsilon\to0}\sup_{x\in\mathcal{E}^{n}}\left|\tilde{F}_{\epsilon}(x)-{\bf f}(n)\right|=0,
\]
where $\tilde{{\bf f}}:\tilde{S}\to\mathbb{R}$ solves $(\lambda-\tilde{\mathfrak{L}})\tilde{{\bf f}}=\tilde{{\bf g}}$.
\end{prop}

Let $\tilde{\mathcal{E}}:=\bigcup_{n\in\tilde{S}}\mathcal{E}^{n}$.
Define the accelerated process $\tilde{X}_{\epsilon}(t):=\tilde{\bm{x}}_{\epsilon}(\theta_{\epsilon}t)$,
the trace process $\tilde{X}^{{\rm tr}}_{\epsilon}$ of $\tilde{X}_{\epsilon}$
on $\tilde{\mathcal{E}}$, and the order process $\tilde{Y}_{\epsilon}$
in $\tilde{S}$ as in \eqref{eq:YN}. Denote by $\tilde{\mathbb{Q}}^{\epsilon}_{x}$,
$\tilde{{\bf Q}}^{\epsilon}_{x}$, and $\tilde{\mathbf{Q}}_{n}$ the
law of $\{\tilde{X}_{\epsilon}(t)\}_{t\ge0}$, $\{\tilde{Y}_{\epsilon}(t)\}_{t\ge0}$
starting from $x\in\mathbb{R}$ and the law $\{\tilde{{\bf y}}(t)\}_{t\ge0}$
of starting from $n\in\tilde{S}$, respectively.

Condition $\mathfrak{K}$ for $\{\tilde{X}_{\epsilon}(t)\}_{t\ge0}$
is obvious since $\{-\mathfrak{k},\dots,\mathfrak{k}\}$ is finite.
Therefore, by Theorem \ref{t:main2} and Proposition \ref{p:R-ref},
we have:
\begin{thm}
\label{t:Lan-meta}The diffusion process $\{\tilde{X}_{\epsilon}(t)\}_{t\ge0}$
is metastable in the following sense:
\begin{enumerate}
\item For any $n\in\tilde{S}$ and sequence $(x_{\epsilon})_{\epsilon>0}$
in $\mathcal{E}^{n}$, the laws $\tilde{{\bf Q}}^{\epsilon}_{x_{\epsilon}}$
converge weakly to the limit law $\tilde{{\bf Q}}_{n}$ as $\epsilon\to0$.
\item The excursions outside $\tilde{\mathcal{E}}$ is negligible, i.e.
for each $n\in\tilde{S}$, sequence $(x_{\epsilon})_{\epsilon>0}$
in $\mathcal{E}^{n}$, and $T>0$,
\[
\lim_{\epsilon\to0}{\rm E}^{\tilde{\mathbb{Q}}^{\epsilon}_{x_{\epsilon}}}\left[\int^{T}_{0}{\bf 1}\left\{ \tilde{X}_{\epsilon}(t)\notin\tilde{\mathcal{E}}\right\} {\rm d}t\right]=0.
\]
\end{enumerate}
\end{thm}

\subsubsection{Proof of Lemma \ref{l:K}}

We are now ready to prove Lemma \ref{l:K}.
\begin{proof}[Proof of Lemma \ref{l:K}]
 Let $(x_{\epsilon})_{\epsilon>0}$ be a sequence in $\mathcal{E}^{n}$.
By coupling $\{X_{\epsilon}(t)\}_{t\ge0}$ and $\{\tilde{X}_{\epsilon}(t)\}_{t\ge0}$
on $[0,\mathcal{H}_{\mathcal{E}^{-\mathfrak{k}}\cup\mathcal{E}^{\mathfrak{k}}})$
in the canonical way, i.e., by identifying the trajectories before
hitting the boundary $\pm\mathfrak{k}$,
\begin{equation}
\mathbb{Q}^{\epsilon}_{x_{\epsilon}}\left[\mathcal{H}_{\mathcal{E}^{-\mathfrak{k}}\cup\mathcal{E}^{\mathfrak{k}}}<T\right]=\tilde{\mathbb{Q}}^{\epsilon}_{x_{\epsilon}}\left[\mathcal{H}_{\mathcal{E}^{-\mathfrak{k}}\cup\mathcal{E}^{\mathfrak{k}}}<T\right].\label{eq:pf_lem_K2-1}
\end{equation}
Since $\tilde{{\bf Q}}^{\epsilon}_{x_{\epsilon}}$ is the law of the
order process $\tilde{Y}_{\epsilon}$ , 
\begin{equation}
\tilde{\mathbb{Q}}^{\epsilon}_{x_{\epsilon}}\left[\mathcal{H}_{\mathcal{E}^{-\mathfrak{k}}\cup\mathcal{E}^{\mathfrak{k}}}<T\right]\le\tilde{{\bf Q}}^{\epsilon}_{x_{\epsilon}}\left[H_{\{\pm\mathfrak{k}\}}\le T\right].\label{eq:pf_lem_K2-2}
\end{equation}
 By Theorem \ref{t:Lan-meta}, the law $\tilde{{\bf Q}}^{\epsilon}_{x_{\epsilon}}$
converge weakly to the limit law $\tilde{{\bf Q}}_{n}$ as $\epsilon\to0$.
By \eqref{e:hit_conv},
\begin{equation}
\limsup_{\epsilon\to0}\tilde{{\bf Q}}^{\epsilon}_{x_{\epsilon}}\left[H_{\{\pm\mathfrak{k}\}}\le T\right]\le\tilde{{\bf Q}}_{n}\left[H_{\{\pm\mathfrak{k}\}}\le T\right].\label{eq:pf_lem_K2-3}
\end{equation}
By coupling $\{{\bf y}(t)\}_{t\ge0}$ and $\{\tilde{{\bf y}}(t)\}_{t\ge0}$
on $[0,H_{\pm\mathfrak{k}})$ in the canonical way and applying \eqref{eq:escape_K},
\begin{equation}
\tilde{{\bf Q}}_{n}\left[H_{\pm\mathfrak{k}}\le T\right]={\bf Q}_{n}\left[H_{\{\pm\mathfrak{k}\}}\le T\right]<\eta.\label{eq:pf_lem_K2-4}
\end{equation}
By \eqref{eq:pf_lem_K2-1}--\eqref{eq:pf_lem_K2-4}, for any sequence
$(x_{\epsilon})_{\epsilon>0}$ in $\mathcal{E}^{n}$, 
\[
\limsup_{\epsilon\to0}\mathbb{Q}^{\epsilon}_{x_{\epsilon}}\left[\mathcal{H}_{\mathcal{E}^{-\mathfrak{k}}\cup\mathcal{E}^{\mathfrak{k}}}<T\right]<\eta.
\]
\end{proof}

\subsection{Mixing condition}
\begin{proof}[Proof of Proposition \ref{p:M}]
 Fix $n\in\mathbb{Z}$ and $(x_{\epsilon})_{\epsilon}$ in $\mathcal{E}^{n}$.
Bound the probability $\mathbb{Q}^{\epsilon}_{x_{\epsilon}}[\mathcal{H}_{\mathfrak{m}_{n}}>\mathcal{H}_{\breve{\mathcal{E}}^{n}}]$
as
\[
\begin{aligned}\mathbb{Q}^{\epsilon}_{x_{\epsilon}}\left[\mathcal{H}_{\mathfrak{m}_{n}}>\mathcal{H}_{\breve{\mathcal{E}}^{n}}\right] & =\mathbb{Q}^{\epsilon}_{x_{\epsilon}}\left[\mathcal{H}_{\mathfrak{m}_{n}}>\mathcal{H}_{\breve{\mathcal{E}}^{n}},\mathcal{H}_{\breve{\mathcal{E}}^{n}}<t\right]+\mathbb{Q}^{\epsilon}_{x_{\epsilon}}\left[\mathcal{H}_{\mathfrak{m}_{n}}>\mathcal{H}_{\breve{\mathcal{E}}^{n}},\mathcal{H}_{\breve{\mathcal{E}}^{n}}\ge t\right]\\
 & \le\mathbb{Q}^{\epsilon}_{x_{\epsilon}}\left[\mathcal{H}_{\breve{\mathcal{E}}^{n}}<t\right]+\mathbb{Q}^{\epsilon}_{x_{\epsilon}}\left[\mathcal{H}_{\mathfrak{m}_{n}}>t\right]\\
 & \le\mathbb{Q}^{\epsilon}_{x_{\epsilon}}\left[\mathcal{H}_{\breve{\mathcal{E}}^{n}}<t\right]+\mathbb{Q}^{\epsilon}_{x_{\epsilon}}\left[\mathcal{H}_{\mathfrak{m}_{n}}>\frac{1}{\epsilon^{2}\theta_{\epsilon}}\right],
\end{aligned}
\]
where $t>0$ is an arbitrarily chosen positive number, and the last
inequality holds since $\lim_{\epsilon\to0}\epsilon^{2}\theta_{\epsilon}=\infty$.
By Proposition \ref{p:hit}, the last probability vanishes so that
\[
\limsup_{\epsilon\to0}\mathbb{Q}^{\epsilon}_{x_{\epsilon}}\left[\mathcal{H}_{\mathfrak{m}_{n}}>\mathcal{H}_{\breve{\mathcal{E}}^{n}}\right]\le\limsup_{\epsilon\to0}\mathbb{Q}^{\epsilon}_{x_{\epsilon}}\left[\mathcal{H}_{\breve{\mathcal{E}}^{n}}<t\right].
\]
Proposition \ref{p:R} (condition $\widehat{\mathfrak{R}}$) together
with Lemma \ref{lem:2.13} implies condition $\mathfrak{R}$. Moreover,
by Proposition \ref{p:K}, condition $\mathfrak{K}$ is satisfied.
Since conditions $\mathfrak{R}$ and $\mathfrak{K}$ yield condition
$\mathfrak{D}$ by Theorem \ref{t:main1}, condition $\mathfrak{K}_{2}$
is then satisfied by Lemma \ref{lem2.11}. Finally, by Lemma \ref{lem2.3},
condition $\mathfrak{N}$ holds true so that
\[
\lim_{t\to0}\limsup_{\epsilon\to0}\mathbb{Q}^{\epsilon}_{x_{\epsilon}}\left[\mathcal{H}_{\breve{\mathcal{E}}^{n}}<t\right]=0,
\]
which completes the proof.
\end{proof}

\begin{acknowledgement*}
The authors would like to thank CIRM (Centre International de Rencontres
Math\'ematiques, Marseille) for their warm hospitality during their
stay in September 2025, during which the part of the collaboration
was conducted. SK has been supported by the Basic Science Research
Program through the National Research Foundation of Korea funded by
the Ministry of Science and ICT (RS-2025-00518980, RS-2026-25518141),
the Yonsei University Research Fund of 2026 (2026-22-0181), and the
POSCO Science Fellowship of POSCO TJ Park Foundation. JL has been
supported by INHA UNIVERSITY Research Grant.
\end{acknowledgement*}

\appendix

\section{\label{appA}Probabilistic Representation of the Resolvent Solution}

Fix a locally compact, normal, separable metric space $\Omega$, a
Markov process $\{X(t)\}_{t\ge0}$ therein, and its corresponding
infinitesimal generator $\mathscr{L}$ acting on $C_{0}(\Omega)$,
as in Section \ref{sec1.3}. For any given $\lambda>0$ and $G\in C_{0}(\Omega)$,
we prove that the unique solution $F\in C_{0}(\Omega)$ to the following
resolvent equation,
\begin{equation}
(\lambda-\mathscr{L})F=G\quad\text{in}\quad\Omega,\label{e: res-1}
\end{equation}
admits a probabilistic representation given as
\begin{equation}
F(\bm{x})={\rm E}^{\mathbb{Q}_{\bm{x}}}\left[\int^{\infty}_{0}e^{-\lambda t}\,G(X(t))\,{\rm d}t\right],\label{eq:pr}
\end{equation}
where $\mathbb{Q}_{\bm{x}}$ denotes the law of $X(t)$ starting from
$\bm{x}$. To see this, we apply the time-dependent martingale problem
with respect to $(x,t)\mapsto e^{-\lambda t}\,F(x)$ to obtain that
\[
e^{-\lambda t}\,F(X(t))-F(X(0))+\int^{t}_{0}e^{-\lambda s}\,(\lambda-\mathscr{L})F(X(s))\,{\rm d}s
\]
is a $\mathbb{Q}_{\bm{x}}$-martingale. Taking expectation and substituting
\eqref{e: res-1},
\[
F(\bm{x})={\rm E}^{\mathbb{Q}_{\bm{x}}}\,[e^{-\lambda t}\,F(X(t))]+{\rm E}^{\mathbb{Q}_{\bm{x}}}\left[\int^{t}_{0}e^{-\lambda s}\,G(X(s))\,{\rm d}s\right].
\]
The first term in the right-hand side vanishes as $t\to\infty$ since
$F\in C_{0}(\Omega)$ is a bounded function. Thus, sending $t\to\infty$,
we obtain
\[
F(\bm{x})={\rm E}^{\mathbb{Q}_{\bm{x}}}\left[\int^{\infty}_{0}e^{-\lambda s}\,G(X(s))\,{\rm d}s\right],
\]
which is exactly \eqref{eq:pr}. In addition, from \eqref{eq:pr}
we have
\begin{equation}
\|F\|_{\infty}\le\|G\|_{\infty}\int^{\infty}_{0}e^{-\lambda t}\,{\rm d}t=\frac{\|G\|_{\infty}}{\lambda}.\label{eq:FG}
\end{equation}

\section{Convergence in Law}

Fix a finite set $S$. Let $D([0,\infty);S)$ be the space of functions
$\omega:[0,\infty)\to S$ that are right-continuous and have left-hand
limits endowed with the usual $J_{1}$ (Skorokhod) topology. Let $d$
be the metric on $D([0,\infty);S)$ which gives the topology.
\begin{lem}
For all $j\in S$ and $T>0$, the set $\{H_{j}\le T\}$ is closed
in $D([0,\infty);S)$.
\end{lem}

\begin{proof}
Let $(w_{n})_{n\ge1}$ be a sequence in the set $\{H_{j}\le T\}$
and $w\in D([0,\infty);S)$. Suppose that $w_{n}\to w$ in the $J_{1}$
topology. Let 
\[
t_{n}:=\inf\{t\ge0:w_{n}(t)=j\},\qquad t_{0}:=\inf\{t\ge0:w(t)=j\}.
\]
It suffices to prove that $w\in\{H_{j}\le T\}$, i.e., $t_{0}\le T$.
Since $w_{n}$ is right-continuous and $S$ is discrete, $w_{n}(t_{n})=j$
and $t_{n}\le T$. Let $K>T$ be such that $w$ is continuous at $K$.
Since $w_{n}\to w$ in $D([0,\infty);S)$,
\[
w_{n}\to w\qquad\text{in}\quad D([0,K];S).
\]
Then, there exists a sequence of increasing functions $\lambda_{n}:[0,K]\to[0,K]$
such that $\lambda_{n}(0)=0$, $\lambda_{n}(K)=K$,
\[
\|\lambda_{n}-I\|_{[0,K]}\to0,\qquad\text{and}\qquad\sup_{0\le s\le K}|w_{n}(s)-w(\lambda_{n}(s))|\to0.
\]
Here, $I$ is the identity function. Hence,
\[
w_{n}(t_{n})-w(\lambda_{n}(t_{n}))=j-w(\lambda_{n}(t_{n}))\to0,
\]
so that since $S$ is discrete, there exists $N\ge1$ such that
\[
n\ge N\quad\Rightarrow\quad w(\lambda_{n}(t_{n}))=j.
\]
Therefore, for $n\ge N$,
\[
t_{0}\le\lambda_{n}(t_{n}).
\]
Since
\[
t_{0}\le\lambda_{n}(t_{n})\le t_{n}+\|\lambda_{n}-I\|_{[0,K]}\le T+\|\lambda_{n}-I\|_{[0,K]},
\]
and $\|\lambda_{n}-I\|_{[0,K]}\to0$as $n\to\infty$, we have
\[
t_{0}\le T,
\]
which completes the proof.
\end{proof}

Let ${\bf Q}$ and ${\bf Q}_{N}$ be probability measures on $D([0,\infty);S)$
such that ${\bf Q}_{N}\to{\bf Q}$ as $N\to\infty$ in the weak topology.
Then, by the above lemma and the Portmanteau theorem,
\begin{equation}
\limsup_{N\to\infty}{\bf Q}_{N}[\mathcal{H}_{j}\le T]\le{\bf Q}[\mathcal{H}_{j}\le T].\label{e:hit_conv}
\end{equation}


\begin{thebibliography}{10}
\bibitem{AGL} I. Armend\'ariz, S. Grosskinsky, M. Loulakis: Metastability
in a condensing zero-range process in the thermodynamic limit. Probab.
Theory Related Fields \textbf{169}, 105--175 (2017)

\bibitem{BL10} J. Beltr\'an, C. Landim: Tunneling and metastability
of continuous time Markov chains. J. Stat. Phys. \textbf{140}, 1065--1114
(2010)

\bibitem{BL12} J. Beltr\'an, C. Landim: Metastability of reversible
condensed zero range processes on a finite set. Probab. Theory Related
Fields \textbf{152}, 781--807 (2012)

\bibitem{BAC96} G. Ben Arous, R. Cerf: Metastability in the three
dimensional Ising model on a torus at very low temperatures. Electron.
J. Probab. \textbf{1}, 1--55 (1997)

\bibitem{Bil99} P. Billingsley: \textit{Convergence of Probability
Measures, Second Edition}. Wiley Series in Probability and Statistics
(1999)

\bibitem{BR16}F. Bouchet, J. Reygner: Generalisation of the Eyring--Kramers
transition rate formula to irreversible diffusion processes. Ann.
Inst. Henri Poincar\'e Probab. Stat. \textbf{17}, 3499--3532 (2016)

\bibitem{BEGK} A. Bovier, M. Eckhoff, V. Gayrard, M. Klein: Metastability
in reversible diffusion process I. Sharp asymptotics for capacities
and exit times. J. Eur. Math. Soc. \textbf{6}, 399--424 (2004)

\bibitem{BdH15} A. Bovier, F. den Hollander: \emph{Metastabillity:
A Potential-Theoretic Approach}. Grundlehren der mathematischen Wissenschaften.
Springer, Cham (2015)

\bibitem{BM02} A. Bovier, F. Manzo: Metastability in Glauber Dynamics
in the Low-Temperature Limit: Beyond Exponential Asymptotics. J. Stat.
Phys. \textbf{107}, 757--779 (2002)

\bibitem{DGLLPN}G. Di Ges\'u, T. Leli\`evre, D. Le Peutrec, B.
Nectoux: Jump Markov models and transition state theory: the quasi-stationary
distribution approach. Faraday Discuss. \textbf{196}, 469--495 (2016)

\bibitem{EK86} S. N. Ethier, T. G. Kurtz: \emph{Markov Processes:
Characterization and Convergence}. Wiley Series in Probability and
Statistics (1986)

\bibitem{FM08} L. R. G. Fontes, P. Mathieu: K-processes, scaling
limit and aging for the trap model in the complete graph. Ann. Probab.
\textbf{36}, 1322--1358 (2008)

\bibitem{GKR07} C. Giardin\`a, J. Kurchan, F. Redig: Duality and
exact correlations for a model of heat conduction. J. Math. Phys.
\textbf{48}, 033301 (2007)

\bibitem{GRV13} S. Grosskinsky, F. Redig, K. Vafayi: Dynamics of
condensation in the symmetric inclusion process. Electron. J. Probab.
\textbf{18}, no 66., 1--23 (2013)

\bibitem{JLT} M. Jara, C. Landim, A. Teixeira: Universality of trap
models in the ergodic time scale. Ann. Probab. \textbf{42}, 2497--2557
(2014)

\bibitem{KL26} S. Kim, J. Lee: Sharp mixing time asymptotics of Glauber
dynamics for the Curie--Weiss--Potts model at low temperatures.
arXiv:2602.19545 (2026)

\bibitem{KS21} S. Kim, I. Seo: Condensation and metastable behavior
of non-reversible inclusion processes. Comm. Math. Phys. \textbf{382},
1343--1401 (2021)

\bibitem{KS25} S. Kim, I. Seo: Approximation method to metastability:
an application to nonreversible, two-dimensional Ising and Potts models
without external fields. Ann. Probab. \textbf{53}, 597--667 (2025)

\bibitem{LLS24} C. Landim, J. Lee, I. Seo: Metastability and Time
Scales for Parabolic Equations with Drift 1: The First Time Scale.\textit{
}Arch. Rational Mech. Anal. \textbf{248}, 78 (2024)

\bibitem{LLS25} C. Landim, J. Lee, I. Seo: Metastability and Time
Scales for Parabolic Equations with Drift 2: The General Time Scale.
arXiv:2402.07695 (2024)

\bibitem{LMS23} C. Landim, D. Marcondes, I. Seo: Metastable behavior
of weakly mixing Markov chains: The case of reversible, critical zero-range
processes. Ann. Probab. \textbf{51}, 157--227 (2023)

\bibitem{LMS25} C. Landim, D. Marcondes, I. Seo: A resolvent approach
to metastability. J. Eur. Math. Soc. (JEMS) \textbf{27}(4), 1563--1618
(2025)

\bibitem{LM26} C. Landim, C. Maura: Full $\Gamma$\textminus expansion
for the level-two large deviation rate functionals of non-reversible
one-dimensional diffusions with periodic boundary conditions. arXiv:2606.17859
(2026)

\bibitem{LS16} C. Landim, I. Seo: Metastability of non-reversible,
mean-field Potts model with three spins. J. Stat. Phys. \textbf{165},
693--726 (2016)

\bibitem{LS19} C. Landim, I. Seo: Metastability of one-dimensional,
non-reversible diffusions with periodic boundary conditions. Ann.
Inst. Henri Poincar\'e Probab. Stat. \textbf{55}, 1850--1889 (2019)

\bibitem{LS22a} J. Lee, I. Seo: Non-reversible metastable diffusions
with Gibbs invariant measure I: Eyring--Kramers formula. Probab.
Theory Related Fields \textbf{182}, 849--903 (2022)

\bibitem{LPM} D. Le Peutrec, L. Michel: Sharp spectral asymptotics
for nonreversible metastable diffusion processes. Probab. Math. Phys.
\textbf{1}, 3--53 (2019)

\bibitem{Lig10} T. M. Liggett: \textit{Continuous Time Markov Processes:
An Introduction}. American Mathematical Society (2010)

\bibitem{Mic95} L. Miclo: Une \'etude des algorithmes de recuit
simul\'e sous-admissibles. Ann. Fac. Sci. Toulouse Math. \textbf{4}(4),
819--877 (1995)

\bibitem{NZ19} F. R. Nardi, A. Zocca: Tunneling behavior of Ising
and Potts models in the low-temperature regime. Stochastic Process.
Appl. \textbf{129}, 4556--4575 (2019)

\bibitem{NS91} E. J. Neves, R. H. Schonmann: Critical Droplets and
Metastability for a Glauber Dynamics at Very Low Temperatures. Comm.
Math. Phys. \textbf{137}, 209--230 (1991)

\bibitem{SV69} D. W. Stroock, S. R. S. Varadhan: Diffusion processes
with continuous coefficients, I. Comm. Pure Appl. Math. \textbf{22}(3),
345--400 (1969)

\end{thebibliography}
\end{document}